\documentclass[12pt]{amsart}
\usepackage{amssymb}
\usepackage{amsmath}
\usepackage{amsthm}
\usepackage{times}
\usepackage{graphicx}
\usepackage{changepage}
\usepackage{caption}
\usepackage{geometry}
\usepackage[curve,all]{xy}
 \xyoption{curve}
 \xyoption{line}
\xyoption{arc}
\xyoption{color}
\usepackage{array}
\usepackage{enumitem}

\usepackage{color}
\usepackage{amsthm}
\newtheorem{theorem}{Theorem}[section]
\newtheorem{lemma}[theorem]{Lemma}
\newtheorem{corollary}[theorem]{Corollary}

\newtheorem{prop}[theorem]{Proposition}

\usepackage{subfig}
 
\theoremstyle{definition}
\newtheorem{definition}[theorem]{Definition}

\theoremstyle{remark}
\newtheorem{remark}[theorem]{Remark}

\numberwithin{equation}{section}

\newcommand{\calO}{\mathcal{O}}

\def\Aut{{\text{Aut}}}

\def\Pic{{\text{Pic}}}
\def\Ker{{\text{Ker}}}
\def\Im{{\text{Im}}}
\def\O{{\text{O}}}

\def\Num{{\rm{Num}}}
\def\NS{{\rm{NS}}}

\def\deg{{\text{deg}}}

\def\Num{{\text{Num}}}

\def\I{{\text{I}}}
\def\II{{\text{II}}}
\def\III{{\text{III}}}
\def\IV{{\text{IV}}}

\def\VI{{\text{VI}}}
\def\VII{{\text{VII}}}
\def\VIII{{\text{VIII}}}
\def\Pic{{\text{Pic}}}
\begin{document}
\title [Enriques surfaces] {Classification of Enriques surfaces with finite automorphism group in characteristic 2}

\author{Toshiyuki Katsura}
\address{Graduate School of Mathematical Sciences, The University of Tokyo,
Meguro-ku, Tokyo 153-8914, Japan}
\email{tkatsura@ms.u-tokyo.ac.jp}

\author{Shigeyuki Kond\=o}
\address{Graduate School of Mathematics, Nagoya University, Nagoya,
464-8602, Japan}
\email{kondo@math.nagoya-u.ac.jp}

\author{Gebhard Martin}
\address{Mathematisches Institut,
Universit\"at Bonn,
Endenicher Allee 60,
53115 Bonn,
Germany}
\email{gmartin@math.uni-bonn.de}


\thanks{Research of the first author is partially supported by
Grant-in-Aid for Scientific Research (B) No.15H03614, the second author by (S) No.15H05738 and
by the Simons Visiting Professorship, and the third author by the DFG Sachbeihilfe LI 1906/3 - 1 "Automorphismen von Enriques Fl\"achen".}

\begin{abstract}
We classify supersingular and classical Enriques surfaces with finite automorphism group in characteristic 2 into 8 types according to their dual graphs of all $(-2)$-curves (nonsigular rational curves).  We give 
examples of these Enriques surfaces together with their canonical coverings.  It follows that the classification of all Enriques surfaces with finite automorphism group in any characteristics has been finished. 
\end{abstract}

\maketitle

{\bf Contents}
\begin{itemize}
\item[\S \ref{sec1}]
Introduction \hfill \pageref{sec1}

\item[\S \ref{sec2}] 
Preliminaries \hfill \pageref{sec2}

\item[\S \ref{conductrix}]
Conductrix \hfill \pageref{conductrix}

\item[\S \ref{possibledualgraph}]  
Possible dual graphs  \hfill \pageref{possibledualgraph}

\item[\S \ref{VectFields}]
Construction of vector fields \hfill \pageref{VectFields}

\item[\S \ref{Equation}]
Equations of Enriques surfaces and their automorphisms \hfill \pageref{Equation}

\item[\S \ref{sec3}] 
Enriques surfaces of type $\tilde{E_6}+\tilde{A_2}$ \hfill \pageref{sec3}

\item[\S \ref{secVII}] 
Enriques surfaces of type ${\rm VII}$ \hfill \pageref{secVII}

\item[\S \ref{sec8}] 
Enriques surfaces of type ${\rm VIII}$ \hfill \pageref{sec8}

\item[\S \ref{sec4}] 
Enriques surfaces of type $\tilde{E_8}$ \hfill \pageref{sec4}

\item[\S \ref{sec6}] 
Enriques surfaces of type $\tilde{E_7}+\tilde{A_1}^{(1)}$, $\tilde{E_7}+\tilde{A_1}^{(2)}$ \hfill \pageref{sec6}

\item[\S \ref{sec7}] 
Enriques surfaces of type $\tilde{D_8}$ \hfill \pageref{sec7}

\item[\S \ref{sec5}] 
Enriques surfaces of type $\tilde{D_4}+\tilde{D_4}$ \hfill \pageref{sec5}

\item[\S \ref{appendix}] 
Appendix \hfill \pageref{appendix}

\end{itemize}

\section{Introduction}\label{sec1}

We work over an algebraically closed field $k$ of characteristic 2.
The main purpose of this paper 
is to give a classification of supersingular and classical Enriques surfaces with finite automorphism group in characteristic 2. More precisely, we classify the possible dual graphs of $(-2)$-curves that can occur on these surfaces and give examples realizing each of these graphs.
Recall that, over the complex numbers, 
a generic Enriques surface has an infinite group of automorphisms (Barth and Peters \cite{BP}).  On the other hand, Fano \cite{F} and Dolgachev \cite{D1} found two examples of Enriques surfaces with finite automorphism group. Then
Nikulin \cite{N} proposed a classification of Enriques surfaces with finite automorphism group in terms of the periods, and the second author \cite{Ko} classified and gave constructions of all such Enriques surfaces, geometrically.  There are seven types ${\I, \II,\ldots, \VII}$ of such Enriques surfaces, distinguished by their dual graphs of $(-2)$-curves.  
The Enriques surfaces of type ${\I}$ and 
${\II}$ form irreducible 1-dimensional families, and each of the remaining types
consists of a unique Enriques surface.  
The first two types contain exactly twelve $(-2)$-curves (i.e. non-singular rational curves), while the remaining five types contain exactly twenty $(-2)$-curves.
The Enriques surfaces of type ${\I}$ (resp. of type ${\VII}$) are the examples given by Dolgachev (resp. by Fano).   
We call the dual graphs of all $(-2)$-curves on the Enriques surface of type $K$ the dual graph of type $K$ ($K = {\I, \II,..., \VII}$).  
We remark that if an Enriques surface has the dual graph of this type, then its automorphism group is finite.
 
In positive characteristic, the question of a classification of Enriques surfaces with finite automorphism group has been raised.  In particular, the case of characteristic 2 is most interesting.  In the paper \cite{BM2}, Bombieri and Mumford classified
Enriques surfaces in characteristic 2 into three classes, namely singular, classical and 
supersingular Enriques surfaces.  
As in the case of characteristic $0$ or $p>2$, an Enriques surface
$X$ in characteristic 2 has a canonical double cover
$\pi : \tilde{X} \to X$, which is a separable ${\bf Z}/2{\bf Z}$-cover,
a purely inseparable $\mu_2$- or $\alpha_2$-cover according to $X$ being singular, classical or supersingular. The surface $\tilde{X}$ might have singularities and it might even be non-normal, but it is $K3$-like in the sense that its dualizing sheaf is trivial.  
Recently, Liedtke \cite{L} showed that 
the moduli space of
Enriques surfaces with a polarization of degree $4$ has two $10$-dimensional irreducible components.  A general point of one component (resp. the other component) corresponds to a singular (resp. classical) Enriques surface.  The intersection of the two components parametrizes supersingular Enriques surfaces.
On the other hand, Ekedahl and Shepherd-Barron \cite{ES} studied certain special Enriques surfaces 
called "exceptional Enriques surfaces", whose deformation functors are badly behaved, and Salomonsson \cite{Sa} gave equations of such Enriques surfaces.  
We remark that some of them have a finite group of automorphisms.

Very recently, the first and the second author \cite{KK2} determined 
the existence or non-existence of Enriques surfaces in characteristic 2 whose dual graphs of all 
$(-2)$-curves are of type {\rm I}, {\rm II}, \ldots , or {\rm VII} as in the following Theorem:

\begin{theorem}{\rm (Katsura, Kondo \cite{KK2})}
The existence or non-existence of Enriques surfaces in characteristic $2$ whose dual graphs of all non-singular rational curves are of type {\rm I}, {\rm II}, \ldots , or {\rm VII} is as in the following Table{\rm \ref{Table1}:}

\begin{table}[!htb]
{\offinterlineskip
\halign{\strut\vrule#&\quad\hfil\rm#\hfil\quad&&
\vrule#&\quad#\hfil\quad\cr\noalign{\hrule}
& {\rm Type}&&${\rm I}$ && ${\rm II}$ && ${\rm III}$&&${\rm IV}$&&${\rm V}$ && ${\rm VI}$ &&${\rm VII}$&\cr
\noalign{\hrule}
& {\rm singular} && {$\bigcirc$} && {$\bigcirc$} && {$\times$} && {$\times$} && {$\times$} && {$\bigcirc$} && {$\times$} &\cr
\noalign{\hrule}
& {\rm classical} && {$\times$} && {$\times$} && {$\times$} && {$\times$} && {$\times$} && {$\times$} && {$\bigcirc$} &\cr
\noalign{\hrule}
& {\rm supersingular} && {$\times$} && {$\times$} && {$\times$} && {$\times$} && {$\times$} && {$\times$} && {$\bigcirc$} &\cr
\noalign{\hrule}
}}
\
\caption{}
\label{Table1}
\end{table}
\noindent
In Table {\rm \ref{Table1}} and in the following Table {\rm \ref{Table2}}, $\bigcirc$ means the existence and $\times$ means the non-existence
of an Enriques surface with the dual graph of type ${\rm I},..., {\rm VII}$.  
\end{theorem}

On the other hand, the third author \cite{Ma} gave a classification of Enriques surfaces with finite automorphism group in characteristic $p > 2$ and in the case of singular Enriques surfaces by using the method given by the second author over the complex numbers.  

\begin{theorem}{\rm (Martin \cite{Ma})}
The following Table {\rm \ref{Table2}} gives the classification of Enriques surfaces with finite automorphism group whose
canonical coverings are smooth. The moduli spaces for type {\rm I} and {\rm II} are $1$-dimensional and irreducible.  For each of the other types, there is a unique such surface.
\begin{table}[!htb]
{\offinterlineskip
\halign{\strut\vrule#&\quad\hfil\rm#\hfil\quad&&
\vrule#&\quad#\hfil\quad\cr\noalign{\hrule}
& {\rm Type}&&${\rm I}$ && ${\rm II}$ && ${\rm III}$&&${\rm IV}$&&${\rm V}$ && ${\rm VI}$ &&${\rm VII}$&\cr
\noalign{\hrule}
& {\rm $p=0$ or $p> 5$} && {$\bigcirc$} && {$\bigcirc$} && {$\bigcirc$} && {$\bigcirc$} && {$\bigcirc$} && {$\bigcirc$} && {$\bigcirc$} &\cr
\noalign{\hrule}
& {\rm $p=5$} && {$\bigcirc$} && {$\bigcirc$} && {$\bigcirc$} && {$\bigcirc$} && {$\bigcirc$} && {$\times$} && {$\times$} &\cr
\noalign{\hrule}
& {\rm $p=3$} && {$\bigcirc$} && {$\bigcirc$} && {$\bigcirc$} && {$\bigcirc$} && {$\times$} && {$\times$} && {$\bigcirc$} &\cr
\noalign{\hrule}
& {\rm $p=2$, singular} && {$\bigcirc$} && {$\bigcirc$} && {$\times$} && {$\times$} && {$\times$} && {$\bigcirc$} && {$\times$} &\cr
\noalign{\hrule}
}}
\
\caption{}
\label{Table2}
\end{table}

\end{theorem}

Thus, the classification problem remains only for the classical and supersingular cases.
Now, we state the main results of this paper.  
In the following Theorems \ref{mainSupersingular} \rm{(B)} and \ref{mainClassical} \rm{(B)}, we give examples of Enriques surfaces with finite automorphisms.  
Some of them are constructed as families of such Enriques surfaces.  In the Tables \ref{mainSupersingularTable} and \ref{mainClassicalTable} "dim" denotes the
dimensions of these families of examples.
The families of type
$\tilde{E}_7+\tilde{A}_1^{(1)}$ supersingular, of type $\tilde{E}_6+\tilde{A}_2$ classical, of type $\VII$ classical, and of type $\VIII$ are non-isotrivial.
The family of type $\tilde{E}_7+\tilde{A}_1^{(1)}$ classical surfaces and the family of type $\tilde{D}_4 + \tilde{D}_4$ contain an at least 1-dimensional non-isotrivial family.
The authors do not know the existence of other examples, that is, the problem of determining the moduli space of such Enriques surfaces is still open.
We denote by $\Aut(X)$, $\Aut_{ct}(X)$ or $\Aut_{nt}(X)$ the automorphism group of $X$, 
the cohomologically trivial automorphism group or the numerically trivial automorphism group
(see Definition \ref{cohnumtrivial}), respectively.
Let $\mathfrak{S}_n$ be the symmetric group of degree $n$ 
and $Q_8$ the quaternion group of order $8$.  
%

\begin{theorem}\label{mainSupersingular}
Let $X$ be a supersingular Enriques surface in characteristic $2$.
\begin{enumerate}[label=${\rm(\Alph*)}$]
\item $X$ has a finite group of automorphisms if and only if the dual graph of all $(-2)$-curves on $X$ is one of the graphs in Table {\rm\ref{mainSupersingularTable} (A)}.
\item All cases exist. More precisely, we construct families of these surfaces whose automorphism groups and dimensions are given in Table {\rm\ref{mainSupersingularTable} (B)}.
\end{enumerate}
\begin{table}[!htb]
\centering
\subfloat[Classification]{
\begin{tabular}{|>{\centering\arraybackslash}m{1.7cm}|>{\centering\arraybackslash}m{6cm}|}
\hline

\text{\rm Type} & \text{\rm Dual Graph of $(-2)$-curves} \\

\hline
$\tilde{E_8}$ &
\resizebox{!}{1cm}{
\xy
(-10,25)*{};
@={(-10,10),(0,10),(10,10),(20,10),(30,10),(40,10),(50,10),(60,10),(70,10),(10,20)}@@{*{\bullet}};
(-10,10)*{};(70,10)*{}**\dir{-};
(10,10)*{};(10,20)*{}**\dir{-};
\endxy
}
\\ \hline
$\tilde{E_7}+\tilde{A_1}^{(1)}$ &
\vspace{4mm}
\resizebox{!}{1cm}{
\xy
(0,25)*{};
@={(80,10),(90,10),(0,10),(10,10),(20,10),(30,10),(40,10),(50,10),(60,10),(70,10),(30,20)}@@{*{\bullet}};
(0,10)*{};(80,10)*{}**\dir{-};
(90,10)*{};(80,10)*{}**\dir{=};
(30,10)*{};(30,20)*{}**\dir{-};
(70,10)*{};(90,10)*{}**\crv{(80,20)};
\endxy
} 
\vspace{1mm}
\\ \hline
$\tilde{E}_6+\tilde{A_2}$  &
\vspace{1mm}
\includegraphics[width=32mm]{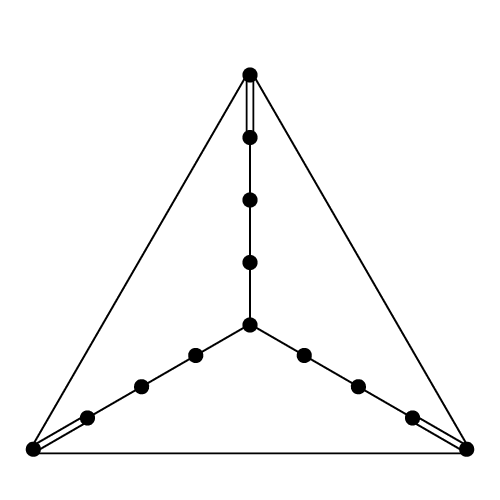}
\vspace{2.4mm}
\\ \hline
$\tilde{D_8}$  &
\resizebox{!}{1cm}{
\xy
(-10,25)*{};
@={(-10,10),(0,10),(10,10),(20,10),(30,10),(40,10),(50,10),(10,20),(60,10),(50,20)}@@{*{\bullet}};
(-10,10)*{};(60,10)*{}**\dir{-};
(10,10)*{};(10,20)*{}**\dir{-};
(50,10)*{};(50,20)*{}**\dir{-};
\endxy
} 
\\ \hline
{\rm VII} & 
\vspace{1mm}
\includegraphics[width=40mm]{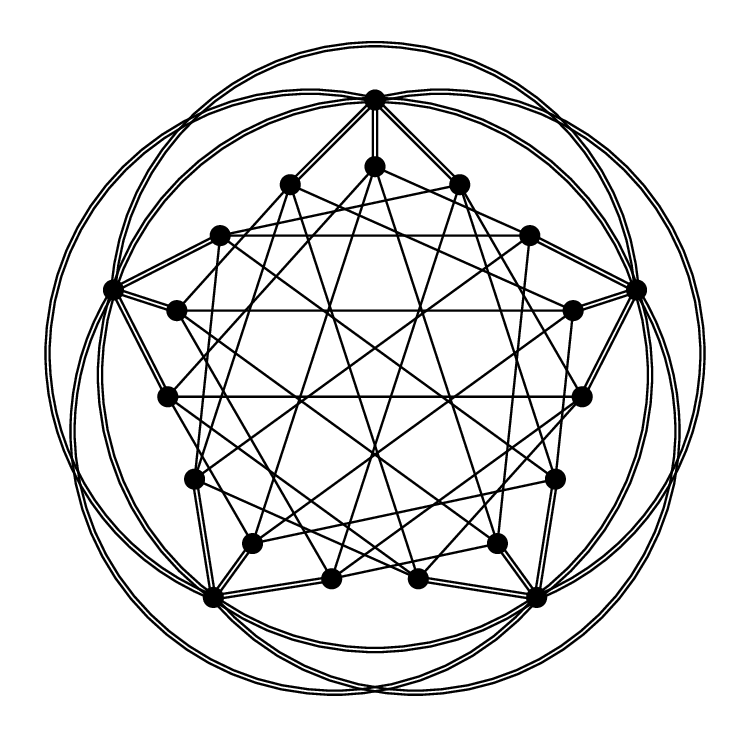}
\\
\hline
\end{tabular}}
\hspace{3mm}
\subfloat[Examples]{
\begin{tabular}{|>{\centering\arraybackslash}m{2cm}|>{\centering\arraybackslash}m{2cm}|>{\centering\arraybackslash}m{0.9cm}|@{}m{0pt}@{}}
\hline

\text{\rm $\Aut(X)$} & \text{\rm $\Aut_{ct}(X)$}& \text{\rm dim}\\

\hline
${\bf Z}/11{\bf Z}$
&
${\bf Z}/11{\bf Z}$
&
${\rm 0}$
&  \resizebox{!}{1cm}{ \phantom{a}}
\\ \hline
${\bf Z}/2{\bf Z}$ or ${\bf Z}/14{\bf Z}$
 &
$\{1\}$ or ${\bf Z}/7{\bf Z}$
&
${\rm 1}$ or\ \ ${\rm 0}$
& 
\resizebox{!}{1.5cm}{ \phantom{a}}
\\  \hline
${\bf Z}/5{\bf Z} \times \mathfrak{S}_3$
&
${\bf Z}/5{\bf Z}$
& 
${\rm 0}$
&  \vspace{2.8cm} \phantom{}
\\ \hline
$Q_8$
&
$Q_8$
& 
${\rm 1}$
&  \resizebox{!}{1.01cm}{ \phantom{a}}
\\  \hline
$\mathfrak{S}_5$
&
$\{1\}$
& 
${\rm 0}$
&  \vspace{3.6cm} \phantom{}
\\
\hline
\end{tabular}}
\vspace{-2mm}
\caption{}
\label{mainSupersingularTable}
\end{table}
\vspace{-5mm}
\end{theorem}

\begin{theorem}\label{mainClassical}
Let $X$ be a classical Enriques surface in characteristic $2$.
\begin{enumerate}[label=${\rm(\Alph*)}$]
\item $X$ has a finite group of automorphisms if and only if the dual graph of all $(-2)$-curves on $X$ is one of the graphs in Table {\rm\ref{mainClassicalTable} (A)}.
\item All cases exist. More precisely, we construct families of these surfaces whose automorphism groups and dimensions are given in Table {\rm\ref{mainClassicalTable} (B)}.
\end{enumerate}
\vspace{-1mm}
\begin{table}[!htb]
\centering
\subfloat[Classification]{
\begin{tabular}{|>{\centering\arraybackslash}m{1.7cm}|>{\centering\arraybackslash}m{7cm}|}
\hline
\rm Type & \rm Dual Graph of $(-2)$-curves \\
\hline
$\tilde{E_8}$ & 
\resizebox{!}{1cm}{
\xy
(-10,25)*{};
@={(-10,10),(0,10),(10,10),(20,10),(30,10),(40,10),(50,10),(60,10),(70,10),(10,20)}@@{*{\bullet}};
(-10,10)*{};(70,10)*{}**\dir{-};
(10,10)*{};(10,20)*{}**\dir{-};
\endxy
}
\\ \hline
$\tilde{E_7}+\tilde{A_1}^{(1)}$ &
\resizebox{!}{1cm}{
\xy
(0,25)*{};
@={(80,10),(90,10),(0,10),(10,10),(20,10),(30,10),(40,10),(50,10),(60,10),(70,10),(30,20)}@@{*{\bullet}};
(0,10)*{};(80,10)*{}**\dir{-};
(90,10)*{};(80,10)*{}**\dir{=};
(30,10)*{};(30,20)*{}**\dir{-};
(70,10)*{};(90,10)*{}**\crv{(80,20)};
\endxy
}
\\ \hline
$\tilde{E_7}+\tilde{A_1}^{(2)}$ &
\resizebox{!}{1cm}{
\xy
(0,25)*{};
@={(80,10),(90,10),(0,10),(10,10),(20,10),(30,10),(40,10),(50,10),(60,10),(70,10),(30,20)}@@{*{\bullet}};
(0,10)*{};(80,10)*{}**\dir{-};
(90,10)*{};(80,10)*{}**\dir{=};
(30,10)*{};(30,20)*{}**\dir{-};
\endxy 
}
\\ \hline
$\tilde{E}_6+\tilde{A_2}$ &
\vspace{0.5mm}
\includegraphics[width=32mm]{E6S6.eps}
\vspace{2.1mm}
\\ \hline
$\tilde{D_8}$ &
\resizebox{!}{1cm}{
\xy
(-10,25)*{};
@={(-10,10),(0,10),(10,10),(20,10),(30,10),(40,10),(50,10),(10,20),(60,10),(50,20)}@@{*{\bullet}};
(-10,10)*{};(60,10)*{}**\dir{-};
(10,10)*{};(10,20)*{}**\dir{-};
(50,10)*{};(50,20)*{}**\dir{-};
\endxy
}
\\ \hline
$\tilde{D}_4 + \tilde{D}_4$
&
\resizebox{!}{2cm}{
\xy
(0,25)*{};
@={(10,0),(50,0),(0,10),(10,10),(20,10),(30,10),(40,10),(50,10),(10,20),(60,10),(50,20)}@@{*{\bullet}};
(0,10)*{};(60,10)*{}**\dir{-};
(10,0)*{};(10,20)*{}**\dir{-};
(50,0)*{};(50,20)*{}**\dir{-};
\endxy
}
\\ \hline
{\rm VII} & 
\vspace{0.5mm}
\includegraphics[width=35mm]{VII1.eps}
\\ \hline
{\rm VIII} &
\vspace{1mm}
\includegraphics[width=35mm]{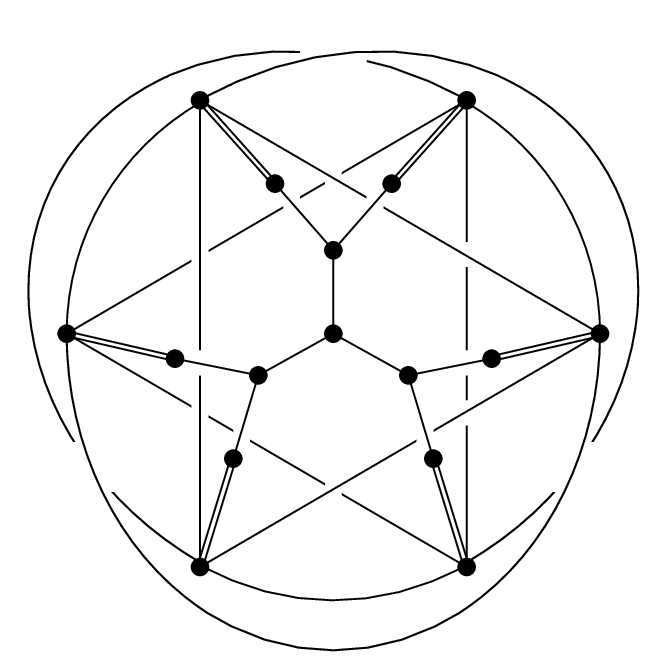}
\\ \hline
\end{tabular}}
\hspace{3mm}
\subfloat[Examples]{
\begin{tabular}{|>{\centering\arraybackslash}m{1.5cm}|>{\centering\arraybackslash}m{1.5cm}|>{\centering\arraybackslash}m{0.8cm}|@{}m{0pt}@{}}
\hline
{\rm $\Aut(X)$} & {\rm $\Aut_{nt}(X)$} & {\rm dim}  \\
\hline
$\{1\}$
&
$\{1\}$
&
${\rm 1}$
& \resizebox{!}{1cm}{ \phantom{a}}
\\ \hline
${\bf Z}/2{\bf Z}$
&
$\{1\}$
& 
${\rm 2}$
& \resizebox{!}{1cm}{ \phantom{a}}
\\  \hline
${\bf Z}/2{\bf Z}$ 
&
${\bf Z}/2{\bf Z}$
& 
${\rm 1}$
& \resizebox{!}{1cm}{ \phantom{a}}
\\ \hline
$\mathfrak{S}_3$
&
$\{1\}$
& 
${\rm 1}$
& \vspace{2.73cm} \phantom{}
\\ \hline
${\bf Z}/2{\bf Z}$
&
${\bf Z}/2{\bf Z}$
& 
${\rm 2}$
& \resizebox{!}{1cm}{ \phantom{a}}
\\ \hline
$({\bf Z}/2{\bf Z})^3$
&
$({\bf Z}/2{\bf Z})^2$
&
${\rm 2}$
& \resizebox{!}{2cm}{ \phantom{a}}
\\  \hline
$\mathfrak{S}_5$
&
$\{1\}$
& 
${\rm 1}$
& \vspace{3.05cm} \phantom{}
\\ \hline
$\mathfrak{S}_4$
&
$\{1\}$
& 
${\rm 1}$
& \vspace{3.1cm} \phantom{}
\\ \hline
\end{tabular}}
\vspace{-2mm}
\caption{}
\label{mainClassicalTable}
\vspace{-30mm}
\end{table}
\clearpage
\end{theorem}

\noindent
We remark that the examples of supersingular Enriques surfaces of type $\tilde{E}_7+\tilde{A}_1^{(1)}$ form a 1-dimensional family, but some of their automorphism groups jump up. 

Over the complex numbers, Enriques surfaces with numerically and cohomologically trivial automorphism groups are completely classified into three types and the groups are cyclic of order $2$ or $4$ (Mukai and Namikawa \cite{MN}, Mukai \cite{Mu}, and also see Kond\=o \cite[Theorem 1.7]{Ko}). This is not true in characteristic 2 and Theorems
\ref{mainSupersingular} \rm{(B)} and \ref{mainClassical} \rm{(B)} give new examples of such automorphisms (that is, the cases of $\tilde{E}_8$ supersingular, 
$\tilde{E}_7+\tilde{A}_1^{(1)}$ supersingular, $\tilde{E}_6+\tilde{A}_2$ supersingular, 
$\tilde{D}_8$ supersingular and $\tilde{D}_4+\tilde{D}_4$ classical) 
(compare with Dolgachev \cite[Theorem 4]{D2}).  Very recently, Dolgachev and the third author \cite{DM} gave
a classification of the possible numerically and cohomologically trivial automorphism groups of Enriques surfaces in positive characteristic up to the examples in Theorems \ref{mainSupersingular} and \ref{mainClassical}. In particular, they show that nothing new happens in characteristic different from 2.

\begin{corollary}
For $G \in \{Q_8,{\bf Z}/11{\bf Z},{\bf Z}/7{\bf Z},{\bf Z}/5{\bf Z}\}$, there exists a supersingular Enriques surface $X$ with $\Aut_{ct}(X) = G$. Moreover, there is a classical Enriques surface $X$ in characteristic $2$ with $\Aut_{nt}(X) = {\bf Z}/2{\bf Z} \times {\bf Z}/2{\bf Z}$.
\end{corollary}

Note that only the dual graph of type $\VII$ in Theorems \ref{mainSupersingular} and \ref{mainClassical} appears over the complex numbers. Moreover, the Enriques surface with the dual graph of type $\VII$ is unique over the complex numbers, whereas our example in characteristic 2 is a 1-dimensional and non-isotrivial family of classical and supersingular Enriques surfaces with such dual graph (see Theorem \ref{TypeVII}).
The canonical cover of any Enriques surface of type $\VII$ has 12 rational double points of type $A_1$ and its minimal resolution is the unique supersingular
$K3$ surface with Artin invariant 1.  The canonical covers of the other Enriques surfaces in Theorems \ref{mainSupersingular} and \ref{mainClassical} are non-normal rational surfaces.

The dual graphs of type $\tilde{E}_8$, $\tilde{E}_7+\tilde{A}_1^{(1)}$, $\tilde{E}_7+\tilde{A}_1^{(2)}$,
$\tilde{D}_8$ and 
$\tilde{D}_4+\tilde{D}_4$ appeared in Cossec and Dolgachev \cite{CD}, Dolgachev and Liedtke \cite{DL} and the first four are called "extra special".  Also,
Enriques surfaces of type $\tilde{E}_8$, $\tilde{E}_7+\tilde{A}_1^{(1)}$, $\tilde{E}_7+\tilde{A}_1^{(2)}$, 
$\tilde{E}_6+\tilde{A}_2$ are called "exceptional" and were studied deeply by Ekedahl and Shepherd-Barron \cite{ES} and Salomonsson \cite{Sa} from a different point of view. In particular, Salomonsson first constructed these exceptional surfaces. However, the existence of Enriques surfaces of type $\tilde{D}_8$ and 
$\tilde{D}_4+\tilde{D}_4$ was not known before our classification.


In case of a classical or supersingular Enriques surface $X$, there exists a non-zero regular global 1-form $\eta$ on $X$.
The divisorial part of the scheme of zeros of $\eta$ is called the bi-conductrix and the half of the biconductrix
the conductrix. By definition, the canonical cover $\pi : \tilde{X} \to X$ has a singularity at $P\in \tilde{X}$ if and only if $\eta$ vanishes at $\pi(P)$.
Ekedahl and Shepherd-Barron \cite{ES} classified possible conductrices of elliptic and quasi-elliptic fibrations on classical and supersingular Enriques surfaces. The conductrix is an additional invariant of Enriques surfaces with non-normal cover and it will play a central role in our classification.

The outline of the proof of Theorems \ref{mainSupersingular} and \ref{mainClassical} is as follows.
First, recall that any Enriques surface $X$ admits a genus one fibration $\pi : X\to {\bf P}^1$, and 
any genus one fibration on an Enriques surface has a double fiber.
Let $J(\pi) : J(X)\to {\bf P}^1$ be the Jacobian fibration associated with $\pi$.  Then the Mordell-Weil group of $J(\pi)$ acts on $X$ effectively as automorphisms.  
Now, assume that the automorphism group $\Aut(X)$ is finite.  
Then, for any genus one fibration $\pi$ on $X$, 
the Mordell-Weil rank of its Jacobian fibration is $0$ (see Proposition \ref{MW-Dolgachev}).  
We will prove that the possible dual graphs of 
$(-2)$-curves on $X$ are nothing but those given in Theorems \ref{mainSupersingular} \rm{(A)} and \ref{mainClassical} \rm{(A)}, by using the condition of Mordell-Weil rank mentioned as above and Ekedahl and Shepherd-Barron's classification of conductrices (Theorem \ref{mainGraph}).  
Then it follows from a result by Vinberg \cite{V} that $\Aut(X)$ is in fact finite for each Enriques 
surface $X$ with one of these dual graphs of $(-2)$-curves (Proposition \ref{Vinberg}).  

On the other hand, for each dual graph $\Gamma$ in Theorems \ref{mainSupersingular} \rm{(A)} and \ref{mainClassical} \rm{(A)}, we will construct Enriques surfaces with $\Gamma$ as the dual graph of $(-2)$-curves (Sections \ref{sec3}--\ref{sec5}).  To do this, we look at a subdiagram
$\Gamma_0$ of $\Gamma$ which is the dual graph of reducible fibers of a special genus one fibration.
Here, a genus one fibration is called special if the fibration has a $(-2)$-curve as a $2$-section.  We first consider a rational genus one fibration $g: R \to {\bf P}^1$ whose dual graph of
reducible fibers is $\Gamma_0$, and we take the Frobenius base change $f : \tilde{R} \to {\bf P}^1$ of $g$.  Then, we construct a rational vector field $D$ on $\tilde{R}$ (in Section \ref{VectFields}, we give a method to find suitable vector fields).  
The vector field $D$ might have isolated singularities and hence we take a resolution of singularities, that is, after blow-ups of $\tilde{R}$ 
we get a non-singular surface $Y$ such that the induced vector field denoted by the same symbol $D$
has no isolated singularities.  Then, the quotient surface $Y^D$ of $Y$ by $D$ is non-singular and the minimal model $X$ of $Y^D$ is the desired Enriques surface. 

Finally, we give a remark on how to calculate the automorphism group $\Aut(X)$, which is isomorphic to a subgroup of the symmetry group of the dual graph $\Gamma$ up to numerically trivial automorphisms.
In cases $\tilde{E}_6+\tilde{A}_2$ (supersingular), 
$\tilde{E}_8$ (supersingular and classical), $\tilde{E}_7+\tilde{A}_1^{(1)}$ (supersingular),  
$\tilde{D}_8$ (supersingular and classical) and $\tilde{D}_4+\tilde{D}_4$, we can not
determine the numerically trivial automorphisms from their dual graphs of $(-2)$-curves geometrically.  
In these cases, we first find an equation of an affine surface birationally equivalent to $X$, and then we reduce the problem to the calculation of 
the automorphism group of this surface (see Section \ref{Equation}).

From Section 3 to 12, unless mentioned otherwise, all our Enriques surfaces are classical or supersingular.

\medskip
\noindent
{\bf Acknowledgement.}
The authors thank I. Dolgachev for stimulating discussions on Enriques surfaces 
in positive characteristics.
The second author thanks N. Shepherd-Barron for explaining the conductrices, and G. van der Geer
for discussions and his hospitality when he was visiting in Oberwolfach and Amsterdam in 2016.
The third author thanks his Ph.D. advisor C. Liedtke for his support and helpful discussions, and the second author for an invitation to Nagoya University in November 2016.
The authors thank the referee for careful reading of the manuscript, for pointing out errors and for many useful suggestions, which greatly improved the readability of this article.

\section{Preliminaries}\label{sec2}

\subsection{Vector fields}\label{derivation}

Let $k$ be an algebraically closed field of characteristic $p > 0$,
and let $S$ be a non-singular complete algebraic surface defined over $k$.
We denote by $K_{S}$ a canonical divisor of $S$.
A rational vector field $D$ on $S$ is said to be $p$-closed if there exists
a rational function $f$ on $S$ such that $D^p = fD$. 
A vector field $D$ is of additive type (resp. of multiplicative type) if $D^p=0$ (resp. $D^p=D$).
Let $\{U_{i} = {\rm Spec} A_{i}\}$ be an affine open covering of $S$. We set 
$A_{i}^{D} = \{\alpha \in A_{i} \mid D(\alpha) = 0\}$. 
The affine varieties $\{U_{i}^{D} = {\rm Spec} A_{i}^{D}\}$ glue together to 
define a normal quotient surface $S^{D}$.

Now, we  assume that $D$ is $p$-closed. Then,
the natural morphism 
\begin{equation}\label{naturalmapPI}
\pi : S \longrightarrow S^D 
\end{equation}
is a purely
inseparable morphism of degree $p$. 
If the affine open covering $\{U_{i}\}$ of $S$ is fine enough, then
taking local coordinates $x_{i}, y_{i}$
on $U_{i}$, we see that there exist $g_{i}, h_{i}\in A_{i}$ and 
a rational function $f_{i}$
such that the divisors defined by $g_{i} = 0$ and by $h_{i} = 0$ have no common components,
and such that
$$
 D = f_{i}\left(g_{i}\frac{\partial}{\partial x_{i}} + h_{i}\frac{\partial}{\partial y_{i}}\right)
\quad \mbox{on}~U_{i}.
$$
By Rudakov and Shafarevich \cite[Section 1]{RS}, the divisors $(f_{i})$ on $U_{i}$
give a global divisor $(D)$ on $S$, and the zero-cycles defined
by the ideal $(g_{i}, h_{i})$ on $U_{i}$ give a global zero cycle 
$\langle D \rangle $ on $S$. A point contained in the support of
$\langle D \rangle $ is called an isolated singular point of $D$.
If $D$ has no isolated singular point, $D$ is said to be divisorial.
Rudakov and Shafarevich \cite[Theorem 1, Corollary]{RS} 
showed that $S^D$ is non-singular
if $\langle D \rangle  = 0$, i.e. $D$ is divisorial.
When $S^D$ is non-singular,
they also showed a canonical divisor formula
\begin{equation}\label{canonical}
K_{S} \sim \pi^{*}K_{S^D} + (p - 1)(D),
\end{equation}
where $\sim$ means linear equivalence.
As for the Euler number $c_{2}(S)$ of $S$, we have a formula
\begin{equation}\label{euler}
c_{2}(S) = \deg \langle D \rangle  - K_{S}\cdot (D) - (D)^2
\end{equation}
(cf. Katsura and Takeda \cite[Proposition 2.1]{KT}). 

Now we consider an irreducible curve $C$ on $S$ and we set $C' = \pi (C)$.
Take an affine open set $U_{i}$ above such that $C \cap U_{i}$ is non-empty.
The curve $C$ is said to be integral with respect to the vector field $D$
if $g_{i}\frac{\partial}{\partial x_{i}} + h_{i}\frac{\partial}{\partial y_{i}}$
is tangent to $C$ at a general point of $C \cap U_{i}$. Then, Rudakov-Shafarevich
\cite[Proposition 1]{RS} showed the following proposition:

\begin{prop}\label{insep}

$({\rm i})$  If $C$ is integral, then $C = \pi^*(C')$ and $C^2 = pC'^2$.

$({\rm ii})$  If $C$ is not integral, then $pC = \pi^*(C')$ and $pC^2 = C'^2$.
\end{prop}

\subsection{Enriques surfaces in characteristic 2}\label{surfaces}

In characteristic 2, a minimal algebraic surface with numerically trivial
canonical divisor is called an Enriques surface if the second Betti
number is equal to 10. Such surfaces $X$ are  divided into three classes
(for details, see Bombieri and Mumford \cite[Section 3]{BM2}):
\begin{itemize}
\item[$({\rm i})$] $K_{X}$ is not linearly equivalent to zero 
and $2K_{X}\sim 0$.  Such an Enriques surface is called a classical Enriques surface.
\item[$({\rm ii})$] $K_{X} \sim 0$, ${\rm H}^{1}(X, {\calO}_{X}) \cong k$
and the Frobenius map acts on  ${\rm H}^{1}(X, {\calO}_X)$ bijectively.
Such an Enriques surface is called a singular Enriques surface.
\item[$({\rm iii})$] $K_{X} \sim 0$, ${\rm H}^{1}(X, {\calO}_{X}) \cong k$
and the Frobenius map is the zero map on  ${\rm H}^{1}(X, {\calO}_X)$.
Such an Enriques surface is called a supersingular Enriques surface.
\end{itemize}

\noindent
It is known that $\Pic^{\tau}_X$ is isomorphic to ${\bf Z}/2{\bf Z}$ if $X$ is classical, $\mu_2$ if $X$ is singular or $\alpha_2$ if $X$ is supersingular (Bombieri-Mumford \cite[Theorem 2]{BM2}). 
As in the case of characteristic $0$ or $p>2$, an Enriques surface
$X$ in characteristic 2 has a canonical double cover
$\pi : \tilde{X} \to X$, which is a separable ${\bf Z}/2{\bf Z}$-cover,
a purely inseparable $\mu_2$- or $\alpha_2$-cover according to $X$ being singular, classical or supersingular. The surface $\tilde{X}$ might have singularities and it might even be non-normal
(see Proposition \ref{ES}), but it is $K3$-like in the sense that its dualizing sheaf is trivial 
and ${\rm H}^{1}(\tilde{X}, {\calO}_{\tilde{X}})=0$. 
Note that Ekedahl and Shepherd-Barron \cite{ES} use the terminology "unipotent" Enriques surfaces for supersingular and classical ones.

\subsection{$(-2)$-curves}

Let $X$ be an Enriques surface and let $\Num(X)$ be the quotient of the N\'eron-Severi group $\NS(X)$ of $X$ by its torsion subgroup.  Then $\Num(X)$ together with the intersection product is
an even unimodular lattice of signature $(1,9)$ (Illusie \cite[Corollary 7.3.7]{Ill}), and hence is isomorphic to $U\oplus E_8$ where $U$ is the even unimodular lattice of signature $(1,1)$ and $E_8$ the even negative definite unimodular lattice of rank 8.
We denote by ${\rm O}(\Num(X))$ the orthogonal group of $\Num(X)$. The set 
$$\{ x \in \Num(X)\otimes {\bf R} \ : \ \langle x, x \rangle > 0\}$$ 
has two connected components.
Denote by $P(X)$ the connected component containing an ample class of $X$.  
For $\delta \in \Num(X)$ with $\delta^2=-2$, we define
an isometry $s_{\delta}$ of $\Num(X)$ by
$$s_{\delta}(x) = x + \langle x, \delta\rangle \delta, \quad x \in \Num(X),$$ 
which is nothing but the reflection with respect to the hyperplane perpendicular to $\delta$.
The isometry $s_{\delta}$ is called the reflection associated with $\delta$.
We call a non-singular rational curve on an Enriques surface or a $K3$ surface a $(-2)$-curve.
For a $(-2)$-curve $E$ on an Enriques surface $X$, we identify $E$ with its class in $\Num(X)$.
Let $W(X)$ be the subgroup of
${\rm O}(\Num(X))$ generated by reflections associated with all $(-2)$-curves on $X$.  Then $P(X)$ is divided into chambers 
each of which is a fundamental domain with respect to
the action of $W(X)$ on $P(X)$.
There exists a unique chamber containing an ample
class which is nothing but the closure of the ample cone $D(X)$ of $X$.
It is known that the natural map
\begin{equation}\label{num-trivial}
\rho_n : \Aut(X) \to {\rm O}(\Num(X))
\end{equation}
has a finite kernel.  
Since the image $\Im(\rho_n)$ preserves the ample cone, we see $\Im(\rho_n) \cap W(X) = \{1\}$.
Therefore $\Aut(X)$ is finite if the index $[\O(\Num(X)) : W(X)]$ is finite.  
Thus, we have the following Proposition (see Dolgachev \cite[Proposition 3.2]{D1}).

\begin{prop}\label{finiteness}
If $W(X)$ is of finite index in ${\rm O}({\rm Num}(X))$, then ${\rm Aut}(X)$ is finite.
\end{prop}

\noindent
Over the field of complex numbers, the converse of Proposition \ref{finiteness} 
holds by using the Torelli type theorem for Enriques surfaces (Dolgachev \cite[Theorem 3.3]{D1}).

\begin{definition}\label{cohnumtrivial}
Denote by $\Aut_{nt}(X)$ the kernel of the map $\rho_n$ given by (\ref{num-trivial}).
Similarly, denote by $\Aut_{ct}(X)$ the kernel of the map
\begin{equation}\label{coh-trivial}
\rho_c : \Aut(X) \to {\rm O}(\NS(X)).
\end{equation}
A non-trivial automorphism is called cohomologically or numerically trivial if it is contained in 
$\Aut_{ct}(X)$ or $\Aut_{nt}(X)$, respectively.
If $S$ is not classical, then $\NS(X) = \Num(X)$ and hence $\Aut_{ct}(X) = \Aut_{nt}(X)$.
\end{definition}

\subsection{Genus one fibrations}

We recall some facts on elliptic and quasi-elliptic fibrations on Enriques surfaces.
For simplicity, we call an elliptic or a quasi-elliptic fibration a genus one fibration.

\begin{prop}\label{genus1}{\rm (Bombieri and Mumford \cite[Theorem 3]{BM2})}
Every Enriques surface $S$ has a genus one fibration.  Conversely, for any primitive nef isotropic divisor $D$ on 
$S$, the linear system $|D|$ or $|2D|$ defines a genus one fibration on $S$.
\end{prop}

\begin{prop}\label{multi-fiber}{\rm (Cossec and Dolgachev \cite[Theorems 5.7.5, 5.7.6]{CD})}

Let $f : X \to {\bf P}^1$ be a genus one fibration on an Enriques surface 
$X$ in characteristic $2$.  Then, the following hold.

$({\rm i})$  If $X$ is classical, then $f$ has two tame double fibers, each of which is
either an ordinary elliptic curve or a singular fiber of additive type.

$({\rm ii})$  If $X$ is singular, then $f$ has one wild double 
fiber which is a smooth ordinary elliptic curve or a singular fiber of multiplicative type.

$({\rm iii})$  If $X$ is supersingular, then $f$ has one wild double fiber which is a
supersingular elliptic curve or a singular fiber of additive type.
\end{prop}

\begin{lemma}\label{smoothdoublefiber}
Let $f: X \to {\bf P}^1$ be an isotrivial genus one fibration on an Enriques surface in characteristic $2$.
Let $F$ be a double fiber of $f$ such that the underlying reduced fiber $F_{red}$ is an elliptic curve.
Then $F_{red}$ has $j$-invariant $0$ if and only if the generic fiber of $f$ also has $j$-invariant $0$.
\end{lemma}

\begin{proof}
We can assume that the general fiber of $f$ is an elliptic curve.
Since $f$ is isotrivial, it becomes trivial after passing to a finite cover of ${\bf P}^1$. Hence, $F$ is isogeneous to the generic fiber of $f$. Since having $j$-invariant $0$ is equivalent to being supersingular in characteristic $2$ and being supersingular is an isogeny-invariant, we get the result.
\end{proof}

We use the symbols $\I_n$ $(n\geq 1)$, $\I_n^*$ $(n\geq 0)$, $\II$, $\III$, $\IV$, $\II^*$, 
$\III^*$, $\IV^*$ of singular fibers of an elliptic or a quasi-elliptic fibration in the sense of Kodaira.  
The dual graph of $(-2)$-curves in a singular fiber of type $\I_n$ $(n\geq 2)$, $\I_n^*$ $(n\geq 0)$, 
$\III$, $\IV$, $\II^*$, $\III^*$ or $\IV^*$ is an extended Dynkin diagram $\tilde{A}_{n-1}$, $\tilde{D}_{n+4}$, 
$\tilde{A_1}$, $\tilde{A_2}$, $\tilde{E_8}$, $\tilde{E_7}$ or $\tilde{E_6}$, respectively.
For a double singular fiber of type $F$, we write $2F$.
Let $f : S \to {\bf P}^1$ be a genus one fibration on a surface $S$.  If, for example, $f$ has a double singular fiber of type $\III$ and a singular fiber of type $\IV^*$, then we say that $f$ has singular fibers $(2\III, \IV^*)$.
If $f$ has a section and its Mordell-Weil group is torsion, then $f$ is called extremal.
We use the following classifications of extremal rational elliptic and rational quasi-elliptic fibrations.

\begin{prop}\label{Lang}{\rm (Lang \cite{L1}, \cite{L2})}
The following are the singular fibers of extremal elliptic fibrations on rational surfaces in characteristic $2:$
$$({\rm II}^*),\ ({\rm II}^*, {\rm I}_1),\ ({\rm III}^*, {\rm I}_2),\ ({\rm IV}^*, {\rm IV}),\ ({\rm IV}^*, {\rm I}_3, {\rm I}_1),\ ({\rm I}_4^*),\ ({\rm I}_1^*, {\rm I}_4),$$
$$({\rm I}_9, {\rm I}_1, {\rm I}_1, {\rm I}_1),\ ({\rm I}_8, {\rm III}),\ ({\rm I}_6, {\rm IV}, {\rm I}_2),\ ({\rm I}_5, {\rm I}_5, {\rm I}_1, {\rm I}_1),\ ({\rm I}_3, {\rm I}_3, {\rm I}_3, {\rm I}_3).$$
\end{prop} 

\begin{prop}\label{Itoh}{\rm (Ito \cite{Ito})}
The following are the singular fibers of quasi-elliptic fibrations on rational surfaces in characteristic $2:$
$$({\rm II}^*),\ ({\rm III}^*, {\rm III}),\ ({\rm I}_4^*),\ ({\rm I}_2^*, {\rm III}, {\rm III}), \ 
({\rm I}_0^*, {\rm I}_0^*),$$
$$({\rm I}_0^*, {\rm III}, {\rm III}, {\rm III}, {\rm III}),\ ({\rm III}, {\rm III}, {\rm III}, {\rm III}, {\rm III}, {\rm III}, {\rm III}, {\rm III}).$$
\end{prop} 

\begin{remark}
Any quasi-elliptic fibration on a rational surface is extremal. This follows from the fact
that any section of a quasi-elliptic fibration with a section is of finite order
(cf. Ito \cite[Theorems 2.2, 2.4 and Corollary 2.12]{I1}.
\end{remark}

Consider a genus one fibration on an Enriques surface $\pi : X \to {\bf P}^1$.  Then the Mordell-Weil group of the Jacobian of $\pi$ acts on $X$ effectively as automorphisms.  This implies the following Proposition.

\begin{prop}\label{MW-Dolgachev}{\rm (Dolgachev \cite[$\S 4$]{D1})}
Assume that the automorphism group of an Enriques surface $X$ is finite.  Then any genus one fibration on $X$ is extremal.  
\end{prop}

Let $X$ be an Enriques surface.
A genus one fibration $f : X \to {\bf P}^1$ is called special if there exists a $(-2)$-curve $R$ with
$R \cdot f^{-1}(P) = 2$ $(P\in {\bf P}^1)$, that is, $f$ has a $(-2)$-curve as a $2$-section.  In this case, $R$ is called a special 2-section.  Note that for any quasi-elliptic fibration, the locus of singular points of irreducible fibers gives a special 2-section of the fibration.  We call this special 2-section
the curve of cusps of the fibration.
The following result is due to Cossec \cite{C} in which he assumed the characteristic $p\not=2$, but the assertion for $p=2$ holds, too.

\begin{prop}\label{Cossec}{\rm (Lang \cite[$\II$, Theorem A3]{L0})}
Assume that an Enriques surface $X$ contains a $(-2)$-curve.  
Then there exists a special genus one fibration on $X$.
\end{prop} 

\begin{remark}\label{quasi-elliptic-special}
Any quasi-elliptic fibration is special because the cuspidal curve (that is, the curve consisting of 
the cusps of the fibers) is a special 2-section.
\end{remark}

For future reference, let us state the following well-known facts about automorphisms of curves of arithmetic genus $1$ (compare e.g. \cite{DM}).

\begin{lemma}\label{genus1auto}
Let $C$ be a curve of arithmetic genus $1$ over an algebraically closed field of characteristic $2$ and let ${\rm id} \neq g \in {\rm Aut}(C)$ be an automorphism of finite order.
\begin{enumerate}
\item If $g$ has even (resp. odd) order and $C$ is smooth, then $g$ fixes at most $2$ (resp. $3$) points on $C$.
\item If $g$ has even (resp. odd) order and $C$ is cuspidal, then $g$ fixes exactly $1$ (resp. $2$) points on $C$.
\end{enumerate}
\end{lemma}

\subsection{Vinberg's criterion}

Let $X$ be an Enriques surface.
We recall Vinberg's criterion,
which guarantees that a group generated by a finite number of reflections is
of finite index in ${\rm O}(\Num(X))$.

Let $\Delta$ be a finite set of $(-2)$-vectors in $\Num(X)$.
Let $\Gamma$ be the graph of $\Delta$, that is,
$\Delta$ is the set of vertices of $\Gamma$ and two vertices $\delta$ and $\delta'$ are joined by $m$-tuple lines if $\langle \delta, \delta'\rangle=m$.
We assume that the cone
$$K(\Gamma) = \{ x \in \Num(X)\otimes {\bf R} \ : \ \langle x, \delta_i \rangle \geq 0, \ \delta_i \in \Delta\}$$
is a strictly convex cone. Such a $\Gamma$ is called non-degenerate.
A connected parabolic subdiagram $\Gamma'$ in $\Gamma$ is a  Dynkin diagram of type $\tilde{A}_m$, 
$\tilde{D}_n$ or $\tilde{E}_k$ (see Vinberg \cite[p. 345, Table 2]{V}).  If the number of vertices of 
$\Gamma'$ is $r+1$, then $r$ is called the rank of $\Gamma'$.  A disjoint union of connected parabolic subdiagrams is called a parabolic subdiagram of $\Gamma$.  We denote by $\tilde{K_1}\oplus \cdots \oplus  \tilde{K_s}$ a parabolic subdiagram which is a disjoint union of 
connected parabolic subdiagrams of type $\tilde{K_1}, \ldots, \tilde{K_s}$, where
$K_i$ is $A_m$, $D_n$ or $E_k$. The rank of a parabolic subdiagram is the sum of the ranks of its connected components.  Note that the dual graph of reducible singular fibers of a genus one fibration on 
$X$ gives a parabolic subdiagram.  
We denote by $W(\Gamma)$ the subgroup of ${\rm O}(\Num(X))$ 
generated by reflections associated with $\delta \in \Gamma$.

\begin{prop}\label{Vinberg}{\rm (Vinberg \cite[Theorem 2.3]{V})}
Let $\Delta$ be a set of $(-2)$-vectors in ${\rm Num}(X)$
and let $\Gamma$ be the graph of $\Delta$.
Assume that $\Delta$ is a finite set, $\Gamma$ is non-degenerate and $\Gamma$ contains no $m$-tuple lines with $m \geq 3$.  Then $W(\Gamma)$ is of finite index in ${\rm O}({\rm Num}(X))$ if and only if every connected parabolic subdiagram of $\Gamma$ is a connected component of some
parabolic subdiagram in $\Gamma$ of rank $8$ {\rm (}= the maximal one{\rm )}.
\end{prop} 
\noindent

\begin{remark}\label{Vinbergremark}
Note that $\Gamma$ as in the above proposition is automatically non-degenerate if it contains the components of the reducible fibers of a special extremal genus one fibration and a special $2$-section of this fibration. Indeed, these curves will generate ${\rm Num}(X)\otimes {\bf Q}$ and hence $K(\Gamma)$ is strictly convex.
\end{remark}

\begin{prop}\label{Namikawa}{\rm (Namikawa \cite[Proposition 6.9]{Na})}
Let $\Delta$ be a finite set of $(-2)$-curves on an Enriques surface $X$ and let
$\Gamma$ be the graph of $\Delta$. 
Assume that $W(\Gamma)$ is of finite index in ${\rm O}({\rm Num}(X))$.  Then $\Delta$ is the set of all
$(-2)$-curves on $X$.
\end{prop}

\section{Conductrix}\label{conductrix}

Let $X$ be a classical or supersingular Enriques surface.  Then it is known that there exists a global 
regular 1-form $\eta$ on $X$.  The canonical cover $\pi : \tilde{X} \to X$ has a singularity at $P\in \tilde{X}$
if and only if $\eta$ vanishes at $\pi(P)$. Since $c_2(X)=12$, $\eta$ always vanishes somewhere, and hence
$\tilde{X}$ is singular.  
The divisorial part $B$ of the zero scheme of $\eta$ is called the bi-conductrix of $X$.
The divisor $B$ is of the form $2A$, where $A$ is a divisor called the conductrix of $X$.

The purpose of this section is to recall the results of Ekedahl and Shepherd-Barron \cite{ES} and study the interplay between the conductrix and genus one fibrations. In particular, we will make extensive use of their tables \cite[p.13]{ES} and \cite[pp.16-18]{ES}. Then, we will apply this knowledge to special extremal genus one fibrations. This will lead to the classification of dual graphs of Enriques surfaces with finite automorphism group in the next section.

In this section, unless mentioned otherwise, all Enriques surfaces are classical or supersingular in characteristic 2. For simplicity, we write $A_1$-singularity or $D_4$-singularity for a rational double
point of type $A_1$ or of type $D_4$, respectively.  Also, we will use the symbol $nA_1$ for $n$ rational double points of type $A_1$ (in characteristic 2, there are two types $D_4^0$, $D_4^1$ of singularities 
with the same dual graph $D_4$ of exceptional curves.  However, we need only the dual graph as 
information and do not keep track of the isomorphism class).

\subsection{Singularities of the canonical cover}
In \cite{ES}, Ekedahl and Shepherd-Barron studied "exceptional" Enriques surfaces using the conductrix associated to their canonical cover. Recall the following structural result.

\begin{prop}\label{ES}{\rm (Ekedahl and Shepherd-Barron \cite[Proposition 0.5]{ES}
)}
Let $X$ be an Enriques surface and $A$ its conductrix.
Assume $A\not= 0$.  Then, $A$ is $1$-connected. Moreover, $A^2=-2$, $A$ is supported on $(-2)$-curves and the normalization of the canonical cover has either four rational double points of type $A_1$ as singularities or one rational double point of type $D_4$.
\end{prop} 

The following Lemma is the reason why there is a relation between the conductrix of an Enriques surface and singular fibers of its genus one fibrations.
\begin{lemma}\label{frobenius}{\rm (}Ekedahl and Shepherd-Barron \cite[Lemma 0.9]{ES}{\rm )}
Let $X$ be an Enriques surface, $\rho: \tilde{X} \to X$ its canonical cover and $\pi: X \to {\bf P}^1$ a genus one fibration.
Then the morphism $\rho$ factors through the pullback $X_F$ of $\pi$ by the Frobenius map on ${\bf P}^1$. The map $\tilde{X} \to X_F$ is an isomorphism outside of the double fibers of $\pi$.
\end{lemma}

\begin{lemma}\label{quellorell}
Let $X$ be an Enriques surface with conductrix $A$. Let $\pi$ be a genus one fibration on $X$.
\begin{enumerate}
\item[$(1)$] If $\pi$ is a quasi-elliptic fibration, then the curve of cusps of $\pi$ is a component of $A$ with multiplicity $1$ and all other components of $A$ are contained in fibers of $\pi$.
\item[$(2)$] If $\pi$ is an elliptic fibration, then $A$ is contained in one fiber of $\pi$. 
\end{enumerate}
In particular, $\pi$ is elliptic if and only if $A$ is contained in a fiber of $\pi$.
\end{lemma}

\begin{proof}
A non-zero regular $1$-form $\eta$ on $X$ is given by the pullback of a regular $1$-form on ${\bf P}^1$ (see \cite{Ka2}). Assume $\pi$ is quasi-elliptic. Let $F$ be a general cuspidal fiber and $t$ a local parameter at $\pi(F)$. Then, locally around the cusp, $F$ is given by the equation $\pi^*t = y^2 + x^3$
(Bombieri-Mumford \cite[Proposition 4]{BM2}), hence $\pi^*(dt) = x^2dx$ which vanishes twice at the cusp. 
Therefore, the curve of cusps is a component of $A$ with multiplicity $1$.  Similarly, for an arbitrary genus one fibration, one shows that $\eta$ does not vanish on any smooth point of a fiber of $\pi$. Since $A$ is connected by Proposition \ref{ES}, this yields the second claim.
\end{proof}

Recall that the minimal dissolution of a double cover $Y \to X$ of surfaces with $X$ smooth and $Y$ normal is the successive blow-up of points on $X$ lying under singular points of $Y$. For an Enriques surface $X$ we call the minimal dissolution of the double cover $\tilde{X}_{norm} \to X$, where $\tilde{X}_{norm}$ is the normalization of the canonical cover $\tilde{X}$, the minimal dissolution of $X$ and denote it by $X_{diss}$. The normalization $\tilde{X}_{sm}$ of $X_{diss}$ in $K(\tilde{X})$ is the minimal resolution of singularities of $\tilde{X}_{norm}$ if $\tilde{X}_{norm}$ has only rational singularities (which holds e.g. if $\tilde{X}_{norm} \neq \tilde{X}$ by Proposition \ref{ES}), but it is not minimal in general. The following diagram, where the vertical arrows are finite morphisms of degree $2$ and the horizontal arrows are birational morphisms, summarizes this discussion.

\centerline{
$
\xymatrix{
\tilde{X}_{sm} \ar[r] \ar[d] & \tilde{X}_{norm} \ar[r] & \tilde{X} \ar[d] & \text{canonical cover} \\
X_{diss}  \ar[rr] & & X & \text{Enriques surface}\\
}
$
}

Now, we recall the results of Ekedahl and Shepherd-Barron \cite{ES} on what happens to $(-2)$-curves on $X$ when taking their inverse image in $\tilde{X}_{sm}$ and additionally study curves of arithmetic genus $1$.

\begin{lemma} \label{blowups}
With the notation introduced above, let $C$ be an irreducible curve of arithmetic genus at most $1$ on an Enriques surface $X$ with conductrix $A$. Denote the irreducible curve on $\tilde{X}_{sm}$ mapping surjectively to $C$ by $\tilde{C}$ and let $\rho: \tilde{X}_{sm} \to \tilde{X}$ and $\pi: \tilde{X} \to X$ be the morphisms from the normalization of the minimal dissolution of $X$ to $\tilde{X}$ and from $\tilde{X}$ to $X$, respectively. We fix the following invariants{\rm :}
\begin{itemize}
\item[$({\rm i})$] The degree $s$ of $(\pi \circ \rho)|_{\tilde{C}}: \tilde{C} \to C$.
\item[$({\rm ii})$] The number $r$ of points $($possibly including infinitely near ones$)$ on $C$ which are 
blown up during the minimal dissolution of $X$, and their multiplicity $m$.
\item[$({\rm iii})$] The intersection number $A\cdot C$.
\item[$({\rm iv})$] The self-intersection numbers $\tilde{C}^2$ and $C^2$.
\item[$({\rm v})$] The arithmetic genera $p_a(C)$ and $p_a(\tilde{C})$.
\item[$({\rm vi})$] If $p_a(C) = 1$, the type $\rm{Sing}$ of singularity of $C$. This is either nodal $n$, cuspidal $c$ or smooth $sm$.
\end{itemize}

Then $\tilde{C}$ satisfies the following{\rm :}
\begin{enumerate}
\item[$(1)$]  $\tilde{C}^2 = (C^2 - m^2r)s^2/2$ and $2p_a(\tilde{C}) - 2 = \tilde{C}^2 - s A\cdot C$
\item[$(2)$] If two curves meet transversally on $X$ and both have $s$-invariant $1$, then they do not meet on $X_{diss}$.
\item[$(3)$] For $A\cdot C \geq -2$ and $p_a(C) = 0$, we have the following possibilities

\begin{center}
$
\begin{array}{ccccc}
 r & s & A\cdot C & \tilde{C}^2 & p_a(\tilde{C}) \\ \hline 
0 & 1 & 1 & -1 & 0 \\
0 & 2 & -1 & -4 & 0 \\
2 & 1 & 0 & -2 & 0 \\
4 & 1 &-1 & -3 & 0 \\
6 & 1 & -2 & -4 & 0 \\
1 & 2 & -2 & -6 & 0 
\end{array}
$
\end{center}

\item[$(4)$] For $p_a(C) = 1$, we have the following possibilities

\begin{center}
$
\begin{array}{ccccccc}
\rm{Sing} &  r & m &  s & A\cdot C & \tilde{C}^2 & p_a(\tilde{C}) \\ \hline 
sm &  0 &  & 1 & 0 & 0 & 1 \\
sm &  0 &  & 2 & 0 & 0 & 1 \\
n & 1 & 2 & 1 & 0 & -2 & 0 \\
c &  0 & & 1 & 0 & 0 & 1 \\
c &  0 & & 2 & 0 & 0 & 1 \\
c &  1 & 2 & 1 & 0 & -2 & 0 \\
c &  4 & 1 & 1 & 0 & -2 & 0 \\
c & 2 & 1 & 1 & 1 & -1 & 0 \\
c & 0 &  & 1 & 2 & 0 & 0 \\
\end{array}
$
\end{center}

\item[$(5)$] If $C$ is a cuspidal curve such that
\begin{itemize}
\item $|C|$ defines a quasi-elliptic fibration, then $r = 0$ and $s = 1$
\item $|C|$ defines an elliptic fibration, then $r = 1$, $m = 2$ and $s = 1$
\item $|C|$ does not define a quasi-elliptic fibration and $|2C|$ defines a quasi-elliptic fibration, then $r = 2$, $m = 1$ and $s = 1$.
\end{itemize}
\end{enumerate}

\end{lemma}

\begin{proof}
Similar to Ekedahl and Shepherd-Barron \cite{ES}, the formulas for the self-intersection number and the genus of $\tilde{C}$ are obtained by observing that the self-intersection number of $C$ drops by $m^2$ for every point of multiplicity $m$ on $C$ which is blown up during the minimal dissolution and from $\omega_{\tilde{X}/X} = \pi^*(\mathcal{O}_X(-A))$. Also, 
the claim (2) is in \cite{ES}.

The first table is contained in \cite{ES} and we will only establish the second one. Therefore, assume that $p_a(C) = 1$. If $C$ is smooth, then $A\cdot C = 0$ by Lemma \ref{quellorell} which only leaves the two possibilities listed. If $C$ has a node, then $|C|$ defines an elliptic fibration $\varphi$ with $C$ as a simple fiber. Therefore, formally locally around $C$, $X$ is isomorphic to the Jacobian of $\varphi$ and by Lemma \ref{frobenius} we can find $\tilde{C}$ by doing Frobenius pullback along the base. But on an $\I_1$ fiber, an elliptic surface acquires an $A_1$-singularity at the singular point of the nodal curve after Frobenius pullback. Therefore, the node of $C$ is blown up during the minimal dissolution. A similar argument works if $C$ is cuspidal and $|C|$ defines an elliptic fibration.

If $C$ is cuspidal, we have enumerated all numerical possibilities except for the ones where $p_a(\tilde{C}) = 0$ and $s = 2$. These cases do not occur. In fact, assume that $s = 2$ and $p_a(\tilde{C}) = 0$. Denote the image of $\tilde{C}$ on $\tilde{X}_{norm}$ by $\tilde{C}'$. Since the singular point of $C$ is not blown up during the dissolution (by the self-intersection formula), we have $\tilde{C}' \cong \tilde{C} \cong {\bf P}^1$. Then, the flat morphism $\varphi: \tilde{X}_{norm} \to X$ restricts to a morphism $\varphi|_{\tilde{C}'}: \tilde{C}' \to C$. Since $s = 2$, we have $\varphi^*C = \tilde{C}'$ so $\varphi|_{\tilde{C}'}$ is nothing but the base change of $\varphi$ along the closed immersion $C \to X$ and as such it is a flat morphism. But a morphism from ${\bf P}^1$ to the cuspidal cubic is never flat.

For the last statement (5), observe that $|C|$ defines a quasi-elliptic fibration if and only if $A\cdot C = 2$, and $|2C|$ defines a quasi-elliptic fibration if and only if $A\cdot C = 1$. This follows immediately from Lemma \ref{quellorell}, which implies that $A\cdot C = D\cdot C$ where $D$ is the curve of cusps of $|C|$ (resp. $|2C|$).
\end{proof}

\begin{remark}
Several of the numerical possibilities in Lemma \ref{blowups} might be excluded by using Lang's list of possible configurations of singular fibers on rational elliptic surfaces in characteristic $2$ \cite{L3} together with Lemma \ref{frobenius}. However, we will not pursue this here.
\end{remark}

\begin{lemma}\label{TwoSingsOnDoubleFiber}
Let $X$ be an Enriques surface with a quasi-elliptic fibration $\varphi$. Let $F$ be a fiber of $\varphi$. If $F$ is a double fiber, then two points on $F$ $($possibly including infinitely near ones$)$ are blown up during the minimal dissolution. If $F$ is simple, then no point on $F$ is blown up.
\end{lemma}

\begin{proof}
If $F$ is reducible, this can be read off from the table in \cite[p.13]{ES}, since every $(-2)$-curve on a simple fiber has $r$-invariant $0$ and exactly one $(-2)$-curve on a double fiber has $r$-invariant $2$ while the others have $r$-invariant $0$. If $F$ is irreducible, this is the last statement of Lemma \ref{blowups}.
\end{proof}

\begin{corollary}\label{SingNormal}
Let $X$ be an Enriques surface with a quasi-elliptic fibration. Then the normalization $\tilde{X}_{norm}$ of the canonical cover has an isolated $D_4$-singularity if and only if $X$ is supersingular.
\end{corollary}

\begin{proof}
Let $\varphi$ be a quasi-elliptic fibration on $X$. Since the conductrix is non-empty by Lemma \ref{quellorell}, $\tilde{X}$ is not normal. Therefore, $\tilde{X}_{norm}$ has either four $A_1$- or one $D_4$-singularity by Proposition \ref{ES}. If $\varphi$ has two double fibers, at least two distinct points on $X$ are blown up during the minimal dissolution by Lemma \ref{TwoSingsOnDoubleFiber}. In this case, $X$ is classical (Proposition \ref{multi-fiber}) and $\tilde{X}$ has four $A_1$-singularities. If $\varphi$ has only one double fiber, at most two distinct points on $X$ are blown up. In this case, $X$ is supersingular and $\tilde{X}$ has one $D_4$-singularity.
\end{proof}

\subsection{Special extremal genus one fibrations}

In this subsection, we present a detailed study of Enriques surfaces with special genus one fibrations, their conductrices and isolated singularities on their canonical cover. Throughout, we will use the observations summed up in the following Lemma.

\begin{lemma}\label{observations}
Let $X$ be an Enriques surface with conductrix $A$ and let $\tilde{X}$ be its canonical cover. The following hold.

\begin{enumerate}
\item[$(1)$] If two $(-2)$-curves which meet transversally have $s$-invariant $1$, then their intersection is blown up.
\item[$(2)$] Every $(-2)$-curve $C$ satisfies $C.A \leq 1$.
\item[$(3)$] Every $(-2)$-curve which is not a component of the support of 
the conductrix has $s$-invariant $1$. 
\end{enumerate}

Now let $\pi: X \to {\bf P}^1$ be a genus one pencil. Then the following hold.

\begin{enumerate}[label=$(\alph*)$]
\item A singular fiber of type ${\rm I}_n$ of $\pi$ gives $n$ $A_1$-singularities on $\tilde{X}$.
\item If $A \neq \emptyset$ and $\pi$ has a singular fiber of type ${\rm I}_n$, then $\tilde{X}$ has four $A_1$-singularities.
\item If $A \neq \emptyset$ and two disjoint $(-2)$-curves have positive $r$-invariant, then $\tilde{X}$ has four $A_1$-singularities.
\item If $A \neq \emptyset$ and the sum of all $r$-invariants of fiber components is less than $4$, then $\tilde{X}$ has one $D_4$-singularity.
\end{enumerate}
\end{lemma}

\begin{proof}
The first claim is obtained by checking intersection numbers, as was done by Ekedahl and Shepherd-Barron in \cite[Definition-Lemma 0.8 (iii)]{ES} and the second is immediate from Lemma \ref{blowups} (3). Since a curve $C$ which is not contained in $A$ has $A\cdot C \geq 0$, the third claim follows from Lemma \ref{blowups}.

For the statements about $\pi$: By Lemma \ref{frobenius}, the singularities lying over a simple fiber of $\pi$ can be read off from the Frobenius base change of the Jacobian fibration. Since an $\I_n$ fiber is simple and we can easily check (e.g. using Tate's algorithm) that the Frobenius base change of a fiber of type $\I_n$ acquires an $A_1$ singularity at every point where two curves in the fiber meet, we obtain Claim $(a)$.

Claim $(b)$ follows immediately from $(a)$, since $\tilde{X}$ has either four $A_1$-singularities or one $D_4$-singularity if $A \neq \emptyset$ (see Proposition \ref{ES}).
Two disjoint curves having positive $r$-invariant means that distinct points are blown up during the dissolution, excluding the possibility of a $D_4$-singularity on the cover. Hence, we obtain Claim $(c)$.
For Claim $(d)$, note that the sum of $r$-invariants of fiber components being less than $4$ means that less than $4$ distinct points are blown up, so the singularity can only be a $D_4$-singularity.
\end{proof}

\begin{remark}
Observe that we have used that the singularities lying over a simple fiber of $\pi$ are the same as the singularities of the Frobenius base change of the Jacobian fibration over the corresponding fiber. This follows from Lemma \ref{frobenius} since in a formal neighborhood of a simple fiber, an elliptic fibration is isomorphic to its Jacobian.
\end{remark}

\begin{lemma}
There are no special elliptic fibrations on Enriques surfaces with a double fiber of type $2{\rm III}^*,2{\rm II}^*$ or $2{\rm I}_4^*$. Moreover, if the conductrix is nonempty, a special elliptic fibration with a double fiber of type ${\rm IV}$ can not exist.
\end{lemma}

\begin{proof}
The statement about $\II^*,\III^*$ and $\I_4^*$ is contained in Ekedahl and Shepherd-Barron \cite[Corollary 3.2]{ES}. We will 
give another argument here. Let $N$ be a special $2$-section and let $C$ be a simple component of the double fiber we want to exclude. Assume that $C$ meets $N$. By checking all possible conductrices $A$ of \cite[Theorem 3.1]{ES}, we obtain that $C$ and $N$ have $s$-invariant $1$. Moreover, $A\cdot C = 0$ if $C$ is a component of $A$ with multiplicity $1$, whereas $A\cdot C = 1$ if $C$ does not occur in the conductrix. Therefore, $N\cdot A = 1$ if and only if $C\cdot A = 0$. Now by Lemma \ref{observations} $(1)$, the intersection of $N$ and $C$ is blown up. But one of them has $r$-invariant $0$ by Lemma \ref{blowups} $(3)$. This is a contradiction.

Now, we prove the second claim. Note that the fiber of type $\IV$ is disjoint from $A$ 
by \cite[Theorem 3.1]{ES}.
It follows from Lemma \ref{quellorell} $(2)$ that the 2-section $N$ is not contained in $A$.
Hence $N$ has $s$-invariant $1$ by Lemma \ref{observations} $(3)$ and every component of the fiber of type $\IV$ also has $s$-invariant $1$ by the same Lemma.  Therefore the intersection of $N$ and the fiber of type $\IV$ is blown up during the dissolution (Lemma \ref{observations}). Additionally, the intersection of the three components of the fiber of type $\IV$ is blown up. Therefore, the canonical cover has four $A_1$-singularities by Proposition \ref{ES} and all blow-ups happen on distinct point of the fiber of type $\IV$. However, since at least one point of $N$ is blown up, $N$ has positive $r$-invariant and, by Lemma \ref{quellorell}, $N.A \geq 0$. Therefore, by Lemma \ref{observations} $(3)$, $N$ has $r$-invariant $2$, i.e., two (possibly infinitely near) points on $N$ are blown up during the dissolution. Since $N$ meets the fiber of type $\IV$ only once, this is impossible.
\end{proof}

Recall that Ekedahl and Shepherd-Barron gave a list of possibilities for the restriction of the conductrix $A$ to the fiber of any elliptic fibration (\cite[Theorem 3.1]{ES}).  Their list contains not only the case $A^2=-2$ but also more general cases (for elliptic surfaces that are not necessarily Enriques surfaces).  Moreover, in the case where $A^2=-2$, there are several possibilities for $A$ for a fixed elliptic fibration.

\begin{remark}
A priori, \cite{ES} gives the restriction of $A$ to a fiber $F$ only up to multiples of $F$ (resp. up to multiples of the half-fiber underlying $F$ if $F$ is a double fiber). However, it is clear that 
the support of $A$ cannot contain $F$ since $h^0(X,\Omega_X) = 1$ and the explicit classification of ``exceptional Enriques surfaces" in \cite{ES} shows that also in the cases where $A$ contains a half-fiber, $A$ is in fact already contained in the tables of \cite{ES}. This implies that the tables in \cite{ES} do in fact give all possibilities for the restriction of $A$ to a fiber, not only up to multiples.
\end{remark}
In the following Lemma \ref{ellipticCondSing}, by using the list of \cite[Theorem 3.1]{ES} and the above observations, we give the list of
possible conductrices together with possible isolated singularities on the canonical cover for any extremal special elliptic fibration on an Enriques 
surface.  In our case, it will turn out that the conductrix is in fact uniquely determined by the type of singular fibers.
The following Tables \ref{ExtremalFibrationsConductrix} and \ref{QuasiellipticFibrationsConductrix} will enable us to set up a case-by-case analysis in the next section.

\begin{lemma}\label{ellipticCondSing}
The isolated singularities on the normalization of the canonical cover of an Enriques surface with a special extremal elliptic fibration and the conductrix are summed up in Table {\rm \ref{ExtremalFibrationsConductrix}}. The self-intersection number of the reduced inverse image of a curve on the minimal resolution of singularities of the canonical cover is given as an index to the multiplicity.

\begin{table}[!htbp]
\centering
\begin{tabular}{|>{\centering\arraybackslash}m{2.5cm}|>{\centering\arraybackslash}m{7cm}|>{\centering\arraybackslash}m{3.8cm}|}
\hline
\rm{Singular fibers} & \rm{Conductrix} & \rm{Isolated singularities}  \\ \hline
$({\rm I}_4^*)$& 
\resizebox{!}{1cm}{
\centerline{
\xy
(0,20)*{};
@={(-10,10),(0,10),(10,10),(20,10),(30,10)}@@{*{\bullet}};
(-10,10)*{};(30,10)*{}**\dir{-};
(-7,13)*{1_{-4}};
(3,13)*{1_{-2}};
(13,13)*{1_{-2}};
(23,13)*{1_{-2}};
(33,13)*{1_{-4}};
\endxy	
}
}		& $4A_1$ \\ \hline
$({\rm II}^*)$&
\resizebox{!}{1.5cm}{
\centerline{
\xy
(0,25)*{};
@={(-10,10),(0,10),(10,10),(20,10),(30,10),(40,10),(50,10),(10,20)}@@{*{\bullet}};
(-10,10)*{};(50,10)*{}**\dir{-};
(10,10)*{};(10,20)*{}**\dir{-};
(-7,13)*{1_{-2}};
(3,13)*{2_{-2}};
(14,13)*{3_{-4}};
(23,13)*{2_{-1}};
(33,13)*{2_{-4}};
(43,13)*{1_{-1}};
(53,13)*{1_{-4}};
(13,23)*{1_{-1}};
\endxy
}
}
& $D_4$ \\ \hline
$(2{\rm III},{\rm I}_8)$&
$\emptyset$
& $12A_1$ \\ \hline
$({\rm III},{\rm I}_8)$&
$\emptyset$
& $D_4, 8A_1$ \\ \hline
$(2{\rm I}_1^*,{\rm I}_4)$ &
\resizebox{!}{1.5cm}{
\centerline{
\xy
(0,25)*{};
@={(20,10),(0,10),(10,10),(10,20)}@@{*{\bullet}};
(0,10)*{};(20,10)*{}**\dir{-};
(10,10)*{};(10,20)*{}**\dir{-};
(23,13)*{1_{-4}};
(3,13)*{1_{-4}};
(14,13)*{1_{-1}};
(13,23)*{1_{-4}};
\endxy
}
}
& $4A_1$ \\ \hline
$({\rm I}_1^*,{\rm I}_4)$ &
\resizebox{7cm}{0.6cm}{
\centerline{
\xy
(13,15)*{};
@={(20,10),(10,10)}@@{*{\bullet}};
(10,10)*{};(20,10)*{}**\dir{-};
(23,13)*{1_{-4}};
(13,13)*{1_{-4}};
\endxy
}
}
& $4A_1$\\ \hline
$({\rm III}^*,{\rm I}_2)$ &
\resizebox{!}{1.5cm}{
\centerline{
\xy
(13,25)*{};
@={(0,10),(10,10),(20,10),(30,10),(40,10),(20,20)}@@{*{\bullet}};
(0,10)*{};(40,10)*{}**\dir{-};
(20,10)*{};(20,20)*{}**\dir{-};
(3,13)*{1_{-4}};
(13,13)*{1_{-1}};
(24,13)*{2_{-4}};
(33,13)*{1_{-1}};
(43,13)*{1_{-4}};
(23,23)*{1_{-2}};
\endxy
}
}
& $4A_1$ \\ \hline
$({\rm II}^*,{\rm I}_1)$ &
\resizebox{!}{1.5cm}{
\centerline{
\xy
(0,25)*{};
@={(-10,10),(0,10),(10,10),(20,10),(30,10),(40,10),(50,10),(10,20)}@@{*{\bullet}};
(-10,10)*{};(50,10)*{}**\dir{-};
(10,10)*{};(10,20)*{}**\dir{-};
(-7,13)*{1_{-2}};
(3,13)*{2_{-2}};
(14,13)*{3_{-4}};
(23,13)*{2_{-1}};
(33,13)*{2_{-4}};
(43,13)*{1_{-1}};
(53,13)*{1_{-4}};
(13,23)*{1_{-1}};
\endxy
}
}
& $4A_1$ \\ \hline
$({\rm IV},2{\rm IV}^*)$
&
\resizebox{!}{2cm}{
\centerline{
\xy
(13,35)*{};
@={(0,10),(10,10),(20,10),(30,10),(40,10),(20,20),(20,30)}@@{*{\bullet}};
(0,10)*{};(40,10)*{}**\dir{-};
(20,10)*{};(20,30)*{}**\dir{-};
(3,13)*{1_{-4}};
(13,13)*{1_{-1}};
(24,13)*{2_{-4}};
(33,13)*{1_{-1}};
(43,13)*{1_{-4}};
(24,23)*{1_{-1}};
(24,33)*{1_{-4}};
\endxy
}
}
& $D_4$ \\ \hline
$({\rm IV},{\rm IV}^*)$
&
\resizebox{!}{1.5cm}{
\centerline{
\xy
(13,25)*{};
@={(10,10),(20,10),(30,10),(20,20)}@@{*{\bullet}};
(10,10)*{};(30,10)*{}**\dir{-};
(20,10)*{};(20,20)*{}**\dir{-};
(13,13)*{1_{-4}};
(24,13)*{1_{-1}};
(33,13)*{1_{-4}};
(24,23)*{1_{-4}};
\endxy
}
}
& $D_4$ \\ 
\hline
$(2{\rm IV},{\rm I}_2,{\rm I}_6)$
&
$\emptyset$
& $12A_1$ \\ \hline
$({\rm IV},{\rm I}_2,{\rm I}_6)$
&
$\emptyset$
& $D_4,8A_1$ \\ \hline
$(2{\rm IV}^*,{\rm I}_1,{\rm I}_3)$
&
\resizebox{!}{2cm}{
\centerline{
\xy
(13,35)*{};
@={(0,10),(10,10),(20,10),(30,10),(40,10),(20,20),(20,30)}@@{*{\bullet}};
(0,10)*{};(40,10)*{}**\dir{-};
(20,10)*{};(20,30)*{}**\dir{-};
(3,13)*{1_{-4}};
(13,13)*{1_{-1}};
(24,13)*{2_{-4}};
(33,13)*{1_{-1}};
(43,13)*{1_{-4}};
(24,23)*{1_{-1}};
(24,33)*{1_{-4}};
\endxy
}
}
& $4A_1$ \\ \hline
$({\rm IV}^*,{\rm I}_1,{\rm I}_3)$
&
\resizebox{!}{1.5cm}{
\centerline{
\xy
(13,25)*{};
@={(10,10),(20,10),(30,10),(20,20)}@@{*{\bullet}};
(10,10)*{};(30,10)*{}**\dir{-};
(20,10)*{};(20,20)*{}**\dir{-};
(13,13)*{1_{-4}};
(24,13)*{1_{-1}};
(33,13)*{1_{-4}};
(24,23)*{1_{-4}};
\endxy
}
}
& $4A_1$ \\ \hline
$({\rm I}_9,{\rm I}_1,{\rm I}_1,{\rm I}_1)$
&
$\emptyset$
& $12A_1$ \\ \hline
$({\rm I}_5,{\rm I}_5,{\rm I}_1,{\rm I}_1)$
&
$\emptyset$
& $12A_1$ \\ \hline
$({\rm I}_3,{\rm I}_3,{\rm I}_3,{\rm I}_3)$
&
$\emptyset$
& $12A_1$ \\ \hline
\end{tabular}
\caption{Singularities on the canonical cover of an Enriques surface with a special and extremal elliptic fibration} \label{ExtremalFibrationsConductrix}
\end{table}
\end{lemma}
\begin{proof}
For the list of rational extremal elliptic fibrations see Proposition \ref{Lang}. Since the conductrix $A$ is contained in one fiber, the tables in \cite[pp.16-18]{ES} will give us the possibilities for $A$. In every case, we denote the special $2$-section by $N$. Recall that $A^2 = -2$ by Proposition \ref{ES}.
\begin{itemize}
\item $(\I_4^*):$ There is only one possibility for $A$ with $A^2 = -2$ in the list of \cite{ES}. The canonical cover has four $A_1$-singularities by Lemma \ref{observations} $(c)$.
\item $(\II^*):$ There are two possible conductrices with $A^2 = -2$ in the list of \cite{ES}. However, since $N\cdot A \leq 1$ by Lemma \ref{observations} $(2)$, we get the one in the table. Since all fibers different from the fiber of type $\II^*$ are smooth and no point on a smooth fiber is blown up during the dissolution by Lemma \ref{blowups}, the sum of all $r$-invariants of fibers is less than $4$. Hence the cover has one $D_4$-singularity by Lemma \ref{observations} $(d)$.
\item $(2\III,\I_8):$ In this case $A = \emptyset$. Thus, any $(-2)$-curve has $s$-invariant $1$ (Lemma \ref{blowups}, (3)).  Since the intersection of $N$ with a component of the 
fiber of type $\III$ is blown up by Lemma \ref{observations}, (1), there are at least $11$ distinct points which are blown up during the dissolution by Lemma \ref{observations} $(a)$. Therefore, the cover has $12$ $A_1$-singularities.
\item $(\III,\I_8):$ Again, we have $A = \emptyset$. By \cite{L3}, the fiber of type $\III$ acquires a $D_4$-singularity after Frobenius pullback. The $8$ $A_1$-singularities come from the fiber of type $\I_8$ by Lemma \ref{observations} $(a)$.
\item $(2\I_1^*,\I_4):$ In this case and the next, there are two possibilities for $A$ with $A^2 = -2$ in the list of \cite{ES}.
By Lemma \ref{observations} $(b)$, we have $4$ $A_1$-singularities. Since $N$ is a 2-section and every point which is blown up lies on the fiber of type $\I_4$, the $r$-invariant of $N$ is at most $1$ and therefore $N\cdot A = 1$ (Lemma \ref{blowups}, (3)). This is only possible for the conductrix in our table.
\item $(\I_1^*,\I_4):$ By the same argument as in the previous case, we have $N\cdot A = 1$. Moreover, $N$ can not meet distinct components of the fiber of type $\I_1^*$ since we would obtain a different fibration with a double fiber of type $\I_4$ or $\I_5$ in these cases. Therefore, $N$ meets a multiplicity $2$ component of the fiber of type $\I_1^*$ . Now, $N$ and some components of the fiber of type $\I_1^*$ form a fiber of type $\I_0^*$ of a different fibration and the only possible conductrix for this behaviour is the one in our table.
\item $(\III^*,\I_2):$ There are two possible conductrices with $A^2 = -2$ in the list of \cite{ES}. If the conductrix has the full fiber as support, $N$ meets the central multiplicity $2$ component since $N\cdot A \leq 1$ by Lemma \ref{observations} $(2)$. But then, there is a fiber of type $\IV^*$ of a different fibration such that two components of the conductrix meet the fiber without being contained in it. This is not possible
by Lemma \ref{quellorell}. Hence, we have the conductrix in our table and the isolated singularities because of Lemma \ref{observations} $(b)$.
\item $(\II^*,\I_1):$ The conductrix is the one in the table by the same argument as in the $(\II^*)$ case. By Lemma \ref{observations} $(b)$, we get the types of isolated singularities.
\item $(\IV,2\IV^*):$ In this case and the next, there are two possibilities for $A$ with $A^2 = -2$ in the list of \cite{ES}.
Since $N$ meets a simple component of the fiber of type $\IV^*$, we can exclude the case where the conductrix does not have the full fiber as support, since in this case every simple component of the fiber of type $\IV^*$ has $s$-invariant $1$ and $r$-invariant $0$ while $N$ has $s$-invariant $1$ by
$N\cdot A=0$, contradicting Lemma \ref{observations} $(1)$. The isolated singularities are as in the table, since by \cite{L3} the fibers of type $\IV$ acquires a $D_4$-singularity after Frobenius pullback.
\item $(\IV,\IV^*):$ Suppose that $A$ has the full fiber of type $\IV^*$ as support. Then $N$ meets a multiplicity $2$ component of this fiber, since $A\cdot N \leq 1$. But then $N$ and components of the fiber of type $\IV^*$ form a 
fiber of type $\I_1^*$ of a different elliptic fibration such that two components of the conductrix meet the fiber without being contained in it. This is not possible by Lemma \ref{quellorell}. As in the previous case, we get a $D_4$-singularity.
\item $(2\IV,\I_2,\I_6)$ and $(\IV,\I_2,\I_6)$: The argument is essentially the same as in the $(2\III,\I_8)$ and $(\III,\I_8)$ cases.
\item $(2\IV^*,\I_1,\I_3)$ and $(\IV^*,\I_1,\I_3)$: The argument is similar to the cases with singular fibers $(\IV,2\IV^*)$ and  
$(\IV,\IV^*)$, except that the fibers of type $\I_n$ give $4$ $A_1$-singularities by Lemma \ref{observations} $(a)$.
\item All singular fibers multiplicative: In these cases, we get $12$ $A_1$-singularities by Lemma \ref{observations} $(a)$.
\end{itemize}
\vspace*{-\baselineskip}\leavevmode
\end{proof}

For the convenience of the reader, we give the corresponding table for quasi-elliptic fibrations. This does not require proof, since the conductrices are uniquely determined (see \cite{ES}) and the isolated singularities depend on the number of double fibers (see Lemma \ref{TwoSingsOnDoubleFiber}). Since fibers of type $\III$ do not contribute to the conductrix, we will not specify their multiplicity. Also, recall that by Lemma \ref{quellorell} the curve of cusps is the only component of the conductrix which is not contained in fibers of the fibration.

\begin{lemma}\label{QellipticCondSing}
The isolated singularities on the normalization of the canonical cover of an Enriques surface with a quasi-elliptic fibration and the conductrix are summed up in table {\rm \ref{QuasiellipticFibrationsConductrix}}. The self-intersection number of the reduced inverse image of the curve on the minimal resolution of singularities of the canonical cover is given as an index to the multiplicity. We do not give multiplicities of the fibers of type ${\rm III}$. The curve of cusps is encircled.

\begin{table}[!htbp]
\centering
\begin{tabular}{|>{\centering\arraybackslash}m{2.5cm}|>{\centering\arraybackslash}m{7cm}|>{\centering\arraybackslash}m{3.8cm}|}
\hline
\rm{Singular fibers} & \rm{Conductrix} & \rm{Isolated singularities}  \\ \hline
$(2{\rm II}^*)$&
\resizebox{!}{1.3cm}{
\centerline{
\xy
(0,25)*{};
@={(-10,10),(0,10),(10,10),(20,10),(30,10),(40,10),(50,10),(10,20),(60,10),(70,10)}@@{*{\bullet}};
(-10,10)*{};(70,10)*{}**\dir{-};
(10,10)*{};(10,20)*{}**\dir{-};
(69.95,10.19) *++={.} *\frm<4.5pt>{o};
(-7,13)*{2_{-4}};
(3,13)*{3_{-1}};
(14,13)*{5_{-4}};
(23,13)*{4_{-1}};
(33,13)*{4_{-4}};
(43,13)*{3_{-1}};
(53,13)*{3_{-4}};
(63,13)*{2_{-2}};
(73,13)*{1_{-2}};
(13,23)*{2_{-1}};
\endxy
}
}	& $4A_1$ or $D_4$\\ \hline
$({\rm II}^*)$&
\resizebox{!}{1.5cm}{
\centerline{
\xy
(0,25)*{};
@={(-10,10),(0,10),(10,10),(20,10),(30,10),(40,10),(50,10),(20,20),(60,10)}@@{*{\bullet}};
(-10,10)*{};(60,10)*{}**\dir{-};
(20,10)*{};(20,20)*{}**\dir{-};
(-10.05,10.19) *++={.} *\frm<4.5pt>{o};
(-7,13)*{1_{-2}};
(3,13)*{2_{-4}};
(13,13)*{2_{-1}};
(24,13)*{3_{-4}};
(33,13)*{2_{-1}};
(43,13)*{2_{-4}};
(53,13)*{1_{-1}};
(63,13)*{1_{-4}};
(24,23)*{1_{-1}};
\endxy
}
}	& $4A_1$ or $D_4$ \\ \hline
$(2{\rm I}_4^*)$& 
\resizebox{!}{1.5cm}{
\centerline{
\xy
(0,25)*{};
@={(-10,10),(0,10),(10,10),(20,10),(30,10),(40,10),(50,10),(10,20)}@@{*{\bullet}};
(-10,10)*{};(50,10)*{}**\dir{-};
(10,10)*{};(10,20)*{}**\dir{-};
(-7,13)*{1_{-2}};
(3,13)*{2_{-2}};
(14,13)*{3_{-4}};
(23,13)*{2_{-1}};
(33,13)*{2_{-4}};
(43,13)*{1_{-1}};
(53,13)*{1_{-4}};
(13,23)*{1_{-1}};
(-10.05,10.19) *++={.} *\frm<4.5pt>{o};
\endxy
}
}	& $4A_1$ or $D_4$ \\ \hline
$({\rm I}_4^*)$& 
\resizebox{!}{1.5cm}{
\centerline{
\xy
(13,25)*{};
@={(0,10),(10,10),(20,10),(30,10),(40,10),(20,20)}@@{*{\bullet}};
(0,10)*{};(40,10)*{}**\dir{-};
(20,10)*{};(20,20)*{}**\dir{-};
(3,13)*{1_{-4}};
(13,13)*{1_{-1}};
(24,13)*{2_{-4}};
(33,13)*{1_{-1}};
(43,13)*{1_{-4}};
(23,23)*{1_{-2}};
(19.95,20.19) *++={.} *\frm<4.5pt>{o};
\endxy
}
}	& $4A_1$ or $D_4$ \\ \hline
$(2{\rm III}^*,{\rm III})$&
\resizebox{!}{1.5cm}{
\centerline{
\xy
(0,25)*{};
@={(-10,10),(0,10),(10,10),(20,10),(30,10),(40,10),(50,10),(20,20),(60,10)}@@{*{\bullet}};
(-10,10)*{};(60,10)*{}**\dir{-};
(20,10)*{};(20,20)*{}**\dir{-};
(-7,13)*{1_{-2}};
(3,13)*{2_{-4}};
(13,13)*{2_{-1}};
(24,13)*{3_{-4}};
(33,13)*{2_{-1}};
(43,13)*{2_{-4}};
(53,13)*{1_{-1}};
(63,13)*{1_{-4}};
(24,23)*{1_{-1}};
(59.95,10.19) *++={.} *\frm<4.5pt>{o};
\endxy
}
}	& $4A_1$ or $D_4$ \\ \hline
$({\rm III}^*,{\rm III})$&
\resizebox{!}{2cm}{
\centerline{
\xy
(13,35)*{};
@={(0,10),(10,10),(20,10),(30,10),(40,10),(20,20),(20,30)}@@{*{\bullet}};
(0,10)*{};(40,10)*{}**\dir{-};
(20,10)*{};(20,30)*{}**\dir{-};
(3,13)*{1_{-4}};
(13,13)*{1_{-1}};
(24,13)*{2_{-4}};
(33,13)*{1_{-1}};
(43,13)*{1_{-4}};
(24,23)*{1_{-1}};
(24,33)*{1_{-4}};
(19.95,30.19) *++={.} *\frm<4.5pt>{o};
\endxy
}
}
& $4A_1$ or $D_4$ \\ \hline
$(2{\rm I}_0^*,2{\rm I}_0^*)$&
\resizebox{!}{1cm}{
\centerline{
\xy
(0,20)*{};
@={(-10,10),(0,10),(10,10),(20,10),(30,10)}@@{*{\bullet}};
(-10,10)*{};(30,10)*{}**\dir{-};
(-7,13)*{1_{-4}};
(3,13)*{1_{-2}};
(13,13)*{1_{-2}};
(23,13)*{1_{-2}};
(33,13)*{1_{-4}};
(9.95,10.19) *++={.} *\frm<4.5pt>{o};
\endxy	
}
}		& $4A_1$ \\ \hline
$(2{\rm I}_0^*,{\rm I}_0^*)$&
\resizebox{!}{1cm}{
\centerline{
\xy
(0,20)*{};
@={(-10,10),(0,10),(10,10),(20,10)}@@{*{\bullet}};
(-10,10)*{};(20,10)*{}**\dir{-};
(-7,13)*{1_{-4}};
(3,13)*{1_{-2}};
(13,13)*{1_{-2}};
(23,13)*{1_{-4}};
(-0.05,10.19) *++={.} *\frm<4.5pt>{o};
\endxy	
}
}		& $4A_1$ or $D_4$  \\ \hline
$({\rm I}_0^*,{\rm I}_0^*)$&
\resizebox{!}{1cm}{
\centerline{
\xy
(0,20)*{};
@={(-10,10),(0,10),(10,10)}@@{*{\bullet}};
(-10,10)*{};(10,10)*{}**\dir{-};
(-7,13)*{1_{-4}};
(3,13)*{1_{-2}};
(13,13)*{1_{-4}};
(-0.05,10.19) *++={.} *\frm<4.5pt>{o};
\endxy	
}
}		& $4A_1$ or $D_4$  \\ \hline
$(2{\rm I}_2^*,{\rm III},{\rm III})$ &
\resizebox{!}{1.5cm}{
\centerline{
\xy
(13,25)*{};
@={(0,10),(10,10),(20,10),(30,10),(40,10),(20,20)}@@{*{\bullet}};
(0,10)*{};(40,10)*{}**\dir{-};
(20,10)*{};(20,20)*{}**\dir{-};
(3,13)*{1_{-4}};
(13,13)*{1_{-1}};
(24,13)*{2_{-4}};
(33,13)*{1_{-1}};
(43,13)*{1_{-4}};
(23,23)*{1_{-2}};
(-0.05,10.19) *++={.} *\frm<4.5pt>{o};
\endxy
}
}	& $4A_1$ or $D_4$ \\ \hline
$({\rm I}_2^*,{\rm III},{\rm III})$ &
\resizebox{!}{1.5cm}{
\centerline{
\xy
(13,25)*{};
@={(10,10),(20,10),(30,10),(20,20)}@@{*{\bullet}};
(10,10)*{};(30,10)*{}**\dir{-};
(20,10)*{};(20,20)*{}**\dir{-};
(13,13)*{1_{-4}};
(24,13)*{1_{-1}};
(33,13)*{1_{-4}};
(24,23)*{1_{-4}};
(19.95,20.19) *++={.} *\frm<4.5pt>{o};
\endxy
}
}
& $4A_1$ or $D_4$ \\ \hline
$(2{\rm I}_0^*, 4 \times {\rm III})$ &
\resizebox{!}{1cm}{
\centerline{
\xy
(0,20)*{};
@={(-10,10),(0,10),(10,10)}@@{*{\bullet}};
(-10,10)*{};(10,10)*{}**\dir{-};
(-7,13)*{1_{-4}};
(3,13)*{1_{-2}};
(13,13)*{1_{-4}};
(-10.05,10.19) *++={.} *\frm<4.5pt>{o};
\endxy	
}
}		& $4A_1$ or $D_4$  \\ \hline
$({\rm I}_0^*,4 \times {\rm III})$ &
\resizebox{!}{1cm}{
\centerline{
\xy
(0,20)*{};
@={(-10,10),(0,10)}@@{*{\bullet}};
(-10,10)*{};(0,10)*{}**\dir{-};
(-7,13)*{1_{-4}};
(3,13)*{1_{-4}};;
(-10.05,10.19) *++={.} *\frm<4.5pt>{o};
\endxy	
}
}		& $4A_1$ or $D_4$  \\ \hline
$(8 \times {\rm III})$ &
\resizebox{!}{1cm}{
\centerline{
\xy
(0,20)*{};
@={(-10,10)}@@{*{\bullet}};
(-7,13)*{1_{-6}};
(-10.05,10.19) *++={.} *\frm<4.5pt>{o};
\endxy	
}
}		& $4A_1$ or $D_4$  \\ \hline
\end{tabular}
\caption{Singularities on the canonical cover of an Enriques surface with a quasi-elliptic fibration} \label{QuasiellipticFibrationsConductrix}
\end{table}
\end{lemma}

\begin{remark}
Recall that any Enriques surface has a genus one fibration (Proposition \ref{genus1}) and if an Enriques surface $X$ has a finite group of automorphisms, then any genus one fibration on
$X$ is extremal (Proposition \ref{MW-Dolgachev}). Therefore, $X$ has an extremal special genus one fibration by Proposition \ref{Cossec}. Lemmas {\rm \ref{ellipticCondSing} and \ref{QellipticCondSing}} imply that the canonical cover of any Enriques surface with finite automorphism group has only $A_1$- or $D_4$-singularities as isolated singularities. In particular, this excludes the exotic case where the cover has an elliptic singularity and we refer the reader to \cite{S}, \cite{Mat} for an overview of the types of singularities that can occur in general.
\end{remark}

\section{Possible dual graphs}\label{possibledualgraph}


\begin{theorem}\label{mainGraph}
Assume that $X$ is a classical or supersingular Enriques surface in characteristic $2$ with a finite group of automorphisms.
Then, the dual graph of $(-2)$-curves
on $X$ is one of the dual graphs given in Theorems {\rm \ref{mainSupersingular} (A)} and
{\rm \ref{mainClassical} (A)}.
\end{theorem}
\begin{proof}
Recall that there exists a special genus one fibration on $X$ (Proposition \ref{Cossec}) which 
is extremal (Proposition \ref{MW-Dolgachev}). Hence, the pair $(A,I)$, where $A$ is the conductrix of $X$ and $I$ is the set of isolated singularities of the normalization of the canonical cover of $X$, is contained in Table \ref{ExtremalFibrationsConductrix} or \ref{QuasiellipticFibrationsConductrix}.

Using Lemmas \ref{ellipticCondSing} and \ref{QellipticCondSing}, and the fact that every genus one fibration on $X$ is extremal by Proposition \ref{MW-Dolgachev}, we obtain a list of all possible special genus one fibration on $X$ depending on $(A,I)$. On the other hand, by Proposition \ref{genus1}, every configuration of $(-2)$-curves whose dual graph is an extended Dynkin diagram is in fact a fiber of a genus one fibration of $X$. These two observations together will allow us to control the $(-2)$-curves on $X$. 
It remains to go through the possibilities for $(A,I)$ and, starting with one of the possible special extremal fibrations, to show that either we arrive at a contradiction or $X$ contains one of the dual graphs of $(-2)$-curves in our list. In the latter case, Vinberg's criterion (Propositions \ref{Vinberg}, \ref{Namikawa} and Remark \ref{Vinbergremark}) shows immediately that these are in fact all $(-2)$-curves on $X$. For the fibrations that actually occur on Enriques surfaces with finite automorphism group, we refer the
reader to the list of the genus one fibrations in the Appendix, Section \ref{AppendixGenus1fibration}.

We will make extensive use of Lemma \ref{quellorell} together with the Tables \ref{ExtremalFibrationsConductrix}, \ref{QuasiellipticFibrationsConductrix}, \cite[p.13]{ES} and \cite[pp.16-18]{ES}, which tell us how reducible fibers of genus one fibrations can meet $A$. We will give details on how to use these results in the first few cases, so that the reader can get familiar with the techniques and fill in the details of the later cases in a similar way.
Also, we denote by $N$ a special $(-2)$-section for a given special genus one fibration.  If the fibration is quasi-elliptic, then $N$ denotes the curve of cusps.

\noindent \hspace{-7mm} (1)\ Conductrix:
\xy
(0,25)*{};
@={(-10,10),(0,10),(10,10),(20,10),(30,10),(40,10),(50,10),(60,10),(70,10),(10,20)}@@{*{\bullet}};
(-10,10)*{};(70,10)*{}**\dir{-};
(10,10)*{};(10,20)*{}**\dir{-};
(-10,13)*{2};
(0,13)*{3};
(13,13)*{5};
(20,13)*{4};
(30,13)*{4};
(40,13)*{3};
(50,13)*{3};
(60,13)*{2};
(70,13)*{1};
(10,23)*{2};
\endxy  \quad
Singularities: $D_4$ or $4A_1$  \\
Possible special extremal fibrations: $(2\II^*)$ quasi-elliptic
\\

This is nothing but the dual graph of $(-2)$-curves of Enriques surfaces of type $\tilde{E}_8$.  The Enriques surfaces are supersingular or 
classical according to the type of singularities (Corollary \ref{SingNormal}).
These are the $\tilde{E}_8$ exceptional surfaces studied in \cite{ES}.


\noindent \hspace{-7mm}  (2)\ Conductrix:
\xy
(0,27)*{};
@={(-20,10),(-10,10),(0,10),(10,10),(20,10),(30,10),(40,10),(50,10),(10,20)}@@{*{\bullet}};
(-20,10)*{};(50,10)*{}**\dir{-};
(10,10)*{};(10,20)*{}**\dir{-};
(-10,13)*{2};
(0,13)*{2};
(13,13)*{3};
(20,13)*{2};
(30,13)*{2};
(40,13)*{1};
(50,13)*{1};
(-20,13)*{1};
(10,23)*{1};
\endxy \quad
Singularities: $D_4$ or $4A_1$ \\
Possible special extremal fibrations: $(\II^*)$ quasi-elliptic, $(2\III^*,\III)$ quasi-elliptic,
 and \\ $(2\III^*,2\III)$ quasi-elliptic
\\

In this case, we will show that $X$ has the dual graph of an Enriques surface of type $\tilde{E_7}+\tilde{A_1}^{(1)}$ or $\tilde{E_7}+\tilde{A_1}^{(2)}$.

First, note that if $X$ admits a quasi-elliptic fibration of type $(2\III^*,\III)$, the support of $A$ is exactly the support of the $\III^*$ fiber and the curve of cusps $N$ by Table \ref{QuasiellipticFibrationsConductrix} and Lemma \ref{quellorell}. More precisely, the right-most vertex of $A$ corresponds to $N$. The $2$-section $N$ meets each component of the singular fiber of type $\III$ because otherwise there would be a $(-2)$-curve in the fiber of type $\III$ meeting $N$ (the conductrix) more than once, contradicting Lemma \ref{observations} (2).
Now, fiber components and the special $2$-section $N$ of the fibration form the dual graph of type 
$\tilde{E}_7 + \tilde{A}_1^{(1)}$.  

If $X$ admits a quasi-elliptic fibration of type $(\II^*)$ or
of type $(2\III^*,2\III)$, then we immediately get the dual graph of type $\tilde{E}_7 + \tilde{A}_1^{(2)}$.
These are the $\tilde{E}_7$ exceptional surfaces of \cite{ES}. 


\noindent \hspace{-7mm}  (3)\ Conductrix:
\xy
(0,27)*{};
@={(-10,10),(0,10),(10,10),(20,10),(30,10),(40,10),(50,10),(10,20)}@@{*{\bullet}};
(-10,10)*{};(50,10)*{}**\dir{-};
(10,10)*{};(10,20)*{}**\dir{-};
(-10,13)*{1};
(0,13)*{2};
(13,13)*{3};
(20,13)*{2};
(30,13)*{2};
(40,13)*{1};
(50,13)*{1};
(10,23)*{1};
\endxy \quad
Singularities: $D_4$ or $4A_1$ \\
Possible special extremal fibrations: $(2\I_4^*)$ quasi-elliptic, $(\II^*)$ elliptic and $(\II^*,\I_1)$ elliptic
\\

In this case, we will show that $X$ has the dual graph of an Enriques surface of type $\tilde{D_8}$.
Starting from the quasi-elliptic fibration of type $(2\I_4^*)$, we immediately obtain the dual graph of Enriques surfaces of type $\tilde{D}_8$.

If we start with a special elliptic fibration with a singular fiber of type $\II^*$, the $2$-section $N$ has to meet this fiber in a component with multiplicity $2$.  Indeed, otherwise $N$ would meet the simple component $E_1$ of the fiber of type $\II^*$ twice and by Table \ref{ExtremalFibrationsConductrix} and Lemma \ref{quellorell}, $E_1$ is not contained in $A$ while the adjacent component $E_2$ of the $\II^*$ fiber is contained in $A$. By Proposition \ref{genus1}, and since $(E_1+N).E_2 = 1$ prevents $|E_1+N|$ from inducing a fibration, $|2(E_1+N)|$ would induce a special genus one fibration with special $2$-section $E_2$ and $E_1$ and $N$ form a fiber of type $2\III$. However, such a fibration cannot occur on $X$ by the above list of possible special extremal fibrations.

Hence, we either get a special genus one fibration with a double fiber of type $\III^*$ (if $N$ meets the left-most vertex of $A$. Then, the right-most vertex is a $2$-section) or a special genus one fibration with a double fiber of type $\I_4^*$ (if $N$ meets the right-most vertex of $A$. Then, the left-most vertex is a $2$-section). The first case is not allowed (it does not appear in the above list of possible special extremal fibrations). Therefore, we arrive at the dual graph of Enriques surfaces of type $\tilde{D}_8$.


\noindent \hspace{-7mm}  (4)\ Conductrix:
\xy
(13,35)*{};
@={(0,10),(10,10),(20,10),(30,10),(40,10),(20,20),(20,30)}@@{*{\bullet}};
(0,10)*{};(40,10)*{}**\dir{-};
(20,10)*{};(20,30)*{}**\dir{-};
(0,13)*{1};
(10,13)*{1};
(23,13)*{2};
(30,13)*{1};
(40,13)*{1};
(23,23)*{1};
(23,33)*{1};
\endxy \quad
Singularities: $D_4$ or $4A_1$ \\
Possible special extremal fibrations: $(\III^*,2\III)$ quasi-elliptic, $(\III^*,\III)$ quasi-elliptic,\\  
$(2\IV^*,\IV)$ elliptic and $(2\IV^*,\I_3,\I_1)$ elliptic
\\

In this case, we will show that $X$ has the dual graph of an Enriques surface of type $\tilde{E_6}+\tilde{A}_2$.

If we start with $(2\IV^*,\IV)$ or 
$(2\IV^*,\I_3,\I_1)$, then $N$ meets a component $E$ of the $\IV$ or $\I_3$ fiber twice. Indeed, otherwise we would find a special genus one fibration with a double fiber of type $\IV$ or of type $\I_3$ consisting of $N$ and two components of the fiber of type $\IV$ or $\I_3$, which is not allowed in the above list of special fibrations.  Then, using our list of possible fibrations, we see that $|2(E+N)|$ induces a quasi-elliptic fibration of type $(\III^*,2\III)$. We add the two simple components of the $\III^*$ fiber to the above diagram and compute the intersection numbers by repeating the above argument that a special $2$-section of the fibration of type $(2\IV^*,\IV)$ or 
$(2\IV^*,\I_3,\I_1)$ meets a component of the reducible simple fiber twice. Finally, we arrive at the dual graph of Enriques surfaces of type $\tilde{E}_6 + \tilde{A}_2$.
This is an $\tilde{E}_6$ exceptional Enriques surface of \cite{ES}.

If $X$ has a quasi-elliptic fibration with a singular fiber of type $\III^*$, 
we will show that $X$ admits a special elliptic fibration with a singular fiber of type $2\IV^*$, returning us to the argument of the previous paragraph.
In this case,
Lemma \ref{quellorell} and Table \ref{QuasiellipticFibrationsConductrix} show that the curve of cusps $N$ meets the component on the short tail of the $\III^*$ fiber. Thus, $N$ and components of the $\III^*$ fiber form a configuration of type $\tilde{E}_6$. By Proposition \ref{genus1}, and since the simple components of the original fiber of type $\III^*$ are $2$-sections, this configuration is the support of a fiber of a special genus one fibration on $X$. Using our list, we see that it is of type $(2\IV^*,\IV)$ or 
$(2\IV^*,\I_3,\I_1)$. 



%
%


\noindent \hspace{-7mm}  (5)\ Conductrix:
\xy
(13,25)*{};
@={(0,10),(10,10),(20,10),(30,10),(40,10),(20,20)}@@{*{\bullet}};
(0,10)*{};(40,10)*{}**\dir{-};
(20,10)*{};(20,20)*{}**\dir{-};
(0,13)*{1};
(10,13)*{1};
(23,13)*{2};
(30,13)*{1};
(40,13)*{1};
(20,23)*{1};
\endxy \quad
Singularities: $D_4$ \\
Possible special extremal fibrations: $(\I_4^*)$ quasi-elliptic, $(2\I_2^*,\III,\III)$ quasi-elliptic and $(2\I_2^*,2\III,\III)$ quasi-elliptic
\\

We will show that this case does not occur on an Enriques surface with finite automorphism group.

In cases $(2\I_2^*,\III,\III)$ and $(2\I_2^*,2\III,\III)$, the curve of cusps $N$ corresponds to the left-most (or right-most) vertex of $A$ by Table \ref{QuasiellipticFibrationsConductrix}. Therefore, $N$ meets both components of a simple $\III$ fiber once by Lemma \ref{observations} (2), hence there exists a genus one fibration with a fiber of type $\III^*$ (here we use Proposition \ref{genus1}). Since the fiber of type $\III^*$ contains the conductrix, the induced fibration is elliptic by Lemma \ref{quellorell} and hence of type $(\III^*,\I_2)$ by Proposition \ref{Lang}. This contradicts the type of singularities (Lemma \ref{observations} (b)).

If we start with $(\I_4^*)$, then we find a special fibration with a double fiber of type $\I_2^*$ (again, use Table \ref{QuasiellipticFibrationsConductrix}, Lemma \ref{quellorell} and a component of $\I_4^*$ as $2$-section).  Thus, we reduce this case to the case of a quasi-elliptic fibration with a singular fiber of type $2\I_2^*$, which we have excluded above.



\noindent \hspace{-7mm} (6)\ Conductrix:
\xy
(13,25)*{};
@={(0,10),(10,10),(20,10),(30,10),(40,10),(20,20)}@@{*{\bullet}};
(0,10)*{};(40,10)*{}**\dir{-};
(20,10)*{};(20,20)*{}**\dir{-};
(0,13)*{1};
(10,13)*{1};
(23,13)*{2};
(30,13)*{1};
(40,13)*{1};
(20,23)*{1};
\endxy \quad
Singularities: $4A_1$ \\
Possible special extremal fibrations: $(2\I_2^*,2\III,\III)$ quasi-elliptic, $(2\I_2^*,\III,\III)$ quasi-elliptic, $(\I_4^*)$ quasi-elliptic and $(\III^*,\I_2)$ elliptic
\\

We will prove that this case does not occur on an Enriques surface with finite automorphism group.

First we show that in every case, there is a quasi-elliptic fibration with a singular fiber of type $\I_4^*$ and with the curve of cusps meeting the central component. 

In the cases with a double fiber of type $\I_2^*$, we may assume that the curve of cusps is the left most vertex of $A$ by Table \ref{QuasiellipticFibrationsConductrix}. 
We observe that the curve of cusps can not meet a component of a simple fiber of type $\III$ twice, because of Lemma \ref{observations} $(2)$. Hence, we obtain a quasi-elliptic fibration with a singular fiber of type $\I_4^*$.

In the case of the special elliptic fibration with singular fibers of type $(\III^*,\I_2)$, note that if the $2$-section meets a simple component of the fiber of type $\III^*$ twice or two simple components, we get a quasi-elliptic fibration with a singular fiber of type $2\III$ or $2\I_8$.  The latter case is impossible.  Thus, we get
a quasi-elliptic fibration with singular fibers of type $(2\I_2^*, 2\III, \III)$ by the above list of
possible special extremal fibrations, and hence reduce this case to the previous case.
If the 2-section meets a component of multiplicity $2$ on one of the long tails, we get a quasi-elliptic fibration with a singular fiber of type $2\I_2^*$ and if it meets the component of multiplicity $2$ on the short tail, there would be a special elliptic fibration with a double fiber of type $\IV^*$.  The last case does not occur by the list of possible special extremal fibrations.

We now start from a quasi-elliptic fibration with a singular fiber of type $\I_4^*$ and exclude this case. By Table \ref{QuasiellipticFibrationsConductrix}, the curve of cusps $N$ 
(which is denoted by the encircled vertex) 
meets the central component of the $\I_4^*$ fiber.
Two of the blown up points lie on the conductrix and two do not (by Table \ref{QuasiellipticFibrationsConductrix} and Lemma \ref{blowups}, $(3)$). Any $(-2)$-curve not meeting the conductrix has $r$-invariant $2$ and therefore it passes through the $2$ blown up points not lying on the conductrix. In particular, any two $(-2)$-curves not meeting the conductrix meet each other at least twice.

The configuration we start with is the following:

\xy
(-40,25)*{};
@={(10,0),(50,0),(10,10),(20,10),(30,10),(40,10),(50,10),(10,20),(50,20),(30,20)}@@{*{\bullet}};
(10,10)*{};(50,10)*{}**\dir{-};
(10,0)*{};(10,20)*{}**\dir{-};
(30,10)*{};(30,20)*{}**\dir{-};
(50,0)*{};(50,20)*{}**\dir{-};
\endxy 
\vspace{2mm}

There are four subdiagrams of type $\tilde{E}_7$. If the automorphism group of an Enriques surface with this conductrix is finite, the elliptic fibrations induced by those subdiagrams have singular fibers of type $(\III^*,\I_2)$ or $(2\III^*,\I_2)$. For any of these diagrams of type $\tilde{E}_7$, the two remaining curves are either $4$- or $2$-sections of the fibration, depending on whether the fiber of type $\III^*$ is double or not
(we do not know whether the fibration is special or not if the $\III^*$ fiber is a double fiber and hence it might not be in our list). 
If such a multisection meets a component of the fiber of type $\I_2$ only once, we obtain a quasi-elliptic fibration with singular fiber of type $\II^*$ (since the conductrix $A$ is not contained in the fiber of type $\II^*$, the fibration is quasi-elliptic (Lemma \ref{quellorell}, (2)) and hence special (Remark \ref{quasi-elliptic-special})), which is not allowed by our list. Hence, the intersection number of each of the multisections with a component of the fiber of type $\I_2$ is $0,2$ or $4$.
If one of the multisections meets only one component of the fiber of type $\I_2$, the other multisection and the other component of the fiber of type $\I_2$ are disjoint from a diagram of type $\tilde{D}_6 + A_1$, hence they meet each other twice. This yields the following dual graph $(A)$, in which the new
vertex is the other component as above. The dual graph $(B)$ occurs if both multisections meet both components of the fiber of type $\I_2$ (this means that the multisections are $4$-sections).



\vspace{2mm}
$
\begin{array}{cc}
 \hspace{-1cm} (A) &  \hspace{3cm} (B) \\

\xy
@={(10,0),(50,0),(10,10),(20,10),(30,10),(40,10),(50,10),(10,20),(50,20),(30,20),(60,30)}@@{*{\bullet}};
(10,10)*{};(50,10)*{}**\dir{-};
(10,0)*{};(10,20)*{}**\dir{-};
(30,10)*{};(30,20)*{}**\dir{-};
(50,0)*{};(50,20)*{}**\dir{-};
(50,20)*{};(60,30)*{}**\dir{=};
\endxy 

%

 & 
 \hspace{3cm}
 
 \xy
(10,25)*{};
@={(10,0),(50,0),(10,10),(20,10),(30,10),(40,10),(50,10),(10,20),(50,20),(30,20),(20,30),(40,30)}@@{*{\bullet}};
(10,10)*{};(50,10)*{}**\dir{-};
(10,0)*{};(10,20)*{}**\dir{-};
(30,10)*{};(30,20)*{}**\dir{-};
(50,0)*{};(50,20)*{}**\dir{-};
(10,20)*{};(20,30)*{}**\dir{=};
(50,20)*{};(40,30)*{}**\dir{=};
(10,20)*{};(40,30)*{}**\dir{=};
(50,20)*{};(20,30)*{}**\dir{=};
(20,30)*{};(40,30)*{}**\dir{=};
\endxy 

\end{array}
$


%

\medskip 
(i)\ We first exclude Case $(B)$. Using one of the diagrams of type $\tilde{A}_1$ which yields a quasi-elliptic fibration with singular fibers $(2\I_2^*,2\III,\III)$, we get the following graph

\vspace{2mm}
 \xy
(-40,25)*{};
@={(10,0),(50,0),(10,10),(20,10),(30,10),(40,10),(50,10),(10,20),(50,20),(30,20),(20,30),(40,30),(30,0)}@@{*{\bullet}};
(10,10)*{};(50,10)*{}**\dir{-};
(10,0)*{};(10,20)*{}**\dir{-};
(30,10)*{};(30,20)*{}**\dir{-};
(50,0)*{};(50,20)*{}**\dir{-};
(10,20)*{};(20,30)*{}**\dir{=};
(50,20)*{};(40,30)*{}**\dir{=};
(10,20)*{};(40,30)*{}**\dir{=};
(50,20)*{};(20,30)*{}**\dir{=};
(20,30)*{};(40,30)*{}**\dir{=};
(30,0)*{};(50,10)*{}**\dir{-};
(30,0)*{};(40,30)*{}**\dir{~};
(20,30)*{};(40,30)*{}**\dir{=};
\endxy 

\noindent
where the new vertex corresponds to a simple component of the fiber of type $\I_2^*$ and where a wiggly line means that the intersection number of the adjacent vertices is 4.
Therefore, there is a subdiagram of type $\tilde{D}_4$ which defines a quasi-elliptic fibration with
a double fiber of type $\I_0^*$ (by Proposition \ref{genus1}).  This is not allowed by our list of special extremal fibrations.

\medskip
(ii)\ Now we exclude Case $(A)$. Since Case $(B)$ does not occur, the three other $\tilde{E}_7$ diagrams give rise to $(-2)$-curves as in case $(A)$ and in particular, we get the following graph

\vspace{2mm}
 \xy
(-40,25)*{};
@={(10,0),(50,0),(10,10),(20,10),(30,10),(40,10),(50,10),(10,20),(50,20),(30,20),(60,30),(60,-10)}@@{*{\bullet}};
(10,10)*{};(50,10)*{}**\dir{-};
(10,0)*{};(10,20)*{}**\dir{-};
(30,10)*{};(30,20)*{}**\dir{-};
(50,0)*{};(50,20)*{}**\dir{-};
(50,20)*{};(60,30)*{}**\dir{=};
(50,00)*{};(60,-10)*{}**\dir{=};
(60,-10)*{};(60,30)*{}**\dir{~};
\endxy 

Here, the wiggly line denotes a non-negative intersection number.
The $\tilde{D}_6$ diagram on the left gives rise to a quasi-elliptic fibration.
Since the two $\tilde{A}_1$ diagrams plus the wiggly line are perpendicular to the $\tilde{D}_6$ diagram, 
they are supported on fibers of this fibration which is only possible if
the wiggly line denotes the intersection number zero. However, this implies that the fibration has singular fibers of type $(2\I_2^*,2\III^*,2\III^*)$ which is a contradiction.
Therefore, an Enriques surface with finite automorphism group and this conductrix can not exist.


\noindent \hspace{-7mm} (7)\ Conductrix:
\xy
(13,25)*{};
@={(0,10),(10,10),(20,10),(30,10),(40,10)}@@{*{\bullet}};
(0,10)*{};(40,10)*{}**\dir{-};
(0,13)*{1};
(10,13)*{1};
(20,13)*{1};
(30,13)*{1};
(40,13)*{1};
\endxy \quad
Singularities: $4A_1$ \\
Possible special extremal fibrations: $(2\I_0^*,2\I_0^*)$ quasi-elliptic and $(\I_4^*)$ elliptic
\\

In this case, we will show that $X$ has the dual graph of an Enriques surface of type $\tilde{D_4}+\tilde{D}_4$.
If we start with a special elliptic fibration with a singular fiber of type $(2\I_0^*,2\I_0^*)$, then
we immediately obtain the dual graph of Enriques surface of type $\tilde{D}_4 + \tilde{D}_4$.

In case of a special elliptic fibration with a singular fiber of type $(\I_4^*)$, we will reduce to the previous paragraph. First, note that the support of $A$ is equal to the support of the components of $\I_4^*$ of multiplicity $2$ by Table \cite[pp.16-18]{ES}. Now, we have to observe that a special $2$-section $N$ has to meet the conductrix, for otherwise we would obtain a special genus one fibration with a singular fiber of type $2\III$, $\I_4$ or $\I_8$, which is not allowed by the above list of special extremal fibrations. If the $2$-section $N$ meets the conductrix, we obtain a special genus one fibration with a singular double fiber of type $\I_2^*$, $\I_1^*$ or $\I_0^*$. The first two are not allowed by the above list. Thus, we get a quasi-elliptic fibration with a double fiber of type $\I_0^*$.


\noindent \hspace{-7mm} (8)\ Conductrix:
\xy
(13,25)*{};
@={(0,10),(10,10),(20,10),(30,10)}@@{*{\bullet}};
(0,10)*{};(30,10)*{}**\dir{-};
(0,13)*{1};
(10,13)*{1};
(20,13)*{1};
(30,13)*{1};
\endxy \quad
Singularities: $D_4$ or $4A_1$ \\
Possible special extremal fibrations: $(2\I_0^*,\I_0^*)$ quasi-elliptic
\\

We will show that this case does not occur on an Enriques surface with finite automorphism group.
Starting with a fibration with singular fibers of type $(2\I_0^*,\I_0^*)$, the curve of cusps $N$ meets
the component with multiplicity 2 of the singular fiber of type $\I_0^*$ (use Tables \ref{QuasiellipticFibrationsConductrix} and \cite[p.13]{ES}), and hence there is a subdiagram of type $\tilde{D}_7$ which defines a non-extremal fibration (Propositions \ref{Lang} and \ref{Itoh}). Therefore, an Enriques surface with this conductrix can not have a finite automorphism group.


\noindent \hspace{-7mm} (9)\ Conductrix:
\xy
(13,25)*{};
@={(0,10),(10,10),(20,10)}@@{*{\bullet}};
(0,10)*{};(20,10)*{}**\dir{-};
(0,13)*{1};
(10,13)*{1};
(20,13)*{1};
\endxy \quad
Singularities: $D_4$ or $4A_1$ \\
Possible special extremal fibrations: $(\I_0^*, \I_0^*)$ quasi-elliptic, $(2\I_0^*,2\III,\III,\III,\III)$ quasi-elliptic and $(2\I_0^*$, $\III$, $\III,\III,\III)$ quasi-elliptic 
%
\\

We will show that this case does not occur on an Enriques surface with finite automorphism group.
Starting with a quasi-elliptic fibration with singular fibers of type $(\I_0^*, \I_0^*)$, we obtain an elliptic fibration with a singular fiber of type $\I_2^*$ (again use Tables \ref{QuasiellipticFibrationsConductrix} and \cite[p.13]{ES}), which is not extremal by Proposition \ref{Lang}.

As for the fibrations with a double fiber of type $2\I_0^*$, by Lemma \ref{observations} (2), the curve of cusps $N$ meets two components of
each simple fiber of type $\III$.  Therefore there is a diagram of type $\tilde{D}_6$ containing the conductrix. By Lemma \ref{quellorell}, the corresponding fibration is elliptic. But an elliptic fibration with a fiber of type $\I_2^*$ can not be extremal by Propositions \ref{Lang}.



\noindent \hspace{-9mm} (10)\ Conductrix:
\xy
(13,25)*{};
@={(0,10),(10,10)}@@{*{\bullet}};
(0,10)*{};(10,10)*{}**\dir{-};
(0,13)*{1};
(10,13)*{1};
\endxy \quad
Singularities: $D_4$ or $4A_1$ \\
Possible special extremal fibrations: $(\I_0^*,\III,\III,\III,\III)$ quasi-elliptic, $(\I_0^*,2\III$, $\III,\III,\III)$ quasi-elliptic, $(\I_0^*,2\III,2\III,\III,\III)$ quasi-elliptic and $(\I_1^*,\I_4)$ elliptic.
\\

We will show that this case does not occur on an Enriques surface with finite automorphism group.
If there is a quasi-elliptic fibration on this surface (i.e., in one of the first four cases in the above list), we see that there is a configuration of type $\tilde{D}_4$ containing the conductrix by using Tables \ref{QuasiellipticFibrationsConductrix} and \cite[p.13]{ES}. 
It defines an elliptic fibration (Proposition \ref{genus1}) which is not extremal by Proposition \ref{Lang}.

Starting with a special elliptic fibration with singular fibers of type $(\I_1^*, \I_4)$, we look at the intersection of $N$ with the fiber of type $\I_1^*$. Using Table \ref{ExtremalFibrationsConductrix} and \cite[Theorem 3.1]{ES}, we can see that the conductrix consists of the two components 
of the fiber of type $\I_1^*$ with multiplicity 2.  If the special 2-section $N$ meets distinct components, we obtain a configuration giving a double fiber of type $\I_4$ or $\I_5$, which is a contradiction. If $N$ meets a double component once, then there is a special fibration with a fiber of type  
$\I_0^*$ containing the coductrix. Such a fibration is not contained in our list.
 If $N$ meets a simple component twice, we get a double fiber of type $\III$ of a quasi-elliptic fibration. Thus, we have reduced this case to the quasi-elliptic case.   


\noindent \hspace{-9mm} (11)\ Conductrix:
\xy
(13,25)*{};
@={(0,10)}@@{*{\bullet}};
(0,13)*{1};
\endxy \quad
Singularities: $D_4$ or $4A_1$ \\
Possible special extremal fibrations: $(\III,\III,\III,\III,\III,\III,\III,\III)$ quasi-elliptic, any multiplicities
\\

We will show that this case does not occur on an Enriques surface with finite automorphism group.
The 2-section $N$ is nothing but the conductrix and hence $N$ meets two components of each simple fiber of type $\III$ by Lemma \ref{observations} (2).  Thus, we have an elliptic fibration with a fiber of type $\I_0^*$ which is not extremal by
Proposition \ref{Lang}.


\noindent \hspace{-9mm} (12)\ Conductrix:
\xy
(13,25)*{};
@={(10,10),(20,10),(30,10),(20,20)}@@{*{\bullet}};
(10,10)*{};(30,10)*{}**\dir{-};
(20,10)*{};(20,20)*{}**\dir{-};
(10,13)*{1};
(23,13)*{1};
(30,13)*{1};
(23,23)*{1};
\endxy		\quad
Singularities: $4A_1$ \\ 
Possible special extremal fibrations: $(\I_2^*,2\III,2\III)$ quasi-elliptic, $(\I_2^*,\III,2\III)$ quasi-elliptic, $(\I_2^*,\III,\III)$ quasi-elliptic, $(2\I_1^*,\I_4)$ elliptic and $(\IV^*,\I_1,\I_3)$ elliptic
\\

In this case, we will show that $X$ has the dual graph of an Enriques surface of type ${\rm VIII}$.
If there is a quasi-elliptic fibration with singular fibers of type $(\I_2^*,2\III,2\III)$, we have the following configuration of $(-2)$-curves (use Tables \ref{QuasiellipticFibrationsConductrix} and \cite[p.13]{ES}):

\xy
(-40,25)*{};
@={(20,20),(40,20),(10,10),(20,10),(40,10),(50,10),(30,0),(30,-10),(20,-20),(40,-20),(10,-30),(50,-30)}@@{*{\bullet}};
(10,-30)*{};(20,-20)*{}**\dir{=};
(50,-30)*{};(40,-20)*{}**\dir{=};
(10,10)*{};(20,10)*{}**\dir{-};
(20,20)*{};(20,10)*{}**\dir{-};
(50,10)*{};(40,10)*{}**\dir{-};
(40,20)*{};(40,10)*{}**\dir{-};
(20,-20)*{};(30,-10)*{}**\dir{-};
(40,-20)*{};(30,-10)*{}**\dir{-};
(30,0)*{};(30,-10)*{}**\dir{-};
(30,0)*{};(40,10)*{}**\dir{-};
(30,0)*{};(20,10)*{}**\dir{-};
\endxy

The special elliptic fibration induced by the diagram of type $\tilde{D}_5$ meeting the two curves at the bottom gives four more $(-2)$-curves. We leave it to the reader to check that the resulting intersection graph is of type $\VIII$.

Next, we consider the case of a special elliptic fibration with singular fibers of type $(2\I_1^*,\I_4)$.
Let $C$ be the component of the fiber of type $2\I_1^*$ with $C\cdot N=1$. We have seen in the proof of Lemma \ref{ellipticCondSing} that $A.N = 1$, hence $C$ is contained in $A$.  The $2$-section $N$ has to meet a component $E$ of the fiber of type $\I_4$ twice, since special genus one fibrations with a double fiber of type $\IV$, $\I_3$ or $\I_4$ are not allowed by the above list.  
Therefore, there is a quasi-elliptic fibration with the double singular fiber $2(N+E)$ of type $\III$ and curve of cusps $C$.  
This has to be a fibration with singular fibers of type $(\I_2^*,2\III,2\III)$, since the $C$ does not meet the component of the second fiber $F$ of type $\III$
which is also a component of the fiber of type $\I_4$ and $C$ can not meet the other component of $F$ twice by Lemma \ref{observations} (2).  Thus, we reduce this case to the previous case.

Starting with a quasi-elliptic fibration with singular fibers of type $(\I_2^*,\III,2\III)$  or $(\I_2^*,\III,\III)$, we immediately get the existence of a special elliptic fibration with a singular double fiber of type $\I_1^*$ (use Tables \ref{QuasiellipticFibrationsConductrix} and \cite[p.13]{ES}), returning us to the case above.

Finally, we consider the case of 
a special elliptic fibration with singular fibers of type $(\IV^*,\I_1,\I_3)$.
If the $2$-section meets two simple components of the fiber of type $\IV^*$, then there is a fibration
with a double fiber of type $\I_6$, which is a contradiction.  Thus,
the $2$-section  meets either a simple component of the fiber of type $\IV^*$ twice or a double component once. In the first case, we get a quasi-elliptic fibration with a singular fiber of type $2\III$ and in the second case, we get a special elliptic fibration with a double fiber of type $\I_1^*$. Both cases have already been dealt with.


\noindent \hspace{-9mm} (13)\ Conductrix:
\xy
(13,25)*{};
@={(10,10),(20,10),(30,10),(20,20)}@@{*{\bullet}};
(10,10)*{};(30,10)*{}**\dir{-};
(20,10)*{};(20,20)*{}**\dir{-};
(10,13)*{1};
(23,13)*{1};
(30,13)*{1};
(23,23)*{1};
\endxy		\quad
Singularities: $D_4$ \\ 
Possible special extremal fibrations:  $(\I_2^*,\III,2\III)$ quasi-elliptic, $(\I_2^*,\III,\III)$ quasi-elliptic and $(\IV^*,\IV)$ elliptic
\\

This case does not occur on an Enriques surface with finite automorphism group. In fact, this follows immediately from the arguments in the last two paragraphs of the previous case.
%
\\

\noindent \hspace{-9mm} (14)\ Conductrix: $\emptyset$ \quad
Singularities: $D_4,8A_1$ \\
Possible special extremal fibrations: $(\IV,\I_2,\I_6)$ elliptic and $(\III,\I_8)$ elliptic
\\

We will show that this case does not occur on an Enriques surface with finite automorphism group.
We start from any of the two special fibrations and a special $2$-section $N$. 
Consider the intersection of $N$ with the fibers of type $\I_6$ or $\I_8$.  If $N$ meets two disjoint components of the fiber, then there is a genus one fibration with a double fiber of type $\I_n$, which is a contradiction.  Hence, either $N$ meets a component of the fiber twice or meets two adjacent components.  Then we can find a special genus one fibration with an additive double fiber of type $\III$ or $\IV$.
However, these fibrations are not allowed by the above list. Hence a surface with these singularities can not have finite automorphism group.
\\


\noindent \hspace{-9mm} (15)\ Conductrix: $\emptyset$ \quad
Singularities: $12A_1$ \\
Possible special extremal fibrations: $(\I_9,\I_1,\I_1,\I_1)$ elliptic, $(\I_5,\I_5,\I_1,\I_1)$ elliptic, $(2\IV,\I_2,\I_6)$ elliptic, $(2\III,\I_8)$ elliptic and $(\I_3,\I_3,\I_3,\I_3)$ elliptic 
\\

In this case, we will show that $X$ has  the dual graph of an Enriques surface of type $\VII$.
If we start with a special fibration with singular fibers of type $(2\III,\I_8)$, the $2$-section has to meet two adjacent components of the fiber of type $\I_8$. Indeed, the twelve blow-ups for the dissolution all happen on the singular fibers and the eight of them occurring on the fiber of type $\I_8$ are the blow-ups of the intersections of any two adjacent components. By Lemma \ref{blowups}, (3), the $r$-invariant of any $(-2)$-curve is 2.  Thus we have to blow up two points on the special $2$-section, and hence 
it has to meet such a point of intersection. 
The dual graph obtained is the one given in \cite[Fig.4.19.1, (iii)]{Ko}.
From this configuration, we leave it to the reader to verify, using the above list, that the dual graph we obtain is the one of type $\VII$ (the argument is similar to \cite[(4.19.2)]{Ko}).

Starting with a special extremal fibration with singular fibers of type $(2\IV,\I_2,\I_6)$, we can check that there is a special fibration with double fiber of type $2\III$, which returns us to the case above. Indeed, if the $2$-section meets distinct components of every fiber, we either obtain a genus one fibration with a double fiber of type $\I_n$, which is impossible, or a special fibration with a singular fiber of type $\II^*$, which is not allowed by our list.

For the other configurations, we also obtain a special elliptic fibration with a degenerate double fiber from the $2$-section and components of the fiber of type $\I_n$ with $n \geq 3$. Hence, the argument of the previous two cases applies.
\vspace*{-\baselineskip}\leavevmode
\end{proof}

\bigskip

\section{Construction of vector fields}\label{VectFields}
In this section, we explain two methods to construct a candidate of a vector field $D$ on 
an algebraic surface $Y$ such that the quotient surface $Y^D$ becomes
an Enriques surface.

\subsection{Enriques surfaces with an elliptic pencil}
Let $f : Y \longrightarrow {\bf P}^1$ be an elliptic surface
with a section. Assume that
$Y$ is either a $K3$ surface or a rational surface. Then, the generic fiber is
an elliptic curve $E$ over the field $k(t)$ with one variable $t$. Therefore,
there exists a non-zero regular vector field $\delta$ on $E$ which we can regard
as a non-zero rational vector field on $Y$. Taking a suitable vector field $g(t)\frac{\partial}{\partial t}$
and a suitable function $f(t)$ on ${\bf P}^1$, we look for a vector field 
$$
       D = f(t)\{g(t)\frac{\partial}{\partial t} + \delta\}
$$
such that $Y^D$ is birational to an Enriques surface. In many cases, 
double fibers of the Enriques surface $Y^D$ exist over the zero points 
of $g(t)$ by the theory of vector fields (cf. Proposition 2.1).
In this way, we construct Enriques surfaces of type $\tilde{E_6} + \tilde{A_2}$ in Section \ref{sec3}, 
of type $\VII$ in Section \ref{secVII} and of type $\VIII$ in Section \ref{sec8}.

\subsection{Enriques surfaces with a quasi-elliptic pencil}\label{VectFieldsQueen}
 By Queen \cite[Theorem 2]{Q1}, we have two normal forms for the generic fibers
of a quasi-elliptic fiber space over the field $K = k(s)$ with a variable $s$:

(1) $u^2 = a + v + cv^2 + dv^4$ 
with $a , c, d  \in K$ and $d \notin k$,

(2) $ u^2 + u = a + dv^4$
with $a , d\in K$ and $d \notin k$.

Here, $u, v$ are variables. Note that the case (3) in Queen \cite[Theorem 2]{Q1}
does not occur in our case, because the transcendental degree of $K = k(s)$ over $k$
is 1.
As for the relative generalized Jacobians of these quasi-elliptic surfaces, 
Queen \cite[Theorem 1]{Q2} showed
the following:

The generalized Jacobian for (1) : $u^2 = v + cv^2 + dv^4$,

The generalized Jacobian for (2) : $ u^2 + u = dv^4$.

Let us explain how to use case (1) to construct our Enriques surfaces (case (2) works similarly).
By the change of coordinates $x = 1/v + c, ~y =u/v^2$, the generalized Jacobian for (1)
is birational to 
$$
   y^2 = x^3 + c^2x + d,
$$
which is a Weierstrass normal form. By Bombieri-Mumford \cite{BM2}, the relative
Jacobian of the quasi-elliptic Enriques surface is a rational surface. Therefore, this surface is birational to a rational quasi-elliptic surface in the list of Ito \cite[Proposition 5.1]{Ito}.

Starting from Ito's list of rational quasi-elliptic surfaces, 
we pursue the converse procedure to construct
a candidate of an Enriques surface $X$, and using the candidate,
we find a vector field $D$ on a rational surface $Y$ such that $Y^D$ is birational 
to the Enriques surface $X$.
Using this technique, we will obtain Enriques surfaces of type $\tilde{E_8}$ in Section \ref{sec4},
of type $\tilde{E_7}+\tilde{A_1}^{(1)}$, $\tilde{E_7}+\tilde{A_1}^{(2)}$ in Section \ref{sec6}, and of type 
$\tilde{D_8}$ in Section \ref{sec7} and of type 
$\tilde{D_4} + \tilde{D_4}$ in Section \ref{sec5}.
Of course, finding a suitable vector field $D$ is the hardest part of this construction and in the next subsection, we will give a detailed explanation on how to find a candidate for $D$ in the case of Enriques surfaces of type $\tilde{D_4} + \tilde{D_4}$.
 
\subsection{Example: Vector fields for Enriques surfaces of type $\tilde{D}_4 + \tilde{D}_4$}\label{ExampleD4D4}
Following Ito \cite[Proposition 5.1]{Ito}, we take the rational quasi-elliptic surface
defined by
$$
y^2 = x^3 + a^4s^2x + s^3  \quad \mbox{with} ~a \in k.
$$
This quasi-elliptic surface has two singular fibers of type $\I_0^{*}$ (namely,
of type $\tilde{D}_4$) over the points on ${\bf P}^1$ defined by
$s = 0$ and $s = \infty$. Taking the change of coordinates
$$
         x = 1/v + a^2s, ~y =s^2u/v^2,~  s = 1/S
$$
we get 
$$
   u^2 = S^4v + a^2S^3v^2 + Sv^4
$$
Now, to introduce double fibers at $s = 0$ and $s = \infty$ without changing the Jacobian fibration, we add a term $S^7 + S^3$ and a parameter $b$ ($b \neq 0$)
as follows:
\begin{equation}\label{equation0}    
u^2 = b^2S^4v + a^2S^3v^2 + Sv^4 + S^7 + S^3.
\end{equation}
We claim that these surfaces are quotients of rational surfaces by a vector field.
For this purpose, we take the base change by the Frobenius morphism:
$$
        S = t^2.
$$
Then, the surface becomes
$$
   u^2 + b^2 t^8v + a^2t^6v^2 + t^2v^4 + t^{14} + t^6 = 0.
$$
Therefore, by this equation we have
$$
\{(u + at^3v  + tv^2 +t^7 + t^3)/bt^4\}^2 = v.
$$
Now, by the change of coordinates
$$
  w = (u + at^3v  + tv^2 +t^7 + t^3)/bt^4,~ v = v,~ t = t,
$$
we have
$$
       u = bt^4w + at^3w^2 + tw^4 + t^7 + t^3, ~v = w^2.
$$
Via these relations, $k(u, v, t) = k(w, t)$, which is a rational function field of two variables.
Since 
$$
\left\{
\begin{array}{l}
u = bt^4w + at^3w^2 + tw^4 + t^7 + t^3  \\
S = t^2 \\
v =w^2,
\end{array}
\right.
$$
we have
$$
\left\{
\begin{array}{l}
\frac{\partial u}{\partial w} = bt^4  \\
\frac{\partial u}{\partial t} = at^2w^2 + w^4 + t^6 + t^2.
\end{array}
\right.
$$
We put
$$
    D' = (1/t^3)\left(bt^4\frac{\partial}{\partial t} + 
(at^2w^2 + w^4 + t^6 + t^2)\frac{\partial}{\partial w}\right).
$$
Then, we see $D'(u) = 0$, $D'(v) = 0$, $D'(S) = 0$  and
$k(t, w)^{D'} = k(u, v, S)$ with the equation (\ref{equation0}).
For the later use, taking new coordinates $(x, y)$, we consider the change of coordinates
$$
    x = 1/t,~y = t/w.
$$
Then, we have
$$
   \frac{\partial}{\partial t} = x^2 \frac{\partial}{\partial x} + xy \frac{\partial}{\partial y}, ~
 \frac{\partial}{\partial w} = xy^2 \frac{\partial}{\partial y}.
$$
By this change of coordinates, $D'$ becomes
\begin{equation}
D= {1\over x^2y^2} \left(bx^3y^2\frac{\partial}{\partial x} + (ax^2y^2 + x^2 + x^4y^4 + y^4 + bx^2y^3)\frac{\partial}{\partial y}\right)
\end{equation}
where $a, b \in k, b\not=0$. These are the vector fields that we will use in Section \ref{sec5} to construct Enriques surfaces of type $\tilde{D}_4 + \tilde{D}_4$.

\section{Equations of Enriques surfaces and their automorphisms}\label{Equation}

In this section, we give a method to calculate the automorphism groups of Enriques surfaces 
$X$ with finite automorphism group in Theorems \ref{mainSupersingular} (B) and \ref{mainClassical} (B).  
We will use this method in cases of type $\tilde{E}_6+\tilde{A}_2$ (supersingular), type 
$\tilde{E}_8$ (supersingular and classical), type $\tilde{E}_7+\tilde{A}_1^{(1)}$ (supersingular),  type 
$\tilde{D}_8$ (supersingular and classical) and type $\tilde{D}_4+\tilde{D}_4$.  

Let $X$ be an Enriques surface and assume that $X$ has a structure of 
a quasi-elliptic fibration $\varphi : X \longrightarrow {\bf P}^1$. Let $t$ be a parameter
of an affine line ${\bf A}^1$ in the base curve ${\bf P}^1$.
We denote by $C$ the curve of cusps of the quasi-elliptic fibration, and assume that
over the point defined by $t = \infty$ it has  a double fiber $2F_{\infty}$.
We assume that
\begin{equation}\label{Equation-equation}
 y^2 = tx^4 + g_1(t)x^2 + g_{2}(t)x + g_3(t) \quad (g_1(t), g_2(t), g_3(t) \in k[t])
\end{equation} 
is the defining equation of an affine normal surface whose resolution of singularities is 
isomorphic to the open set 
$X \setminus (C \cup F_{\infty})$ of $X$.
Under these conditions, let $\sigma$ be an automorphism of X which preserves $F_{\infty}$. Then
$\sigma$ preserves the structure of the quasi-elliptic fibration $\varphi : X \longrightarrow {\bf P}^1$,
and it acts on the base curve ${\bf P}^1$ with a fixed point at infinity:
$$
\begin{array}{lccc}
    \sigma : &{\bf P}^1& \longrightarrow& {\bf P}^1\\
                  &\cup  &                                & \cup \\
                  & {\bf A}^1 &     &   {\bf A}^1  \\
       &t       & \mapsto   &         c_1 t  + c_2 
\end{array}
$$
Here, $c_1, c_2$ are elements of $k$ with $c_1 \neq 0$.

We set $A = k[t, x, y]/(y^2 + tx^4 + g_1(t)x^2 + g_{2}(t)x + g_3(t))$. Then $A$ is normal
by our assumption. As $k[t,x]$-module, we have
\begin{equation}\label{ring}
          A = k[t, x] \oplus k[t, x]y,\hspace{5.0cm}
\end{equation}
which is a free $k[t, x]$-module.
Since $\sigma$ preserve $C$ and $F_{\infty}$, $\sigma$ acts
on the open set $X \setminus (C \cup F_{\infty})$ of $X$. 

\begin{lemma}
$\sigma$ induces an automorphism of ${\rm Spec}(A)$.
\end{lemma}
\begin{proof}
We consider the change of coordinates
$$
    u = \frac{1}{x}, ~ v = \frac{y}{x^2}.
$$
Then, the equation becomes $v^2 =t + g_1(t)u^2 + g_2(t)u^3 + g_3(t)u^4$,
and the curve $C$ of cusps is given by $u = 0$. 
The open set $X \setminus (C \cup F_{\infty})$
is constructed as some blow-ups of ${\rm Spec}(A)$:
$$
    \pi : X \setminus (C \cup F_{\infty}) \longrightarrow {\rm Spec}(A).
$$
Note that $\pi$ is surjective.
Since $\sigma$ is an automorphism of  $X \setminus (C \cup F_{\infty})$,
we have a morphism
$$
(\pi, \pi\circ \sigma) :  X \setminus (C \cup F_{\infty}) \longrightarrow {\rm Spec}(A) \times {\rm Spec}(A).
$$
We denote by $\Gamma$ the image of the morphism $(\pi, \pi\circ \sigma)$. 
We denote by 
$$p_i : {\rm Spec}(A) \times {\rm Spec}(A) \longrightarrow  {\rm Spec}(A)\ (i = 1, 2)$$
the first and the second projection.
Then, restricting the projection $p_1$ to $\Gamma$, we have a morphism
$$
      p_1|_{\Gamma} : \Gamma \longrightarrow {\rm Spec}(A).
$$
Since ${\rm Spec}(A)$ is affine, the exceptional curves of the blow-ups collapse
by the morphism $(\pi, \pi\circ \sigma)$. Therefore, the morphism $p_1|_{\Gamma}$
is a finite birational morphism. Since ${\rm Spec}(A)$ is normal by our assumption,
we see that by the Zariski main theorem $p_1|_{\Gamma}$ is an isomorphism.
Therefore, we have  a morphism $p_2|_{\Gamma} \circ p_1|_{\Gamma}^{-1} :
{\rm Spec}(A) \longrightarrow {\rm Spec}(A)$ which is the induced automorphism by $\sigma$.
\end{proof}

By this lemma, $\sigma$ acts on 
${\rm Spec}(A)$ and induces an automorphism 
\begin{equation}\label{auto}
   \sigma^* : A \longrightarrow A.\hspace{6.0cm}
\end{equation}

Now we consider  the generic fiber of $\varphi : X \longrightarrow {\bf P}^1$. 
It is a curve  of arithmetic genus one
over $k(t)$ whose affine part is given by the equation (\ref{Equation-equation}). The curve $C$ of cusps
gives a point $P_{\infty}$ of degree 2 on the curve of genus one. We denote by 
${\tilde{L}}(P_{\infty})$ the vector space over $k(t)$ 
associated with the linear system $|P_{\infty}|$
on the curve of genus one. By the Riemann-Roch theorem, we have $\dim {\tilde{L}}(P_{\infty}) = 2$
and we see that $1$ and $x$ give the basis of ${\tilde{L}}(P_{\infty})$. Since $\sigma$ preserves
the curve $C$ of cusps, $\sigma^{*}(x)$ is contained in ${\tilde{L}}(P_{\infty})$. Therefore,
there exist $d_1(t) , d_2(t) \in k(t)$ such that 
$$
   \sigma^{*}(x) =d_1(t)x + d_2(t).
$$
By (\ref{ring}) and (\ref{auto}), there exist $d_3(t,x), d_4(t, x) \in k[t, x]$ such that 
$$
\sigma^{*}(x) = d_3(t,x) + d_4(t, x)y.
$$
Therefore, considering $\sigma^{*}(x)^2$, we have
$$
d_1(t)^2x^2 + d_2(t)^2 = d_3(t,x)^2 + d_4(t, x)^2(tx^4 + g_1(t)x^2 + g_{2}(t)x + g_3(t)).
$$
Since the right-hand-side is in $k[t, x]$, we see that $d_1(t)$ and $d_2(t)$ are also
polynomials of $t$. Therefore, we see that $\sigma$ is of the following form:
\begin{equation}\label{automorphism}
 \sigma: \left\{
\begin{array}{lcl}
   t     & \mapsto   &    c_1 t  + c_2 \quad (c_1, c_2 \in k; c_1 \neq 0)\\
           x   & \mapsto &   d_1(t)x + d_2(t) \quad (d_1(t), d_2(t) \in k[t]; d_1(t) \not\equiv 0)\\
      y & \mapsto & e_1(t,x)y +e_2(t,x) \quad  (e_1(t, x), e_2(t, x) \in k[t, x]; e_1(t, x) \not\equiv 0)
\end{array}
\right.
\end{equation}

\noindent
We will use this method in Theorems \ref{main2aut}, \ref{E8S-main2}, \ref{E8C-mainAut}, \ref{E7S-mainAut}, \ref{D8S-main2},
\ref{D8C-mainAut}, \ref{D4C-main2}. 

\begin{remark}\label{equ-automorphism}
Let $X$ be an Enriques surface which has a structure of 
elliptic or quasi-elliptic fibration $\varphi : X \longrightarrow {\bf P}^1$
defined by
$$
y^2 + g_0(t) y = tx^4 + g_1(t) x^2 + g_2(t) x + g_3(t)
$$
with $g_0(t), g_1(t), g_2(t), g_3(t) \in k[t]$.
Here,  $t$ is a parameter
of an affine line ${\bf A}^1$ in the base curve ${\bf P}^1$.
We denote by $C$ the 2-section defined by $x = \infty$, and by $F_{\infty}$
the fiber over the point on ${\bf P}^1$ defined by $t = \infty$.
We assume that the equation is the defining equation of an affine normal surface 
whose resolution of singularities is isomorphic to the open set 
$X \setminus (C \cup F_{\infty})$ of $X$. 
Under these conditions, let $\sigma$ be an automorphism of X which preserves 
the curve $C$ and the fiber $F_{\infty}$. 
Then, the automorphism $\sigma$ is also expressed as the
form (\ref{automorphism}), and a similar argument to the above works.
\end{remark}

We use the following trivial lemma.
\begin{lemma}\label{trivial}
$k[x, y]$ is a free $k[x^2, y^2]$-module of rank $4$. A basis is given by $1,x, y, xy$.
\end{lemma}

\section{Enriques surfaces of type $\tilde{E_6}+\tilde{A_2}$}\label{sec3}

From Section \ref{sec3} to Section \ref{sec5}, we will construct the examples of 
Enriques surfaces given in Theorem \ref{mainSupersingular} (B) and Theorem \ref{mainClassical} (B). 
First, we consider the Enriques surfaces of type $\tilde{E_6}+\tilde{A_2}$, of type $\VII$ and of type $\VIII$.  In these cases, we use rational elliptic fibrations. 
Next we consider the remaining cases.  In these cases, we use a rational quasi-elliptic fibration.
In this section, we construct Enriques surfaces of type $\tilde{E_6}+\tilde{A_2}$.

\subsection{Supersingular case}

We consider the relatively minimal
non-singular complete  elliptic surface $\psi : R \longrightarrow {\bf P}^1$
associated with a Weierstrass equation
$$
y^2 + sy = x^3
$$
with a parameter $s$.
This surface is a unique rational elliptic surface with a singular fiber of type $\IV$
over the point given by $s = 0$ and a singular fiber of type $\IV^*$
over the point given by $s = \infty$ (Lang \cite[\S 2]{L2}).
Note that all non-singular fibers are supersingular elliptic curves.
We consider the base change of $\psi : R \longrightarrow {\bf P}^1$
by $s = t^2$. 
Then, we have the elliptic surface defined by 
\begin{equation}\label{pencil31}
 y^2 + t^2y = x^3.
\end{equation}
We consider the relatively minimal non-singular complete model 
of this elliptic surface :
\begin{equation}
f : \tilde{R} \longrightarrow {\bf P}^1.
\end{equation}
By considering the change of coordinates defined by $x'=x/t^2, y'=y/t^3, t' = 1/t$,
we have 
$$
y'^2 + t'y' = x'^3.
$$
Thus the surface $\tilde{R}$ is isomorphic to $R$.
The rational elliptic surface $f : \tilde{R} \to {\bf P}^1$ 
has a singular fiber of type $\IV^*$ over the point
given by $t=0$ and a singular fiber of type $\IV$ over the point given by $t=\infty$.

The elliptic surface $f: \tilde{R} \longrightarrow {\bf P}^1$
has three sections $s_{i}$ $(i = 0, 1, 2)$ given as follows:
$$
\begin{array}{ll}
s_0 : \mbox{the zero section}. \\
s_1 : x = y = 0. \\
s_2 : x = 0, y = t^2. \\
\end{array}
$$

On the singular elliptic surface (\ref{pencil31}), we denote by $F_0$ the fiber
over the point defined by $t = 0$, and by $F_{\infty}$ the fiber over the point defined by
$t=\infty$. Both $F_0$ and $F_{\infty}$ are irreducible, and on each $F_i$ $(i=0, \infty)$
the surface (\ref{pencil31}) has only one singular point $P_i$.
The surface $\tilde{R}$ is the surface obtained by the minimal resolution of singularities of the surface (\ref{pencil31}).
We denote the proper transform of $F_i$ on $\tilde{R}$ again by $F_i$ if no confusion can occur.
We have six exceptional curves $E_{0,k}$ $(k = 1,2, \ldots, 6)$ over the point $P_0$ such that
$F_0$ and these six exceptional curves form a singular fiber of type $\IV^*$ of the elliptic surface 
$f : \tilde{R} \longrightarrow {\bf P}^1$ as follows: 
The blow-up at the singular point $P_0$ gives one exceptional curve
$E_{0,1}$, and the surface is non-singular along $F_0$ and has a unique singular point $P_{0,1}$ on
$E_{0,1}$. The blow-up
at the singular point $P_{0,1}$ gives two exceptional curves $E_{0,2}$ and $E_{0,3}$. 
We denote the proper transform of $E_{0,1}$ by $E_{0,1}$.
The three curves  $E_{0,1}$, $E_{0,2}$ and $E_{0,3}$ meet at one point $P_{0,2}$ which
is a singular point of the surface obtained. The blow-up
at the singular point $P_{0,2}$ again gives two exceptional curves $E_{0,4}$ and $E_{0,5}$. 
The three curves  $E_{0,1}$, $E_{0,4}$ and $E_{0,5}$ meet at one point $P_{0,3}$ which
is a singular point of the surface obtained.
The curve $E_{0,2}$ (resp. $E_{0,3}$) intersects $E_{0,4}$ (resp. $E_{0,5}$) and does not
meet other curves.  Finally, the blow-up
at the singular point $P_{0,3}$ gives an exceptional curve $E_{0,6}$ and the resulting surface is non-singular over these curves.
The curve $E_{0,6}$ meets $E_{0,1}$, $E_{0,4}$ and $E_{0,5}$ transversally.
The dual graph of the curves $F_0, E_{0,1},\ldots , E_{0,6}$ is of type $\tilde{E}_6$.
Thus, the cycle 
$$F_0 + E_{0,2} + E_{0,3} + 2(E_{0,1} + E_{0,4} + E_{0,5}) + 3E_{0,6}$$
is a singular fiber of type $\IV^*$.
On the other hand, the blow-up
at the singular point $P_{\infty}$ gives two exceptional curves $E_{\infty,1}$ and $E_{\infty,2}$.
The surface obtained is now relatively minimal and non-singular, that is, nothing but $\tilde{R}$.
The three curves $F_{\infty}$, $E_{\infty,1}$ and $E_{\infty,2}$ form a singular fiber of type $\IV$.
The configuration of these curves is as in the following Figure \ref{enriques1}.
\begin{figure}[!htb]
 \begin{center}
  \includegraphics[width=100mm]{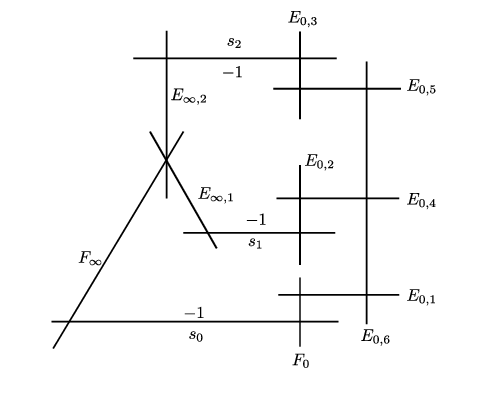}
 \end{center}
 \caption{}
 \label{enriques1}
\end{figure}

\noindent
The sections $s_i$ has the self-intersection number $-1$ and the fiber components have the self-intersection number $-2$.

Now, we consider a rational vector field on $\tilde{R}$ induced from
\begin{equation}\label{E6A2SSderivation}
 D = \frac{\partial}{\partial t} + t^2\frac{\partial}{\partial x}.
\end{equation}
Then, we have $D^2 = 0$, that is, $D$ is $2$-closed. 
However $D$ has an isolated singularity at the point $P$ which is the singular point of
the fiber of type $\IV$, that is, the intersection point of three curves 
$F_{\infty}$, $E_{\infty,1}$ and $E_{\infty,2}$ (note that $(t,x)$ is not a local parameter along the fiber defined by $t=0$). To resolve this singularity, we first blow up at $P$.
Denote by $E_{\infty, 3}$ the exceptional curve.  We denote the proper transforms of 
$F_{\infty}$, $E_{\infty,1}$ and $E_{\infty,2}$  by the same symbols.  Then blow up at three points
$E_{\infty, 3}\cap (F_{\infty} + E_{\infty,1} + E_{\infty,2}).$  
Let $Y$ be the resulting surface and $\psi : Y \to \tilde{R}$ the successive blow-ups.
We denote by $E_{\infty,4}$, $E_{\infty,5}$ or $E_{\infty,6}$ the exceptional curve over the point 
$E_{\infty, 3}\cap F_{\infty}$, $E_{\infty, 3}\cap E_{\infty,1}$ or $E_{\infty, 3}\cap E_{\infty,2}$
respectively.  Then we have the following Figure \ref{enriques2}.  In this Figure \ref{enriques2} we give the self-intersection numbers of the curves except for the curves with the self-intersection number $-2$, and the thick lines are integral curves with respect to $D$.

\begin{figure}[!htb]
 \begin{center}
  \includegraphics[width=100mm]{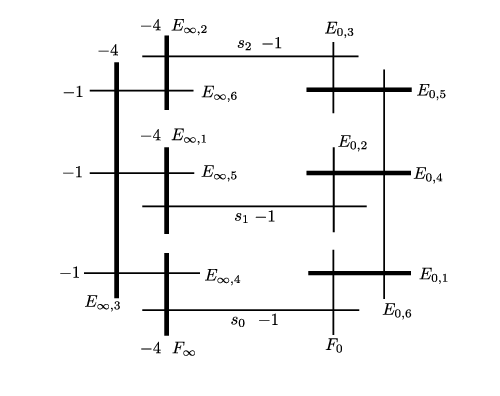}
 \end{center}
 \caption{}
 \label{enriques2}
\end{figure}

Now, according to the above blow-ups, we see the following:

\begin{lemma}\label{divisorialintegral}
{\rm (i)}  The divisorial part $(D)$ on $Y$ is given by
$$-2(E_{0,1} + E_{0,4} + E_{0,5} + E_{0,6} + E_{\infty, 3} + E_{\infty,4}+ E_{\infty,5} + E_{\infty,6})
-(F_{\infty} + E_{\infty,1} + E_{\infty,2}).$$

{\rm (ii)} The integral curves in Figure {\rm \ref{enriques2}} are 
$$E_{0,1}, E_{0,4}, E_{0,5}, F_{\infty}, E_{\infty,1}, E_{\infty,2}, E_{\infty,3}.$$ 

\end{lemma}

\begin{lemma}\label{canonical3}
{\rm (i)} $(D)^2 = -12$. 

{\rm (ii)} The canonical divisor $K_Y$ of $Y$ is given by
$$K_{Y} = -2(E_{\infty, 3} + E_{\infty,4}+ E_{\infty,5} + E_{\infty,6})
-(F_{\infty} + E_{\infty,1} + E_{\infty,2}).$$

{\rm (iii)} $K_Y\cdot (D) = -4.$
\end{lemma}

\begin{lemma}\label{e6non-singular}
$D$ is divisorial and the quotient surface $Y^{D}$ is non-singular.
\end{lemma}
 \begin{proof}
Since $\tilde{R}$ is a rational elliptic surface and $Y$ is the blow-up at 4 points, we have $c_{2}(Y) = 16$.
Using $(D)^2 = -12$, $K_Y\cdot (D) = -4$ and the equation (\ref{euler}), we have 
$$
16 = c_{2}(Y) = \deg \langle D\rangle - K_Y\cdot (D) - (D)^2 = \deg \langle D\rangle  + 4 + 12.
$$
Therefore, we have $\deg  \langle D\rangle = 0$. This means that $D$ is divisorial, and that 
$Y^{D}$ is non-singular.
\end{proof}

Let $\pi : Y\to Y^D$ be the natural map.
By the result on the canonical divisor formula (\ref{canonical}),
we have
$$
        K_Y = \pi^{*} K_{Y^D} + (D).
$$
\begin{lemma}\label{exceptional}
{\rm (i)} The images of the curves $E_{0,1}, E_{0,4}, E_{0,5}$ in $Y^D$
are exceptional curves.

{\rm (ii)} The self-intersection numbers of the images of $F_0, E_{0,2}, E_{0,3}, E_{0,6}$ in $Y^D$ are $-4$.

{\rm (iii)} 
The self-intersection numbers of the images of 
$F_{\infty}, E_{\infty,i}$ $(i=1,\ldots, 6)$ and the three sections $s_i$ $(i=0,1,2)$ in $Y^D$ are $-2$.
\end{lemma}
\begin{proof}
The assertions follow from Proposition \ref{insep} and Lemma \ref{divisorialintegral}, (ii).
\end{proof}

Let $E_{0,1}', E_{0,4}', E_{0,5}', E_{0,6}'$ be the image of $E_{0,1}, E_{0,4}, E_{0,5}, E_{0,6}$ in 
$Y^D$, respectively.  Then we have the following Figure \ref{enriques3} in which we give the self-intersection numbers of the curves except for the curves with self-intersection number $-2$.

\begin{figure}[!htb]
 \begin{center}
  \includegraphics[width=80mm]{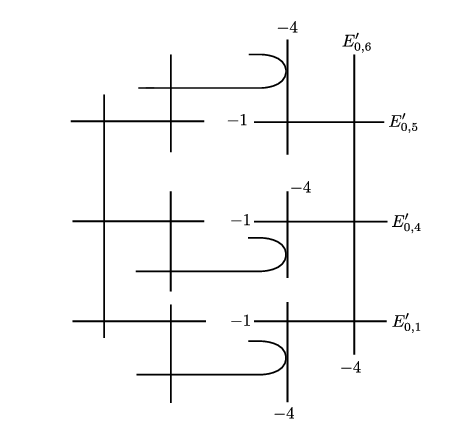}
 \end{center}
 \caption{}
 \label{enriques3}
\end{figure}

Let 
$$\varphi_1 : Y^{D} \to X'$$ 
be the blow-downs of $E_{0,1}', E_{0,4}', E_{0,5}'$. 
Then the image of $E_{0,6}'$ in $X'$ is an exceptional curve.  
Let 
$$\varphi_2 : X' \to X$$ 
be the blow-down of this exceptional curve.
Now we have the following diagram
$$
\begin{array}{ccc}\label{maps}
       \quad    Y^D & \stackrel{\pi}{\longleftarrow} & Y \\
                \varphi_1 \downarrow &    & \quad  \downarrow \psi \\
      \quad        X'     &        & \tilde{R} \\
         \varphi_2   \downarrow & \quad   &  \\
    \quad        X &   &
\end{array}
$$
 
We have thirteen $(-2)$-curves $E_1,\ldots, E_{13}$ with the self-intersection number
$-2$ which form the following Figure \ref{enriques4}.  

\begin{figure}[!htb]
 \begin{center}
  \includegraphics[width=80mm]{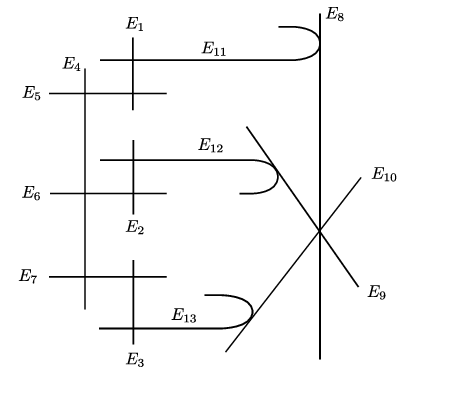}
 \end{center}
 \caption{}
 \label{enriques4}
\end{figure}

Then, we have
$$K_{Y^D} = \varphi_1^{*}(K_{X'}) + E_{0,1} + E_{0,4} + E_{0,5}$$
$$= \varphi_1^* \circ \varphi_2^{*}(K_{X}) + E_{0,6} + 2(E_{0,1} + E_{0,4} + E_{0,5}).$$
\begin{lemma} 
The canonical divisor $K_{X}$ of $X$ is numerically equivalent to $0$.
\end{lemma}
\begin{proof}
By Lemma \ref{canonical3}, (ii),
$$K_Y = -2(E_{\infty, 3} + E_{\infty,4}+ E_{\infty,5} + E_{\infty,6})
-(F_{\infty} + E_{\infty,1} + E_{\infty,2}).$$
On the other hand,  
$$
K_Y  = \pi^{*}K_{Y^D} + (D) 
 = \pi^{*}( \varphi_1^{*}\circ \varphi_2^*(K_{X}) + E_{0,6} + 2(E_{0,1} + E_{0,4} + E_{0,5}))  + (D) = 
$$
$$
 \pi^{*}( \varphi_1^{*}\circ \varphi_2^*(K_{X})) + 2(E_{0,6} + E_{0,1} + E_{0,4} + E_{0,5})  + (D)
=\pi^{*}( \varphi_1^{*}\circ \varphi_2^*(K_{X})) + K_Y.
$$
Here we use the fact that $E_{0,1}, E_{0,4}, E_{0,5}$ are integral and $E_{0,6}$ is non-integral (Lemma \ref{divisorialintegral}, (ii) and Lemma \ref{insep}).
Therefore, $K_{X}$ is numerically equivalent to zero.
\end{proof}

\begin{lemma} 
The surface $X$ has $b_{2}(X) = 10$.
\end{lemma}
\begin{proof}
Since $\pi : Y \longrightarrow Y^{D}$ is finite and
purely inseparable, the \'etale cohomology of $Y$ is isomorphic to 
the \'etale cohomology of $Y^{D}$. Therefore, we have
$b_{1}(Y^{D}) = b_{1}(Y) = 0$, 
$b_{3}(Y^{D})= b_{3}(Y) = 0$ and $b_{2}(Y^{D}) 
= b_{2}(Y) = 14$. Since $\varphi_2\circ \varphi_1$ is the blow-down
of four exceptional curves, we see
$b_{0}(X) =b_{4}(X) = 1$, $b_{1}(X) =b_{3}(X) = 0$ and $b_{2}(X) = 10$.
\end{proof}

\begin{theorem}\label{main}
With the notation above, $X$ is a supersingular Enriques surface.
\end{theorem}
\begin{proof}
Since $K_{X}$ is numerically trivial,
$X$ is minimal and the Kodaira dimension $\kappa (X) $  is equal to $0$.
Since $b_2(X) = 10$, $X$ is an Enriques surface.
Since ${\tilde Y}$ is a rational surface, $X_a$ is either supersingular or classical.
Consider the elliptic fibration $g : X \longrightarrow {\bf P}^1$ induced by
$f : \tilde{R} \longrightarrow {\bf P}^1$.
Note that the fiber over the point given by $t=\infty$ is a double fiber of 
type $\IV^*$ and the fiber over the point given by $t=0$ is a simple fiber of type $\IV$.
Since $f$ has only two singular fibers and any smooth fiber is a supersingular elliptic curve,
the other fibers of $g$ are smooth and supersingular elliptic curves by Lemma \ref{smoothdoublefiber}, and 
hence they are simple by Proposition \ref{multi-fiber}.  
Therefore $X$ is a supersingular Enriques surface by Proposition \ref{multi-fiber}.
\end{proof}

The dual graph of the thirteen $(-2)$-curves $E_1,\ldots, E_{13}$ is as in the following Figure \ref{enriques5}.

\begin{figure}[!htb]
 \begin{center}
  \includegraphics[width=70mm]{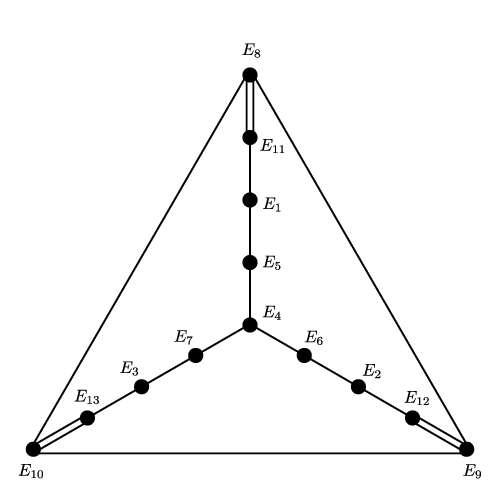}
 \end{center}
 \caption{}
 \label{enriques5}
\end{figure}

\noindent 
We can easily determine all divisors of Kodaira type in Figure \ref{enriques5}.  It follows from Proposition \ref{genus1} that each of them defines a genus one fibration on $X$.  Thus, on $X$, 
there exist exactly one elliptic fibration with singular fibers of type $(2\IV^*, \IV)$ defined by the linear system $|E_8+E_9+E_{10}|$ and 
three quasi-elliptic fibrations with singular fibers of type $(\III^*, 2\III)$ defined by $|2(E_9+E_{12})|, |2(E_8+E_{11})|, |2(E_{10}+E_{13})|$
respectively.

We now have the following theorem.

\begin{theorem}\label{main2}
The automorphism group ${\rm Aut}(X)$ is finite and $X$ contains exactly thirteen $(-2)$-curves.
\end{theorem}
\begin{proof} Consider the dual graph $\Gamma$ of the 13 $(-2)$-curves in Figure \ref{enriques5}. By Remark \ref{Vinbergremark}, $\Gamma$ is non-degenerate.
We can easily prove that any maximal parabolic subdiagram in $\Gamma$ is of type $\tilde{E}_6\oplus \tilde{A}_2$ or of type $\tilde{E}_7\oplus \tilde{A}_1$.  It follows from Propositions \ref{finiteness}, \ref{Vinberg} and \ref{Namikawa} that
$\Aut(X)$ is finite and $X$ contains exactly 13 $(-2)$-curves.
\end{proof}

\begin{theorem}\label{main2aut}
The automorphism group ${\rm Aut}(X)$ is isomorphic to ${\bf Z}/5{\bf Z} \times \mathfrak{S}_3$
and the numerically trivial automorphism group ${\rm Aut}_{nt}(X)$ is isomorphic to ${\bf Z}/5{\bf Z}$.
\end{theorem}
\begin{proof} 
To calculate $\Aut(X)$ we first give an equation of a birational model $X$ which has a structure of
a special elliptic fibration with singular fibers of type $(2\IV^*, \IV)$.
Then we calculate the gouup of automorphisms preserving this fibaration and a 2-secton.
Finally, using this result, we determine the automorphism group of $X$.

We consider the elliptic surface defined by the equation (\ref{pencil31})
and the vector field given in (\ref{E6A2SSderivation}).
Put $T = t^2$, $u = x + t^3$, $v = y + tx^2$.
Then, we have $D(T) = 0$, $D(u) =0$, $D(v) = 0$ and we have the relation
\begin{equation}\label{E6A2SSequation2}
v^2 + Tv = Tu^4 + u^3 + T^3u + T^7.  
\end{equation}
Since we have $k(x, y, t)^{D} = k(u, v, T)$, 
the quotient surface by $D$ is
birational to the surface defined by the equation (\ref{E6A2SSequation2}).
Note that the minimal normal completion of this surface is a normal elliptic surface 
$f : X \longrightarrow {\bf P}^1$ which is birational to our Enriques surface.
We replace variables $u$, $v$, $T$ 
by new variables $x$, $y$, $t$, respectively for convenience, and set
$$
A = k[t, x, y]/(y^2 + ty + tx^4 + x^3 + t^3x + t^7).
$$ 

Let $\sigma$ be an automorphism of
our Enriques surface. The double fiber of $f$, denoted by $2F_{\infty}$, of type $\IV^*$ exists over the point defined by $t = \infty$.
Since $\sigma$ preserves the diagram of $(-2)$-curves, $\sigma$ preserves
$2F_{\infty}$. Therefore, $\sigma$ preserves the structure of this elliptic surface.
We denote by $\tilde{C}$ be the 2-section at infinity and assume that $\sigma$
preserves $\tilde{C}$. Then, as in the case of a quasi-elliptic surface,
$\sigma$ has the form in (\ref{automorphism}) (cf. Remark \ref{equ-automorphism}).
Moreover, since this elliptic surface has a unique second singular fiber over the point defined by $t = 0$,
$\sigma$ preserves also the singular fiber. Therefore, we know $c_2 = 0$ and  we have
$\sigma^{*}(t) = c_1t$.

Therefore, 
together with the equation 
(\ref{E6A2SSequation2}), we have an identity
$$
\begin{array}{l}
e_1(t,x)^2(ty + tx^4 + x^3 + t^3x + t^7) +e_2(t,x)^2+ 
c_1t(e_1(t,x)y +e_2(t,x)) \\
= c_1 t (d_1(t)x + d_2(t))^4 + (d_1(t)x + d_2(t))^3\\
\quad + (c_1t)^3(d_1(t)x + d_2(t)) + (c_1t)^7.
\end{array}
$$
$A$ is a free $k[t, x]$-module, and $1$ and $y$ are linearly independent over $k[t, x]$.
Taking the coefficient of $y$,  we have $e_1(t,x)^2t + c_1te_1(t,x) = 0$.
Since $e_1(t,x) \not\equiv 0$, we have $e_1(t,x) = c_1 \in k^*$.
Therefore, we have
$$
\begin{array}{l}
c_1^2(tx^4 + x^3 + t^3x + t^7) +e_2(t,x)^2+ 
c_1te_2(t,x) \\
= c_1 t (d_1(t)x + d_2(t))^4 + (d_1(t)x + d_2(t))^3\\
\quad + (c_1t)^3(d_1(t)x + d_2(t)) + (c_1t)^7.
\end{array}
$$
As a polynomial of $x$, if $e_2(t,x)$ has a term of degree greater than or equal to 3,
then $e_2(t,x)^2$ has a term greater than or equal to 6. We cannot kill this term in the equation.
Therefore, we can put $e_2(t,x) = a_0(t) + a_1(t) x + a_2(t)x^2$ with $a_0(t), a_1(t), a_2(t) \in k[t]$. We take terms which contain only the variable $t$. Then, we have an equality
$$
c_1^2t^7  + a_0(t)^2 + c_1ta_0(t) = c_1td_2(t)^4 + d_2(t)^3+  c_1^3t^3d_2(t) + c_1^7t^7.
$$
Put $\deg~ d_2(t) = \ell$. Suppose $\ell \geq 2$. Then, the right-hand-side has an odd term
whose degree is equal to $4\ell + 1 \geq 9$. Therefore, the left-hand-side must have 
an odd term which is of degree $4\ell + 1$. This means $\deg~ a_0(t) = 4\ell$.
However, in the equation we cannot kill the term of degree $8\ell$ which comes from 
$a_0(t)^2$. Therefore, we can put $d_2(t) = b_0 + b_1t$ with $b_0, b_1 \in k$.
Then, the equation becomes
$$
\begin{array}{l}
a_0(t)^2 + c_1ta_0(t) + c_1^2t^7\\ 
= c_1b_0^4t +c_1b_1^4t^5 +
b_0^3 + b_0^2b_1t + b_0b_1^2t^2 + b_1^3t^3 
+  c_1^3b_0t^3 +
c_1^3b_1t^4 + c_1^7t^7.
\end{array}
$$
If $\deg ~a_0(t) \geq 4$, we cannot kill the term of degree greater than or equal to 8
which comes from $a_0(t)^2$. Therefore, we can put 
$a_0(t) = \alpha_0 + \alpha_1t + \alpha_2t^2 + \alpha_3t^3$.
Then, we have equations:
$$
\begin{array}{l}
c_1^2 = c_1^7, \alpha_3^2 = 0, 0 = c_1b_1^4, 
\alpha_2^2 + c_1\alpha_3 = c_1^3b_1, c_1\alpha_2 = b_1^3 +  c_1^3b_0,\\
\alpha_1^2 + c_1\alpha_1 = b_0b_1^2, c_1\alpha_0 = c_1b_0^4 + b_0^2b_1,
\alpha_0^2 = b_0^3.
\end{array}
$$
Solving these equations, we have
$$
b_0 = 0, b_1 = 0,\alpha_0 = 0,  \alpha_2 = 0, \alpha_3 = 0, c_1^5 = 1, 
\alpha_1 = 0~\mbox{or}~ c_1.
$$
Therefore, we have
$
c_1 =\zeta, e_1(t, x) = \zeta, a_0(t) = 0~\mbox{or}~\zeta t, d_2(t) = 0.
$
with $\zeta^5 = 1$, $\zeta \in k$.
Putting these date into the original equation, we have
$$
\begin{array}{l}
\zeta^2(tx^4 + x^3 + t^3x) + a_1(t)^2x^2 + a_2(t)^2x^4 + \zeta ta_1(t)x + \zeta ta_2(t)x^2 \\
\quad = \zeta td_1(t)^4x^4 + d_1(t)^3x^3 + \zeta^3t^3d_1(t)x.
\end{array}
$$
Considering the coefficients of $x^4$,
we have 
$\zeta^2 t + a_2(t)^2 + \zeta t d_1(t)^4 =0$. Therefore, we have $a_2(t) = 0$ and 
$d_1(t) ={\zeta}^4$. Considering the coefficients of $x^2$, we have $a_1(t) = 0$.
Therefore we have 
$$
c_1 =\zeta,\quad d_1(t) ={\zeta}^4,\quad d_2(t) = 0,\quad  e_1(t, x) = \zeta,\quad e_2(t, x) = 0~\mbox{or}~\zeta t.
$$
Now we set
$$
\begin{array}{l}
\sigma : t \mapsto \zeta t,\quad x \mapsto {\zeta}^4x,\quad  y \mapsto \zeta y\\
\tau : t \mapsto t,\quad x \mapsto x,\quad y \mapsto y + t.
\end{array}
$$
Then, we have
$$
\sigma\circ \tau : t \mapsto \zeta t,\quad x \mapsto {\zeta}^4x,\quad y \mapsto  \zeta y + \zeta t
$$
and $\langle\sigma\circ \tau\rangle \cong {\bf Z}/10{\bf Z}$. 
We now conclude that the group of automorphisms of $X$ preserving a 2-section of the elliptic fibration
with singular fibers of type $(\VI^*,\VI)$ is isomorphic to ${\bf Z}/10{\bf Z}$.

It is easy to see that the automorphism $\tau$ of order $2$ is not numerically trivial. Indeed, if it were, it would preserve all $2$-sections of $f$. Therefore, it would fix at least $2$ points on them, namely the intersection with the reducible fibers (note that they touch the fiber of type $\IV$, since the other fibrations on this surface have a fiber of type $\III$). Hence, $\tau$ would fix all $2$-sections of $f$ pointwise and therefore it would also fix a general fiber of $f$ by Lemma \ref{genus1auto}. This is a contradiction, hence $\tau$ is not numerically trivial.

Finally, consider the relative Jacobian variety of $f : X \longrightarrow {\bf P}^1$.
It has singular fibers of types ${\rm IV}$ and ${\rm IV}^*$, 
and the Mordell-Weil group is isomorphic to ${\bf Z}/3{\bf Z}$ (cf. Ito \cite{I}). 
We denote by $\rho$ 
a generator of the group. It acts on
$X$ and permutes the three smooth rational 2-sections because if it fixes a 2-section, then it is contained
in the above group ${\bf Z}/10{\bf Z}$.

Therefore, considering the action of the subgroup
$\langle\tau, \rho\rangle$ generated by $\tau$ and $\rho$ on the dual graph of $(-2)$-curves, 
we see that $\langle\tau, \rho\rangle$ acts as the symmetric group $\mathfrak{S}_3$ of degree 3. The automorphism $\sigma$ is numerically trivial because the dual graph of nodal curves has no symmetries of order $5$.
We conclude that ${\rm Aut}(X)$ is an extension of $\mathfrak{S}_3$ by ${\bf Z}/5{\bf Z}$. By the above calculations, $\tau$ commutes with $\sigma$, hence ${\rm Aut}(X) \cong {\bf Z}/5{\bf Z} \times \mathfrak{S}_3$.
\end{proof}

\begin{remark}
Note that $\Aut_{ct}(X) = \Aut_{nt}(X)$ because $X$ is supersingular.  The numerically trivial automorphism $\sigma$ of order $5$ found here leads to one of the exceptions in \cite{DM}.
\end{remark}

\subsection{Classical case}

We consider the relatively minimal
non-singular complete  elliptic surface $\psi : R \longrightarrow {\bf P}^1$
associated with the Weierstrass equation
$$
y^2 + xy + sy = x^3
$$
with a parameter $s$.
This surface is a rational elliptic surface with a singular fiber of type $\I_3$
over the point given by $s = 0$, a singular fiber of type $\I_1$ over the point given by $s=1$ and a singular fibers of type $\IV^*$
over the point given by $s = \infty$ (cf. Lang \cite[\S 2]{L2}).
We consider the base change of $\psi : R \longrightarrow {\bf P}^1$
by $s = t^2$. 
Then, we have the elliptic surface associated with the Weierstrass equation 
\begin{equation}\label{E6CWeierstrass}
y^2 + xy + t^2y = x^3.
\end{equation}
We consider the relatively minimal non-singular complete model 
of this elliptic surface :
\begin{equation}
f : \tilde{R} \longrightarrow {\bf P}^1.
\end{equation}
The rational elliptic surface $f : \tilde{R} \to {\bf P}^1$ 
has a singular fiber of type $\I_6$ over the point
given by $t=0$, a singular fiber of type $\I_2$ over the point given by $t=1$ and a singular fiber of type $\IV$ over the point given by $t=\infty$ (see Figure \ref{E6C1}).
The fibration $f$ has six sections.  In Figure \ref{E6C1}, $(-1)$-curves denote the $0$-section and the two sections defined by
the equations
$$x=y=0, \ \quad  x=y+t^2=0$$
respectively.

\begin{figure}[!htb]
 \begin{center}
  \includegraphics[width=70mm]{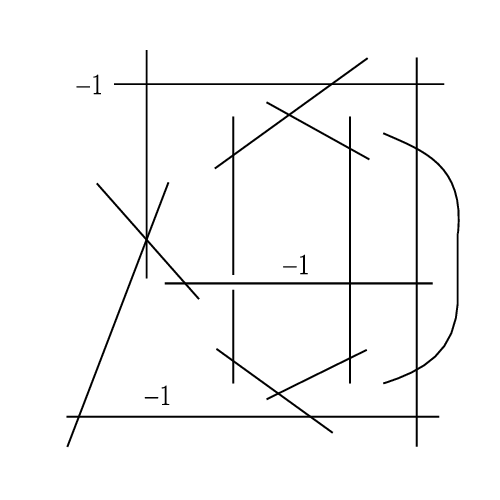}
 \end{center}
 \caption{}
 \label{E6C1}
\end{figure}

\noindent
Now, we consider a rational vector field on $\tilde{R}$ defined by
\begin{equation}\label{DE6C}
D=D_a = (t+a) \frac{\partial}{\partial t} + (x + t^2)\frac{\partial}{\partial x}
\end{equation}
where $a\in k \setminus \{0,1\}$.
We see that $D^2 = D$, that is, $D$ is $2$-closed. 
Note that the non-singular fiber $F_a$ over the point defined by $t=a$ is integral with respect to $D$.
The vector field $D$ has an isolated singularity at the point $P$ which is the singular point of
the fiber of type $\IV$.
Denote by $F_{\infty}$, $E_{\infty,1}$ and $E_{\infty,2}$ the three components of the singular fiber
of type $\IV$.  Then $P$ is the intersection point of these three curves. 
To resolve this singularity, we first blow up at $P$.
Denote by $E_{\infty, 3}$ the exceptional curve.  We denote the proper transforms of 
$F_{\infty}$, $E_{\infty,1}$ and $E_{\infty,2}$  by the same symbols.  Then blow up at three points
$E_{\infty, 3}\cap (F_{\infty} + E_{\infty,1} + E_{\infty,2}).$  
Let $Y$ be the resulting surface and $\psi : Y \to \tilde{R}$ the successive blow-ups.
We denote by the same symbol $D$ the induced vector field on $Y$. 
We denote by $E_{\infty,4}$, $E_{\infty,5}$ or $E_{\infty,6}$ the exceptional curve over the point 
$E_{\infty, 3}\cap F_{\infty}$, $E_{\infty, 3}\cap E_{\infty,1}$ or $E_{\infty, 3}\cap E_{\infty,2}$
respectively.  Then we have the following Figure \ref{E6C2} in which we give the self-intersection numbers of the curves, and the thick curves are integral with respect to $D$.

\begin{figure}[!htb]
 \begin{center}
  \includegraphics[width=80mm]{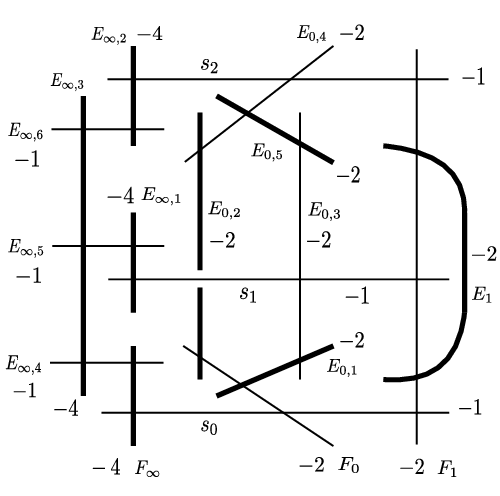}
 \end{center}
 \caption{}
 \label{E6C2}
\end{figure}

\noindent

A direct calculation shows the following Lemmas.

\begin{lemma}\label{E6Clemma1}
{\rm (i)}  The divisorial part $(D)$ of $D$ on $Y$ is given by
$$-(E_1 + E_{0,1} + E_{0,2} + E_{0,5} + F_{\infty} + E_{\infty,1} + E_{\infty,2}) - 2(E_{\infty, 3} + E_{\infty,4}+ E_{\infty,5} + E_{\infty,6}).$$

{\rm (ii)} The integral curves in Figure {\rm \ref{E6C2}} are
$$E_{0,1}, E_{0,2}, E_{0,5}, F_{\infty}, E_{\infty,1}, E_{\infty,2}, E_{\infty,3}, E_1.$$
\end{lemma}

\begin{lemma}\label{E6Clemma2}
{\rm (i)} $(D)^2 = -12$. 

{\rm (ii)} The canonical divisor $K_Y$ of $Y$ is given by
$$K_Y = -(F_{\infty} + E_{\infty,1} + E_{\infty,2}) -2(E_{\infty, 3} + E_{\infty,4}+ E_{\infty,5} + E_{\infty,6}).$$

{\rm (iii)} $K_Y\cdot (D) = -4.$
\end{lemma}

Now, by taking the quotient by $D$, we have the following Figure \ref{E6C3}.
Here the numbers $-1, -4$ denote the self-intersection numbers of curves.  The other curves have the self-intersection number $-2$.

\begin{figure}[!htb]
 \begin{center}
  \includegraphics[width=70mm]{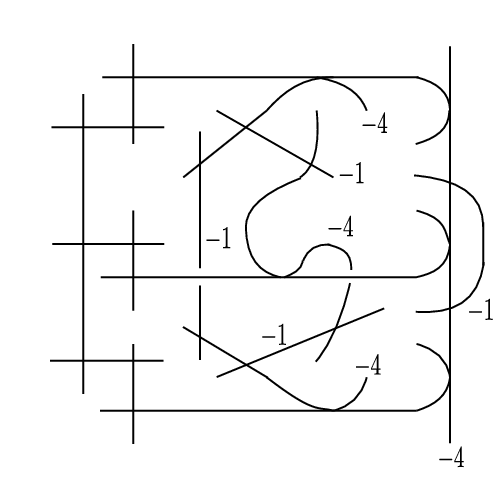}
 \end{center}
 \caption{}
 \label{E6C3}
\end{figure}

We now contract four $(-1)$-curves in Figure \ref{E6C3}, and denote by $X_a$ the surface obtained
which has the dual graph of $(-2)$-curves given in Figure \ref{enriques5}
(recall that the vector field (\ref{DE6C}) contains a parameter $a$). We use the notation of Figure \ref{enriques5}.
On $X_a$, there exist exactly one elliptic fibration with singular fibers of type $(2\IV^*, \I_3, \I_1)$ defined by the linear system $|E_8+E_9+E_{10}|$ and 
three quasi-elliptic fibrations with singular fibers of type $(\III^*, 2\III)$ defined by $|2(E_9+E_{12})|, |2(E_8+E_{11})|, |2(E_{10}+E_{13})|$
respectively.

\begin{theorem}\label{E6Cmain}
The surfaces $\{X_a\}$ form a $1$-dimensional non-isotrivial family of classical Enriques surfaces with the dual graph given in Figure {\rm \ref{enriques5}}. 
\end{theorem}
\begin{proof}
By using Lemmas \ref{E6Clemma1} and \ref{E6Clemma2} and the same argument as in the case of the supersingular surface in the previous subsection, $X_a$ is an Enriques surface.  Since the image of $F_a$ and the singular fiber of type $\IV^*$ are double fibers, $X_a$ is classical by Proposition \ref{multi-fiber}.  
Moreover, since the fibration we used to construct $X_a$ is not isotrivial, the $j$-invariant of the double fiber $F_a$ varies and hence the family $\{X_a\}$ is non-isotrivial. 
By the same proof as that of Theorem \ref{main2}, we prove that $X_a$ contains exactly 13 $(-2)$-curves
whose dual graph is given in Figure \ref{enriques5}.
\end{proof}

\begin{lemma}\label{injective}
The map $\rho_n : {\rm Aut}(X_a) \to {\rm O}({\rm Num}(X_a))$ is injective.
\end{lemma}
\begin{proof}
Let $g \in \Ker({\rho_n})$.  Then $g$ preserves each of the thirteen curves $E_1$,..., $E_{13}$
(see Figure \ref{enriques5}).  First note that 
$g$ fixes three points on each of $E_8, E_9, E_{10}$ (in contrast to the supersingular case, where only two distinct points are fixed). Hence, $g$ fixes $E_8, E_9$ and $E_{10}$ pointwise.
Let $p$ be the quasi-elliptic fibration with singular fibers of type $(\III^*, 2\III)$ defined by the linear system $|2(E_8+E_{11})|$ and let $F$ be a general fiber of $p$.
The two curves $E_{9}, E_{10}$ are $2$-sections of the fibration $p$.  
Then, $g$ fixes at least three points on $F$ which are the intersection with $E_9$ and $E_{10}$ and the cusp of $F$. 
Hence, $g$ fixes $F$ pointwise.  Thus $\rho_n$ is injective.
\end{proof}

By the same arguments as in Theorems \ref{main2} and \ref{main2aut}, we now have the following theorem.

\begin{theorem}\label{main2C}
The automorphism group ${\rm Aut}(X_a)$ is isomorphic to the symmetric group $\mathfrak{S}_3$ of degree three and $X_a$ contains exactly thirteen $(-2)$-curves.
\end{theorem}
\begin{proof} 
By Lemma \ref{injective}, $\Aut(X_a)$ is a subgroup of the symmetry group of the dual graph of $(-2)$-curves which
is isomorphic to $\mathfrak{S}_3$.  By considering the actions of the Mordell-Weil groups of the Jacobian fibrations of genus one fibrations on $X_a$,
any symmetry of the dual graph can be realized by an automorphism of $X_a$.
\end{proof}

\section{Enriques surfaces of type ${\rm VII}$}\label{secVII}

The first and the second author proved the following theorem based on a method given in \cite{KK1}.

\begin{theorem}\label{TypeVII}$($\cite{KK2}$)$
There exists a $1$-dimensional non-isotrivial family of Enriques surfaces with the dual graph of 
$(-2)$-curves
given in Figure {\rm \ref{VIIDynkin}}.  A general member of this family is classical and one special member is supersingular.
The automorphism group of any member in this family is isomorphic to the symmetric group $\mathfrak{S}_5$ of degree $5$.  The canonical cover of any member in this family has $12$ ordinary nodes and 
its minimal resolution is the supersingular $K3$ surface with Artin invariant $1$.
\end{theorem}

\begin{figure}[!htb]
 \begin{center}
\includegraphics[width=60mm]{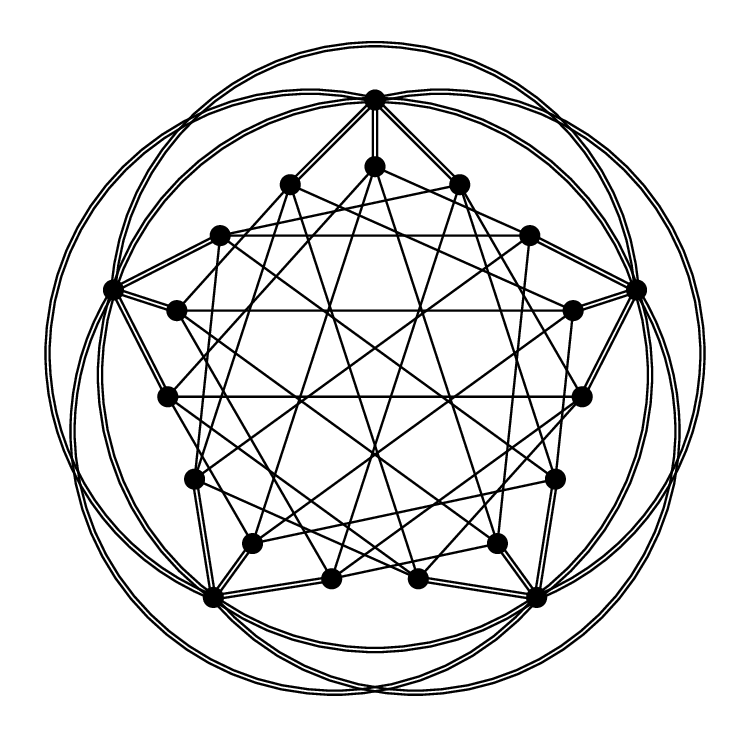}
 \end{center}
 \caption{}
 \label{VIIDynkin}
\end{figure}

There exist elliptic fibrations with singular fibers of type
$(\I_9, \I_1, \I_1, \I_1)$, $(\I_5, \I_5, \I_1, \I_1)$, 
$(\I_8, 2\III)$
or $(\I_6, 2\IV, \I_2)$ on Enriques surfaces of type $\VII$.
For more details, we refer the reader to \cite{KK2}.


\section{Enriques surfaces of type ${\rm VIII}$}\label{sec8}

In this section we give a construction of a one-dimensional family of classical Enriques surfaces with the dual graph of type $\VIII$.

We consider the relatively minimal
non-singular complete  elliptic surface $\psi : R \longrightarrow {\bf P}^1$
associated with the Weierstrass equation
$$
y^2 + sxy = x^3 + s^2x
$$
with a parameter $s$.
This surface is a rational elliptic surface with a singular fiber of type $\I_1^*$
over the point given by $s = 0$ and a singular fiber of type $\I_4$
over the point given by $s = \infty$ (Lang \cite[\S 2]{L2}).
We consider the base change of $\psi : R \longrightarrow {\bf P}^1$
by $s = t^2$. 
Then, we have the Weierstrass model defined by 
\begin{equation}\label{pencilVIII}
y^2 + txy + ty= x^3 + x^2
\end{equation}
(see Lang \cite[\S 2]{L2}).  
We consider the relatively minimal non-singular complete model 
of this elliptic surface :
\begin{equation}
f : \tilde{R} \longrightarrow {\bf P}^1.
\end{equation}
The rational elliptic surface $f : \tilde{R} \to {\bf P}^1$ 
has a singular fiber of type $\III$ over the point
given by $t=0$ and a singular fiber of type $\I_8$ over the point given by $t=\infty$.

On the singular elliptic surface (\ref{pencilVIII}), we denote by $F_0$ the fiber
over the point defined by $t = 0$, and by $E_0$ the fiber over the point defined by
$t=\infty$. Both $F_0$ and $E_0$ are irreducible, and on each $F_0$ and $E_0$,
the surface (\ref{pencilVIII}) has only one singular point $P_0$ and $P_{\infty}$ respectively.
The surface $\tilde{R}$ is a surface obtained by the minimal resolution of singularities of (\ref{pencilVIII}).
We use the same symbol for the proper transforms of curves on $\tilde{R}$.
The blow-up at the singular point $P_0$ gives one exceptional curve
$F_1$, and the surface is non-singular along $F_0$ and $F_1$. The two curves
$F_1$ and $F_0$ form a singular fiber of type $\III$ of the elliptic surface 
$f : \tilde{R} \longrightarrow {\bf P}^1$.  On the other hand,  
the blow-up at the singular point $P_{\infty}$ gives two exceptional curves
$E_1$, $E_2$, and the surface is non-singular along $E_0$ and has a unique singular point $P_1$ which is
the intersection of $E_1$ and $E_2$.
The blow-up
at the singular point $P_1$ gives two exceptional curves $E_3$ and $E_4$. 
The curves  $E_3$ and $E_4$ meet at one point $P_2$ which
is a singular point of the resulting surface. The blow-up
at the singular point $P_2$ again gives two exceptional curves $E_5$ and $E_6$. 
The curves  $E_5$ and $E_6$ meet at one point $P_3$ which
is a singular point of the resulting surface.
Finally, the blow-up
at the singular point $P_3$ gives an exceptional curve $E_7$ and the resulting surface is non-singular over these curves.
The cycle 
$$E_0 + E_1 + E_2 + E_3 + E_4 + E_5 + E_6 + E_7$$
forms a singular fiber of type $\I_8$ given in Figure \ref{VIII1}.

The elliptic surface $f: \tilde{R} \longrightarrow {\bf P}^1$
has four sections $s_{i}$ $(i = 0, 1, 2, 3)$ given as follows:
$$
\begin{array}{ll}
s_0 : \mbox{the zero section}. \\
s_1 : x = y = 0. \\
s_2 : x = t,\ y = 0. \\
s_3 : x = 0,\ y = t. \\
\end{array}
$$
Also we consider the following two $2$-sections $b_1, b_2$ defined by:
$$
\begin{array}{ll}
b_1 : x+y=0,\ x^2 + tx +t =0.\\
b_2 : x + y + tx + t=0,\ x^2 +tx + t =0.\\
\end{array}
$$

The configuration of singular fibers, four sections and two $2$-sections is given in the following
Figure \ref{VIII1}:

\begin{figure}[!htb]
 \begin{center}
  \includegraphics[width=110mm]{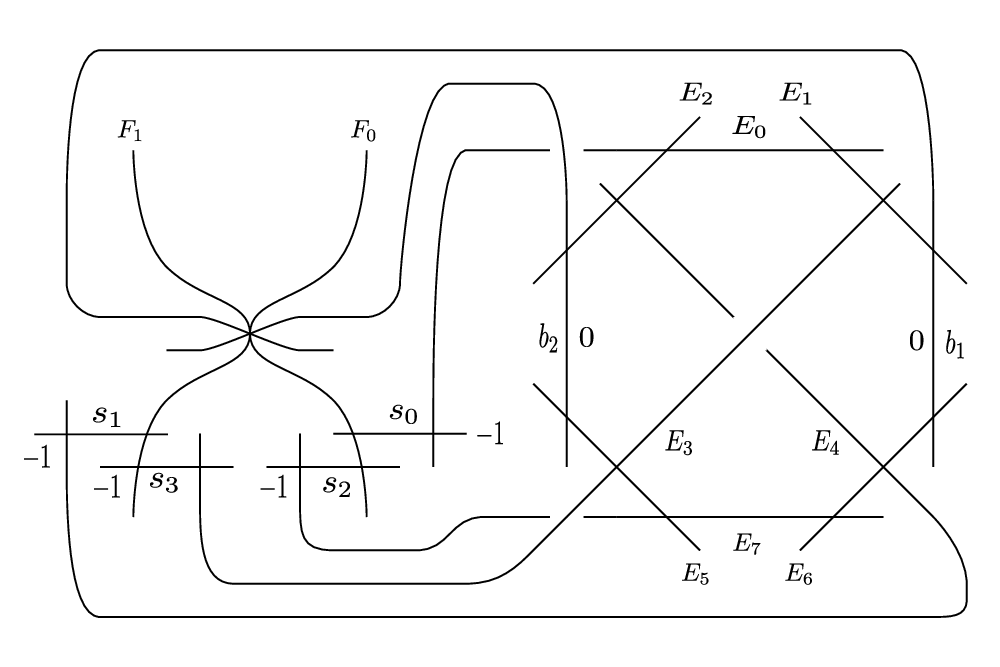}
 \end{center}
 \caption{}
 \label{VIII1}
\end{figure}

Now, we consider a rational vector field on $\tilde{R}$ defined by
$$
 D = D_a = t(at+1)\frac{\partial}{\partial t} + (x+1)\frac{\partial}{\partial x}, \ 0 \neq a \in k.
$$
Then, we have $D^2 = D$, that is, $D$ is $2$-closed. 
However, $D$ has an isolated singularity at the point $P$ which is the singular point of
the fiber of type $\III$, that is, the intersection point of two curves $F_0$ and $F_1$ 
(note that $(x,t)$ is not a local parameter along $F_0$).
To resolve this singularity, we first blow up at $P$.
Denote by $F_2$ the exceptional curve.  We denote the proper transforms of 
$F_0$ and $F_1$ by the same symbols.  Then the induced vector field has three isolated singularities 
one of which is the intersection of three curves and other two of which lie on the curve $F_2$.
Blow up at these three points.  Let $Y$ be the resulting surface and $\psi : Y \to \tilde{R}$ the successive blow-ups.  We denote the induced vector field by the same symbol $D$, and the four
exceptional curves by $F_2$, $F_3$, $F_4$, $F_5$.
Then we have the following Figure \ref{VIII2}.  

\begin{figure}[!htb]
 \begin{center}
  \includegraphics[width=110mm]{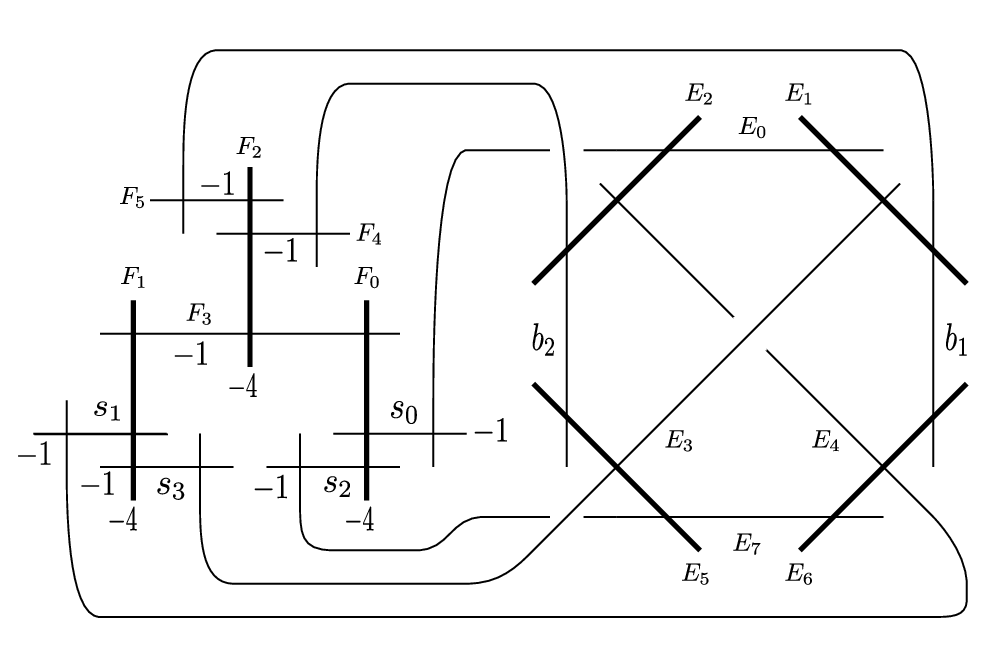}
 \end{center}
 \caption{}
 \label{VIII2}
\end{figure}

\noindent
In the Figure \ref{VIII2} we give the self-intersection numbers of the curves except for the curves with the self-intersection number $-2$.  Also the thick lines are integral curves with respect to $D$. 
Denote by $F_a$ the fiber over the point defined by $t=a^{-1}$.  
Then $F_a$ is integral with respect to $D$.
Now, according to the above blow-ups, we see the following lemmas.

\begin{lemma}\label{divisorialintegralVIII}
{\rm (i)}  The divisorial part $(D)$ of the vector field $D$ on $Y$ is given by
$$-(F_0+F_1+F_2+2F_3+ E_1+E_2+E_5+E_6).$$

{\rm (ii)} The integral curves in Figure {\rm \ref{VIII2}} are 
$$F_0, F_1, F_2, E_1, E_2, E_5, E_6.$$ 

\end{lemma}

\begin{lemma}\label{canonical3VIII}
{\rm (i)} $(D)^2 = -12$. 

{\rm (ii)} The canonical divisor $K_Y$ of $Y$ is given by
$$K_Y = -(F_0+F_1+F_2+2F_3).$$

{\rm (iii)} $K_Y\cdot (D) = -4.$
\end{lemma}

\noindent
Now take the quotient $Y^{D}$ of $Y$ by $D$.  By using the same argument as in the proof of Lemma \ref{e6non-singular}, $D$ is divisorial and $Y^{D}$ is non-singular.
By Proposition \ref{insep}, we have the configuration of curves in Figure \ref{VIII3}.
In the Figure \ref{VIII3} we give the self-intersection numbers of the curves except for the curves with the self-intersection number $-2$.

\begin{figure}[!htb]
 \begin{center}
  \includegraphics[width=100mm]{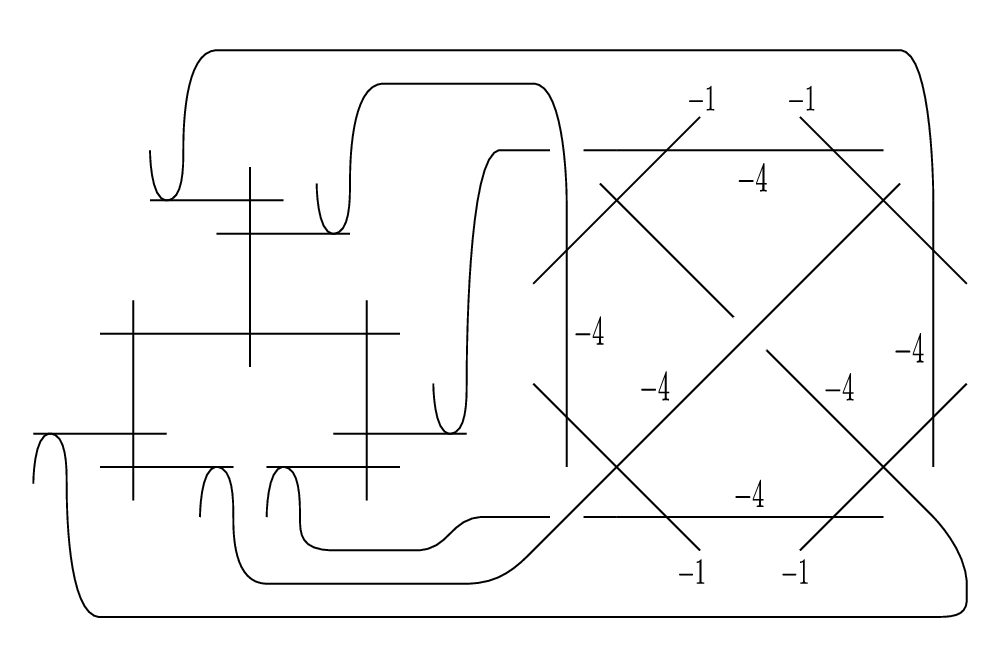}
 \end{center}
 \caption{}
 \label{VIII3}
\end{figure}

Let $X_a$ be the surface obtained by contracting the four exceptional curves in Figure \ref{VIII3}
(Recall that the vector field $D$ contains a parameter $a$).
Then we have the following configuration of $(-2)$-curves in Figure \ref{VIII4}.

\begin{figure}[!htb]
 \begin{center}
  \includegraphics[width=100mm]{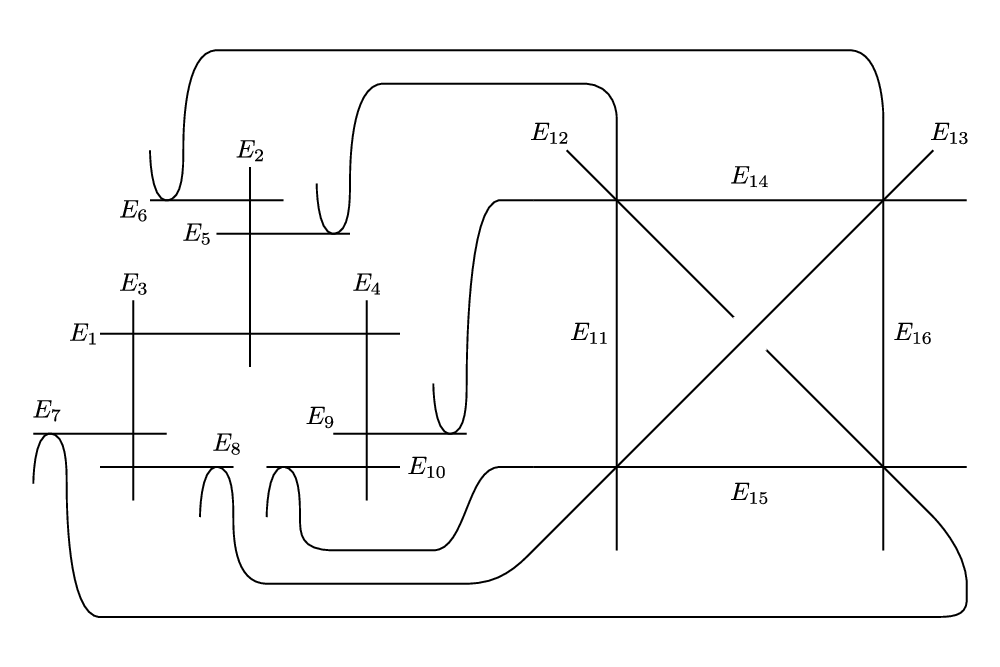}
 \end{center}
 \caption{}
 \label{VIII4}
\end{figure}
\noindent
The dual graph of the sixteen $(-2)$-curves in Figure \ref{VIII4} is nothing but the one given in
Figure \ref{VIIIDynkin}.  Note that any maximal parabolic subdiagram of this diagram is of type
$\tilde{D}_5\oplus \tilde{A}_3$, $\tilde{D}_6\oplus \tilde{A}_1\oplus \tilde{A}_1$ or $\tilde{E}_6\oplus \tilde{A}_2$.  
On $X_a$, there are three types of genus one fibrations:
three elliptic fibrations with singular fibers of type $(2\I_1^*, \I_4)$, three quasi-elliptic fibrations
with singular fibers of type $(\I_2^*, 2\III, 2\III)$ and eight elliptic fibrations with singular fibers of type $(\IV^*, \I_3, \I_1)$.

\begin{figure}[!htb]
 \begin{center}
  \includegraphics[width=80mm]{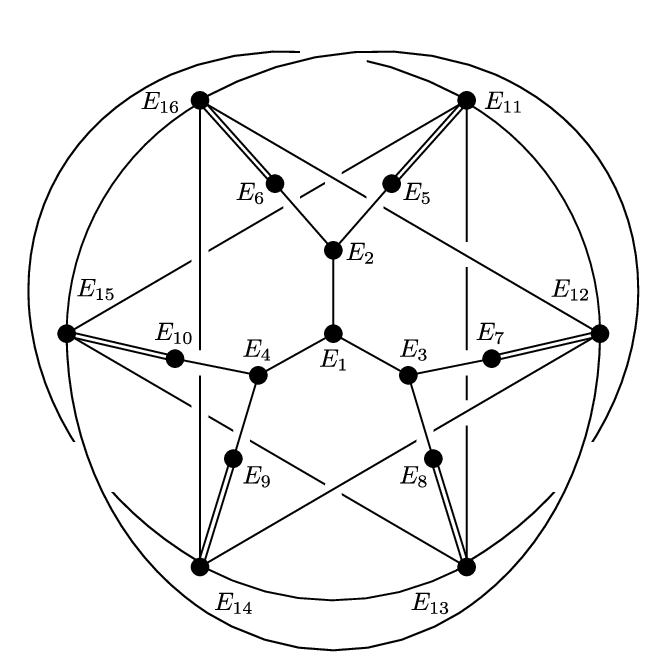}
 \end{center}
 \caption{}
 \label{VIIIDynkin}
\end{figure}

\begin{theorem}
The surfaces $\{X_a\}$ form a non-isotrivial $1$-dimensional family of classical Enriques surfaces with the dual graph given in Figure {\rm \ref{VIIIDynkin}}.  
\end{theorem}
\begin{proof}
By using Lemmas \ref{divisorialintegralVIII} and \ref{canonical3VIII} and the same argument as in the proof of Theorem \ref{main},
$X_a$ is an Enriques surface.
Since $X_a$ has a quasi-elliptic fibration defined by $|2(E_5+E_{11})| = |2(E_6+E_{16})|$ with two double
fibers, $X_a$ is classical (Proposition \ref{multi-fiber}).  Note that the image of $F_a$ is a double fiber of an elliptic fibration with singular fibers of type $(2\I_1^*, \I_4)$.  Since the fibration is non-isotrivial, the $j$-invariant of $F_a$ varies and hence the family $\{X_a\}$ is non-isotrivial.  
By the same proof as that of Theorem \ref{main2}, $X_a$ contains exactly 16 $(-2)$-curves whose dual graph is given in Figure \ref{VIIIDynkin}.
\end{proof}

\begin{theorem}
The automorphism group ${\rm Aut}(X_a)$ is isomorphic to $\mathfrak{S}_4$.
\end{theorem}
\begin{proof}
The quasi-elliptic fibration defined by $|2(E_5+E_{11})|$ has five $2$-sections
$E_2$, $E_{12}$, $E_{13}$, $E_{14}$, $E_{15}$.  Each of these 2-sections meets another $(-2)$-curves at three different points, and hence they are fixed by any numerically trivial automorphism.  Therefore, by the same proof as that of Lemma \ref{injective},
the natural map $\rho_n: \Aut(X_a) \to \O(\Num(X_a))$ is injective.  Note that the automorphism group of the dual graph is isomorphic to the symmetric group $\mathfrak{S}_4$.
By considering the actions of the Mordell-Weil groups of the Jacobian fibrations of genus one fibrations on $X_a$, we have proved that 
$\Aut(X_a)\cong \mathfrak{S}_4$.
\end{proof}

\section{Enriques surfaces of Type $\tilde{E_8}$}\label{sec4}

In this section we give constructions of supersingular and classical Enriques surfaces with the following dual graph of all $(-2)$-curves given in Figure \ref{E10Dynkin}.

\begin{figure}[htbp]
\centerline{
\xy
@={(-10,10),(0,10),(10,10),(20,10),(30,10),(40,10),(50,10),(60,10),(70,10),(10,20)}@@{*{\bullet}};
(-10,10)*{};(70,10)*{}**\dir{-};
(10,10)*{};(10,20)*{}**\dir{-};
\endxy
}
 \caption{}
 \label{E10Dynkin}
\end{figure}

\subsection{Supersingular case}

Let $(x,y)$ be affine coordinates of ${\bf A}^2 \subset {\bf P}^2$.
Consider a rational vector field $D$ defined by
\begin{equation}\label{DE8S}
D= {1\over x^5} \left((xy^6+x^3)\frac{\partial}{\partial x} + (x^6 + y^7 +x^2y)\frac{\partial}{\partial y}\right)
\end{equation}
Then $D^2 = 0$, that is, $D$ is 2-closed.
Note that $D$ has a pole of order 5 along the line $\ell$ defined by $x=0$ and this line is integral with respect to $D$.  We see that $D$ has a unique isolated singularity $(x,y)=(0,0)$.
First blow up at the point $(0,0)$.  Then we see that the exceptional curve is not integral and
the induced vector field has a pole of order 2 along the exceptional curve.
Moreover, the induced vector field has a unique isolated singularity at the intersection of the proper transform of
$\ell$ and the exceptional curve.
Then continue this process until the induced vector field has no isolated singularities.
The final configuration of curves is given in Figure \ref{E8S1}.  Here $F_0$ is the proper transform of
$\ell$ and the suffix $i$ of the exceptional curve $E_i$ corresponds to the order of successive blow-ups.
\begin{figure}[!htb]
 \begin{center}
  \includegraphics[width=140mm]{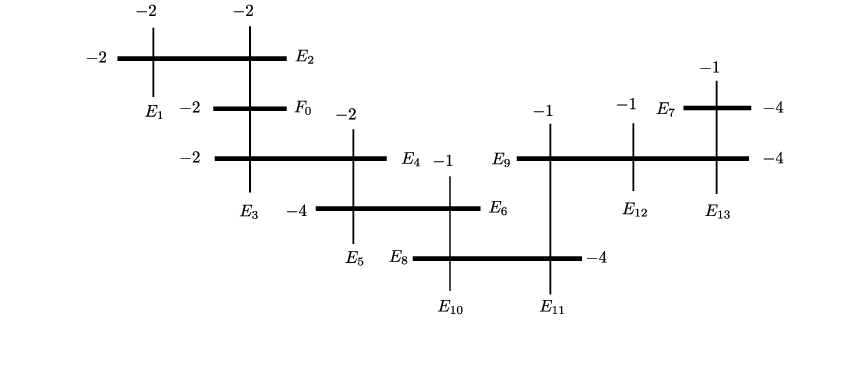}
 \end{center}
 \caption{}
 \label{E8S1}
\end{figure}

We denote by $Y$ the surface obtained by this process.  
Also we denote by the same symbol $D$ the induced vector field on $Y$.
By direct calculations, we have the following lemmas.

\begin{lemma}\label{e8-integral}
{\rm (i)} The integral curves with respect to $D$ in Figure {\rm \ref{E8S1}} are all horizontal curves {\rm (}thick lines{\rm )}.

{\rm (ii)} $(D) =-(5F_0+2E_1+6E_2+8E_3+7E_4+4E_5+3E_6+2E_7+4E_8+5E_9+6E_{10}+8E_{11}+4E_{12}+6E_{13})$.
\end{lemma}

\begin{lemma}\label{e8-canonical}
{\rm (i)} $(D)^2 = -12$. 

{\rm (ii)} The canonical divisor $K_{Y}$ of ${Y}$ is given by
$K_{Y} = -(3F_0+2E_1+4E_2+6E_3+5E_4+4E_5+3E_6+2E_7+4E_8+5E_9+6E_{10}+8E_{11}+4E_{12}+6E_{13}).$

{\rm (iii)} $K_{Y}\cdot (D) = -4.$
\end{lemma}

\noindent
Now take the quotient $Y^{D}$ of $Y$ by $D$.  By using the same argument as in the proof of Lemma \ref{e6non-singular}, $D$ is divisorial and hence $Y^{D}$ is non-singular.
By Proposition \ref{insep}, we have the following configuration of curves on $Y^{D}$ in Figure \ref{E8-2}:
\begin{figure}[!htb]
 \begin{center}
  \includegraphics[width=140mm]{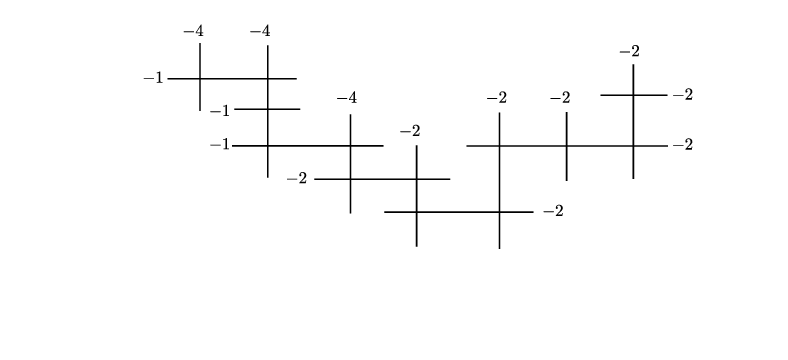}
 \end{center}
 \caption{}
 \label{E8-2}
\end{figure}

\noindent
By contracting the three exceptional curves, we get a new exceptional curve which is the image of the $(-4)$-curve meeting the three exceptional curves.
Let $X$ be the surface obtained by contracting the exceptional curve.
The surface $X$ contains 10 $(-2)$-curves whose dual graph is given by Figure \ref{E10Dynkin}.
Note that this diagram contains a unique maximal parabolic subdiagram which is of type $\tilde{E}_8$.
The pencil of lines in ${\bf P}^2$ through $(x,y)=(0,0)$ induces a quasi-elliptic fibration on $X$ with
a double fiber of type $\II^*$ (the fibration is quasi-elliptic since it is dominated by
a pencil of lines).  

\begin{theorem}\label{E8S-main}
The surface $X$ is a supersingular Enriques surface with the dual graph
given in Figure {\rm \ref{E10Dynkin}}.  
\end{theorem}
\begin{proof}
By using Lemmas \ref{e8-integral} and \ref{e8-canonical} and the same arguments as in the proofs of Theorems \ref{main} and \ref{main2},
$X$ is an Enriques surface with the dual graph given in Figure \ref{E10Dynkin}.  Note that the normalization of the canonical cover of
$X$ is obtained from $Y$ by contracting the divisor $F_0+E_2+E_3+E_4$, and hence it has a rational double point of type $D_4$.  It follows from Lemma \ref{SingNormal} that $X$ is supersingular.
\end{proof}

\begin{theorem}\label{E8S-main2}
${\rm Aut}(X) = {\rm Aut}_{nt}(X)= {\rm Aut}_{ct}(X) \cong {\bf Z}/11{\bf Z}$.
\end{theorem}
\begin{proof}
First note that the dual graph has no symmetries and hence $\Aut(X) = \Aut_{nt}(X)$.  Since $X$ is supersingular, $\Aut_{ct}(X) = \Aut_{nt}(X)$. 

Now we consider the vector field (\ref{DE8S}), and 
we set $u = x^2$, $v = y^2$, $z = x^7 + xy^7 + x^3y$.
Then, we have $D(u) =0$, $D(v) = 0$, $D(z) = 0$ with the equation
\begin{equation}\label{E8SSequation}
z^2 = u^7 + uv^7 + u^3v. 
\end{equation} 
Therefore, the quotient surface of ${\bf P}^2$ by $D$ is
birational to the surface defined by (\ref{E8SSequation}),
which is birational to our Enriques surface. To do a change of coordinates,
we define new variables $x, y, t$ by 
$$
   x = 1/u, \quad y = z/u^4,\quad t = v/u.
$$
Then, the equation becomes 
\begin{equation}\label{E8SSequation2}
y^2 + tx^4 + x + t^7= 0. 
\end{equation}
This equation gives a non-singular affine chart of a
quasi-elliptic surface
$$f : X \to {\bf P}^1$$
by sending $(x,y,t)$ to $t$.  Set
$$
A = k[t, x, y]/(y^2 + tx^4 + x + t^7)
$$ 
and let $\sigma$ be an automorphism of
our Enriques surface. The double fiber of $f$, denoted by $2F_{\infty}$, of type $\II^*$ exists over the point defined by $t = \infty$.
Since $\sigma$ preserves the diagram of $(-2)$-curves, $\sigma$ preserves
the curve $C$ of cusps and $2F_{\infty}$.
Therefore,  $\sigma$ has the form given in (\ref{automorphism}).

Together with the equation (\ref{E8SSequation2}), we have an identity
$$
\begin{array}{l}
e_1(t,x)^2(tx^4 + x + t^7) +e_2(t,x)^2 \\
= (c_1 t  + c_2)(d_1(t)x + d_2(t))^4 + 
(d_1(t)x + d_2(t)) + (c_1 t  + c_2)^7.
\end{array}
$$
Using Lemma~\ref{trivial} and taking the coefficients of $x$, we have $e_1(t,x)^2 + d_1(t) = 0$.
Hence $e_1(t,x)$ is a polynomial of $t$, i.e. we can put $e_1(t,x) = e_1(t)$, and
$d_1(t) = e_1(t)^2$. Taking the coefficients of $t$, we have
$e_1(t)^2x^4 + e_1(t)^2t^6 + c_1(d_1(t)x + d_2(t))^4  + d_2(t)_{odd}/t + c_1(c_1 t  + c_2)^6 = 0$.
Here, $d_2(t)_{odd}$ is the odd terms of $d_2(t)$.
Considering the coefficients of $x^4$ of this equation, we have 
$e_1(t)^2 = c_1d_1(t)^4=c_1e_1(t)^8$. Since we have $e_1(t)\not\equiv 0$, we have
$e_1(t)^6 = 1/c_1$. Therefore, $e_1(t)$ is a constant and we set $e_1(t) =e_1\in k$. Then,
$e_1^6 = 1/c_1$.
Thus we have an identity $e_1^2t^6 + c_1d_2(t)^4 + d_2(t)_{odd}/t  +
 c_1(c_1 t  + c_2)^6 =0$ with $e_1^6 = 1/c_1$.
Let $d_2(t)$ be of degree $m$. If $m\geq 2$ , then we have $\deg ~d_2(t)^4 \geq 8$
and we cannot kill the highest term of $d_2(t)^4$ in the equation. 
Therefore, we can put $d_2(t) = b_0 + b_1t$ $(b_0, b_1 \in k)$
and we have an identity
$$
 (e_1^2 + c_1^7)t^6 + (c_1b_1^4 + c_1^5c_2^2)t^4 + c_1^3c_2^4t^2+ (c_1b_0^4 +b_1 + c_1c_2^6) = 0.
$$
Hence  we have
$e_1^2 + c_1^7 = 0$, $c_1b_1^4 + c_1^5c_2^2= 0$, $c_1^3c_2^4= 0$, $c_1b_0^4 +b_1 + c_1c_2^6 = 0$ 
with $e_1^6 = 1/c_1$. Since $c_1 \neq 0$, we have $c_2 = b_1 = b_0 = 0$ and
$c_1 = \zeta$, $e_1 = \zeta^9$, $d_1 = \zeta^7$ with $\zeta^{11} = 1$.
Putting these data into the original equation, we have $e_2(t, x) = 0$.
Thus we have
$$\sigma(t) = \zeta t, \quad \sigma(x) = \zeta^7x,\quad \sigma(y) = \zeta^9y,$$
and we conclude
${\rm Aut}(X) \cong {\bf Z}/11{\bf Z}$.
\end{proof}

\begin{remark}
The numerically trivial automorphism $\sigma$ of order $11$ found here leads to one of the exceptions in \cite{DM}.
\end{remark}

\subsection{Classical case}

Let $Q = {\bf P}^1\times {\bf P}^1$ be a non-singular quadric and let $((u_0, u_1), (v_0, v_1))$ be
homogeneous coordinates of $Q$. Let $x=u_0/u_1,\ x'=u_1/u_0,\ y =v_0/v_1, y'=v_1/v_0$.
Consider a rational vector field $D$ defined by
\begin{equation}\label{E8C-der}
D= {1\over x^3y^2} \left(x^4y^2\frac{\partial}{\partial x} + (x^2 + ax^4y^4 + y^4)\frac{\partial}{\partial y}\right), \quad  a  \in k^*.
\end{equation}

Then $D^2 = D$, that is, $D$ is 2-closed.
Note that $D$ has a pole of order 3 along the divisor defined by $x=0$, a pole of order 1 along the divisor defined by $x=\infty$ and a pole of order 2 along the divisor defined by $y=0$.
Moreover $D$ has two isolated singularities at $(x,y)= (0,0), (\infty, 0)$.
As in the case of supersingular Enriques surfaces of type $E_8$, we repeatedly blow up the points of isolated 
singularities of $D$ and finally get a vector field $D$, denoted by the same symbol, without isolated singularities.
The configuration of curves is given in Figure \ref{E8C1}.
Here $F_0$, $E_1$, or $E_2$ is the proper transform of the curve defined by $y=0$, $x=0$, or $x=\infty$, respectively, and the suffix $i$ of the other exceptional curve $E_i$ corresponds to the order of successive blow-ups.  We denote by $Y$ the surface obtained by these successive blow-ups.

\begin{figure}[!htb]
 \begin{center}
  \includegraphics[width=140mm]{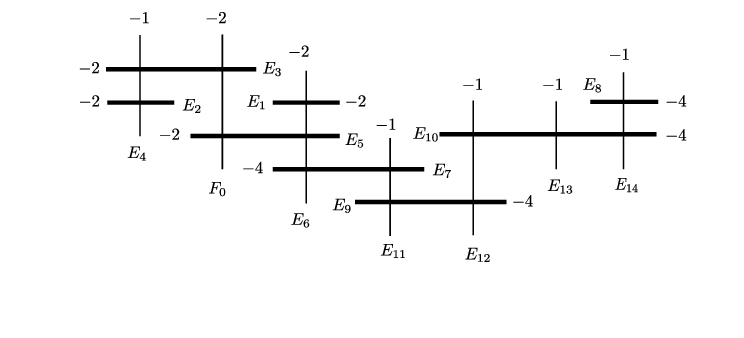}
 \end{center}
 \caption{}
 \label{E8C1}
\end{figure}

A direct calculation shows the following two lemmas \ref{e8C-integral} and \ref{e8C-canonical}.

\begin{lemma}\label{e8C-integral}
{\rm (i)} The integral curves with respect to $D$ in Figure {\rm \ref{E8C1}} are all horizontal curves {\rm (}thick lines{\rm )}.

{\rm (ii)} $(D) =-(2F_0+3E_1+E_2+2E_3+4E_5+4E_6+3E_7+2E_8+4E_9+5E_{10}+6E_{11}+8E_{12}+4E_{13}+6E_{14})$.
\end{lemma}

\begin{lemma}\label{e8C-canonical}
{\rm (i)} $(D)^2 = -12$. 

{\rm (ii)} The canonical divisor $K_{Y}$ of ${Y}$ is given by
$K_{Y} = -(2F_0+2E_1+E_3+3E_5+4E_6+3E_7+2E_8+4E_9+5E_{10}+6E_{11}+8E_{12}+4E_{13}+6E_{14}).$

{\rm (iii)} $K_{Y}\cdot (D) = -4.$
\end{lemma}

\noindent
Now take the quotient $Y^{D}$ of $Y$ by $D$.  By using the same argument as in the proof of Lemma \ref{e6non-singular}, $Y^{D}$ is non-singular.
By Proposition \ref{insep}, we have the following configuration of curves in Figure \ref{E8-1}:
\begin{figure}[!htb]
 \begin{center}
  \includegraphics[width=140mm]{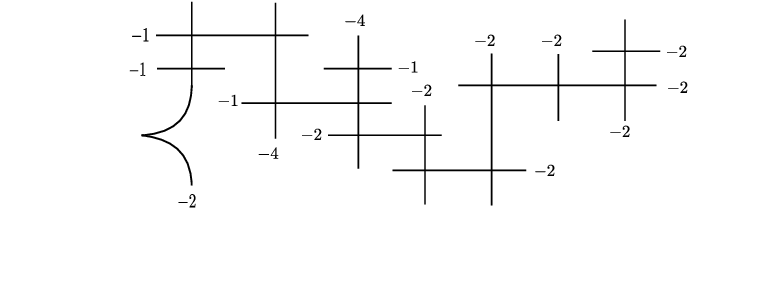}
 \end{center}
 \caption{}
 \label{E8-1}
\end{figure}

Let 
$X_{a}$ be the surface obtained by contracting the four exceptional curves in Figure \ref{E8-1}
(Recall that the vector field $D$ contains  
one parameter $a$ (see (\ref{E8C-der}))).  Then $X_a$ contains 10 $(-2)$-curves whose dual graph is
given by Figure \ref{E10Dynkin}.
Recall that this diagram contains a unique maximal parabolic subdiagram which is of type $\tilde{E}_8$.
The first projection from $Q$ to ${\bf P}^1$ is a ${\bf P}^1$ bundle and it induces a quasi-elliptic fibration on $X_a$ with two double fibers of type $\II^*$ and of type $\II$.

\begin{theorem}\label{E8C-main}
The surfaces
$\{X_{a}\}$ form a 
$1$-dimensional family of classical Enriques surfaces with the dual graph given in Figure {\rm \ref{E10Dynkin}}.
\end{theorem}
\begin{proof}
By using Lemmas \ref{e8C-integral} and \ref{e8C-canonical} and the same arguments as in the proofs of Theorems \ref{main} and \ref{main2}, 
$X_{a}$ is an Enriques surface with the dual graph given in Figure \ref{E10Dynkin}.
Since $X_{a}$ has a genus one fibration with two double fibers of type $\II^*, \II$, 
$X_{a}$ is classical (Proposition \ref{multi-fiber}).  
\end{proof}

\begin{theorem}\label{E8C-mainAut}
The automorphism group 
${\rm Aut}(X_{a})$ is trivial.
\end{theorem}
\begin{proof}
We consider the vector field (8.2), and we set 
$u = x^2$, $v = y^2$, $z = x^3 + ax^5y^4 + xy^4+ x^4y^3$.
Then, we have $D(u) =0$, $D(v) = 0$, $D(z) = 0$ with the equation
\begin{equation}\label{E8Cequation}
z^2 = u^3 + a^2u^5v^4 + uv^4 + u^4v^3\ \ (a \in k^*).
\end{equation}
Therefore, the quotient surface 
${\bf P}^1\times {\bf P}^1$ by $D$ is
birational to the surface defined by (\ref{E8Cequation}) 
which is birational to our Enriques surface. To do a change of coordinates,
we define new variables $x, y, t$ by 
$$
   x = 1/a^{\frac{3}{4}}uv,\quad y = z/a^{\frac{7}{4}}u^4v^2,\quad  t = 1/\sqrt{a}u
$$
and we replace $1/a^{\frac{5}{4}}$ by $b$ for the sake of simplicity.
Then, the equation becomes $y^2 + tx^4 + bt^3x + t^3 + t^7= 0$. 
This equation gives a normal affine
surface.  Now by a similar calculation to the one in the proof of Theorem \ref{E8S-main2},
we see that there are no non-trivial automorphisms, that is,
${\rm Aut}(X_a)$ is trivial.
\end{proof}

\section{Enriques surfaces of type $\tilde{E_7}+\tilde{A_1}^{(1)}$ and $\tilde{E_7}+\tilde{A_1}^{(2)}$}\label{sec6}

\subsection{Classical case of type $\tilde{E_7}+\tilde{A_1}^{(1)}$}\label{E7C1}

In this subsection we give a construction of an Enriques surface with the following dual graph of all 
$(-2)$-curves given in Figure \ref{E7SDynkin}.

\begin{figure}[!htb]
 \centerline{
\xy
(-10,25)*{};
@={(80,10),(90,10),(0,10),(10,10),(20,10),(30,10),(40,10),(50,10),(60,10),(70,10),(30,20)}@@{*{\bullet}};
(0,10)*{};(80,10)*{}**\dir{-};
(90,10)*{};(80,10)*{}**\dir{=};
(30,10)*{};(30,20)*{}**\dir{-};
(70,10)*{};(90,10)*{}**\crv{(80,20)};
\endxy
}
 \caption{}
 \label{E7SDynkin}
\end{figure}

Let $(X_0,X_1,X_2)\in {\bf P}^2$ and $(S,T)\in {\bf P}^1$ be homogeneous coordinates.
Consider the surface $R$ in ${\bf P}^2\times {\bf P}^1$ defined by
\begin{equation}\label{E7Cconic}
S(aX_0^2 +bX_2^2) + T(X_1^2+aX_1X_2 + bX_0X_2) = 0 \ (a, b \in k^*).
\end{equation}
Note that the projection to ${\bf P}^1$ defines a fiber space $\pi:R\to {\bf P}^1$ whose
general fiber is a non-singular conic.
Let $E_1$ be the fiber over the point $(S,T)=(0,1)$ which is non-singular.
The fiber over the point $(S,T)=(1,0)$ is a double line denoted by $2E_2$ and the fiber over the point $(b^2,a^3)$ is a union of two lines denoted by $E_3, E_4$. 
The line defined by $X_2=0$ is a $2$-section of the fiber space which is denoted by $F_0$.
The surface $R$ has two rational double points $Q_i =((\alpha,\beta_i,1), (1,0))$ $(i=1,2)$ of type $A_1$,
where 
$\alpha =\sqrt{b/a}$ and the $\beta_i$'s are the roots of the equation $y^2+ay+\sqrt{b^3/a}=0$.

Let $(x=X_0/X_2,\ y=X_1/X_2, s=S/T)$ be affine coordinates.
Define 
\begin{equation}\label{E7Cderivation}
D= {1\over s} \left(a(s^2+c)\frac{\partial}{\partial x} + (as^2x^2 + bc)\frac{\partial}{\partial y}\right)
\quad (b\not=a^2c)
\end{equation}
where $c$ is a root of the equation of $t^2 + (b/a)t + 1=0$.  Then $D^2=aD$, that is, $D$ is 2-closed.
A direct calculation shows that $D$ has two isolated singularities at the intersection points of $F_0$ and $E_1, E_2$.  
As in the case of supersingular Enriques surfaces of type $E_8$, we blow up the two rational double points and the points of isolated singularities
of $D$ successively, and finally get a vector field, denoted by the same symbol $D$, without isolated singularities.  The configuration of curves is given in Figure \ref{E7-3}.  Here, the suffix $i$ of the exceptional curve $E_i$ corresponds to the order of successive blow-ups.
\begin{figure}[!htb]
 \begin{center}
  \includegraphics[width=150mm]{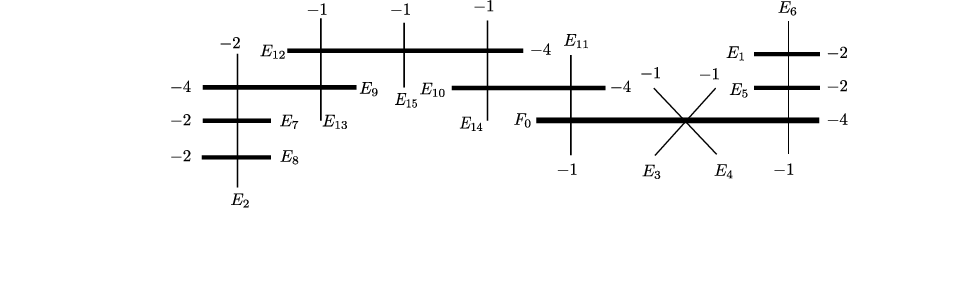}
 \end{center}
 \caption{}
 \label{E7-3}
\end{figure}

Now we denote by $Y$ the surface obtained by successive blow-ups.
By direct calculations, we have the following lemmas.

\begin{lemma}\label{E7C1-integral}
{\rm (i)} The integral curves with respect to $D$ in Figure {\rm \ref{E7-3}} are all horizontal curves {\rm (}thick lines{\rm )}.

{\rm (ii)} $(D) =-(F_0+E_1+2E_2+E_5+2E_7+2E_8+2E_9+2E_{10}+2E_{11}+3E_{12}+4E_{13}+4E_{14}+2E_{15})$.
\end{lemma}

\begin{lemma}\label{E7C1-canonical}
{\rm (i)} $(D)^2 = -12$. 

{\rm (ii)} The canonical divisor $K_Y$ of $Y$ is given by
$K_Y = -(F_0+2E_2+E_7+E_8+2E_9+2E_{10}+2E_{11}+3E_{12}+4E_{13}+4E_{14}+2E_{15}).$

{\rm (iii)} $K_Y\cdot (D) = -4.$
\end{lemma}

\noindent
Now take the quotient $Y^{D}$ of $Y$ by $D$.  By using the same argument as in the proof of Lemma \ref{e6non-singular}, $D$ is divisorial and hence $Y^{D}$ is non-singular.
By Proposition \ref{insep}, we have the following configuration of curves in Figure \ref{E7-4}:

\begin{figure}[!htb]
 \begin{center}
  \includegraphics[width=150mm]{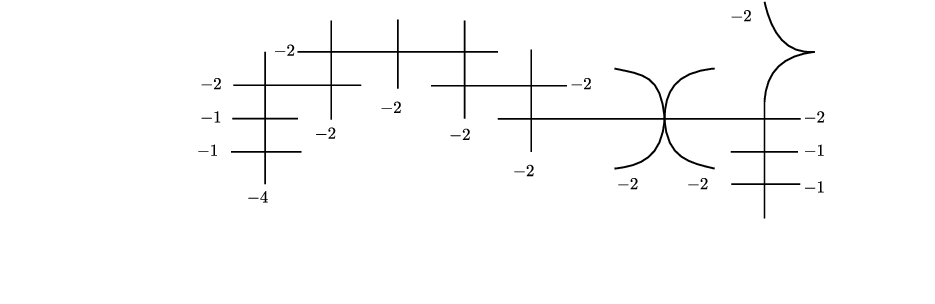}
 \end{center}
 \caption{}
 \label{E7-4}
\end{figure}

Let $X_{a,b}$ be the surface obtained by contracting the four exceptional curves.
The surface $X_{a,b}$ contains 11 $(-2)$-curves whose dual graph is given by Figure \ref{E7SDynkin}.
Note that any maximal parabolic subdiagram of this diagram is of type $\tilde{E}_7\oplus \tilde{A}_1$ 
or $\tilde{E}_8$.

\begin{theorem}\label{E7C-main}
The surfaces $\{X_{a,b}\}$ are classical Enriques surfaces with the dual graph
given in Figure {\rm \ref{E7SDynkin}}.  It contains an at least $1$-dimensional non-isotrivial family. 
\end{theorem}
\begin{proof}
By using Lemmas \ref{E7C1-integral} and \ref{E7C1-canonical} and the same arguments as in the proofs of Theorems \ref{main} and \ref{main2}, 
$X_{a,b}$ is an Enriques surface with the dual graph given in Figure \ref{E7SDynkin}.
Let $p_1$ be the genus one fibration on $X$ with a singular fiber $\III^*$ induced from the fiber space $\pi: R \to \mathbb{P}^1$.  By construction,
$p_1$ has two double fibers (see Figure \ref{E7-4}).  Hence
$X_{a,b}$ is classical (Proposition \ref{multi-fiber}).

In the next Subsection \ref{E7C2III}, we will construct classical Enriques surfaces with double fibers 
of type $\III^*$ and $\III$ which are specializations of $\{X_{a,b}\}$. 
Note that the surface $X_{a,b}$ and the one given in the next subsection are not isomorphic because
their dual graphs of all $(-2)$-curves are different (Figures \ref{E7SDynkin}, \ref{E7CDynkin}).
It will follow from
Matsusaka and Mumford \cite[Theorem 1]{MM} that the family $\{X_{a,b}\}$ contains an at least $1$-dimensional non-isotrivial family. 
\end{proof}

Before computing the automorphism group of $X_{a,b}$, let us summarize what we know about the genus one fibrations on $X_{a,b}$. Using Tables \cite[p.9 and pp.16-18]{ES}, we easily see that the left-most vertex in Figure \ref{E7SDynkin} is part of the conductrix and hence, by Lemma \ref{quellorell} and Proposition \ref{Itoh}, it is the curve of cusps of the two quasi-elliptic fibrations of type $(\II^*)$. Then, similarly, Table \ref{QuasiellipticFibrationsConductrix} shows that the $\tilde{E}_7$ diagram is a double fiber of a quasi-elliptic fibration of type $(2\III^*,\III)$, which is the fibration induced by $\pi$. 

\begin{theorem}\label{E7C-main2}
The automorphism group ${\rm Aut}(X_{a,b})$ is ${\bf Z}/2{\bf Z}$ which is not numerically trivial.
\end{theorem}
\begin{proof}
First, note that it suffices to show that $X_{a,b}$ admits no numerically trivial automorphisms. Indeed, the symmetry group of the dual graph of $(-2)$-curves of $X_{a,b}$ is 
${\bf Z}/2{\bf Z}$ (see Figure \ref{E7SDynkin}) and so is  the Mordell-Weil group of the Jacobian fibration of $p_1$ (see \cite{Ito}). Since the Mordell-Weil group acts effectively on $X_{a,b}$ and it can not act trivially on the graph, it realizes all of ${\rm Aut}(X_{a,b})$.

So, let $g$ be a numerically trivial automorphism.  
Consider a quasi-elliptic fibration $p_2$ with a singular fiber of type $\II^*$.  Let $C_1$, $C_2$ be the double fibers of $p_2$, both of which are rational curves with a cusp, and let $C$ be the cuspidal double fiber of $p_1$.  
Note that $g$ preserves $C$ and fixes at least two points
on it, namely the cusp of $C$ and the intersection of $C$ and the curve of cusps of $p_1$. If $g$ has odd order, then it also fixes $C_i \cap C$, since it will preserve the $C_i$. Hence, by Lemma \ref{genus1auto}, $g$ fixes $C$ pointwise.  

Since $C$ is a $2$-section of $p_2$, $C_i$ is preserved by $g$.
Thus, $g$ fixes
three points on $C_i$, namely the cusp of $C_i$ and the intersection points of $C_i$ with the two double fibers of $p_1$. Hence, again by Lemma \ref{genus1auto}, $g$ fixes $C_1$ and $C_2$ pointwise.  Similarly, $g$ fixes at least 
three points on a general fiber $F$ of $p_1$, namely its cusp and the intersection points with $C_1$ and $C_2$. Therefore $g$ fixes $F$ pointwise.  Thus $g$ is the identity.
\end{proof}

\subsection{The case of type $\tilde{E_7}+\tilde{A_1}^{(2)}$}\label{E7C2III}

In this subsection we give a construction of classical Enriques surfaces with the following dual graph  of all $(-2)$-curves given in Figure \ref{E7CDynkin}.

\begin{figure}[!htb]
 \centerline{
\xy
(-10,25)*{};
@={(80,10),(90,10),(0,10),(10,10),(20,10),(30,10),(40,10),(50,10),(60,10),(70,10),(30,20)}@@{*{\bullet}};
(0,10)*{};(80,10)*{}**\dir{-};
(90,10)*{};(80,10)*{}**\dir{=};
(30,10)*{};(30,20)*{}**\dir{-};
\endxy
}
 \caption{}
 \label{E7CDynkin}
\end{figure}

This example is a specialzation of the previous example given in \ref{E7C1}.
In the previous equations (\ref{E7Cconic}), (\ref{E7Cderivation}), we set $b=0$ and then $c=1$.
The fiber over the point $(S,T)=(0,1)$ is a union of two lines, denoted by $E_1, E_2$, defined by 
$X_1(X_1+aX_2)=0$. 
The fiber over the point $(S,T)=(1,0)$ is a double line denoted by $2E_3$. 
The line defined by $X_2=0$ is a $2$-section of the fiber space which is denoted by $F_0$.
The surface $R$ has two rational double points $Q_1 =((0,0,1), (1,0)), Q_2=((0,a,1), (1,0))$ of type $A_1$.

A direct calculation shows that $D$ has two isolated singularities at the intersection points of the 
$2$-section $F_0$ and the two fibers over the points $(S,T)=(1,0), (0,1)$. As in the previous case, we blow up the two rational double points and the points of isolated singularities
of $D$ successively, and finally get a vector field $D$, denoted by the same symbol, without isolated singularities. 
The configuration of curves is given in Figure \ref{E7-1}.  

\begin{figure}[!htb]
 \begin{center}
  \includegraphics[width=150mm]{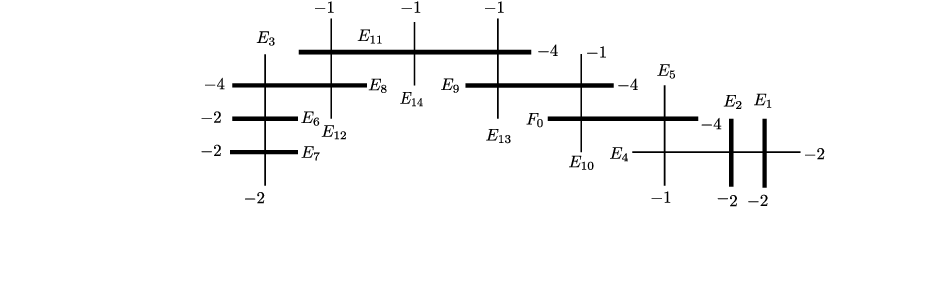}
 \end{center}
 \caption{}
 \label{E7-1}
\end{figure}
\noindent
Here we use the same symbols $F_0$, $E_1$, $E_2$, $E_3$ for the curves and their proper transforms, and the suffixes $i$ of the other exceptional curves $E_i$ correspond to the order of successive blow-ups.  The thick lines are integral curves.
We denote by $Y$ the surface obtained by successive blow-ups.
By direct calculations, we have the following lemmas.

\begin{lemma}\label{E7C2-integral}
{\rm (i)} The integral curves with respect to $D$ in Figure {\rm \ref{E7-1}} are $F_0, E_1, E_2, E_6, E_7$, $E_8, E_9, E_{11}$ $($thick lines$)$.

{\rm (ii)} $(D) =-(F_0+E_1+E_2+2E_3+2E_6+2E_7+2E_8+2E_9+2E_{10}+3E_{11}+4E_{12}+4E_{13}+2E_{14})$.
\end{lemma}

\begin{lemma}\label{E7C2-canonical}
{\rm (i)} $(D)^2 = -12$. 

{\rm (ii)} The canonical divisor $K_Y$ of $Y$ is given by
$K_{Y} = -(F_0+2E_3+E_6+E_7+2E_8+2E_9+2E_{10}+3E_{11}+4E_{12}+4E_{13}+2E_{14}).$

{\rm (iii)} $K_Y\cdot (D) = -4.$
\end{lemma}

\noindent
Now take the quotient $Y^{D}$ of $Y$ by $D$.  By using the same argument as in the proof of Lemma \ref{e6non-singular}, $Y^{D}$ is non-singular.
By Proposition \ref{insep}, we have the following configuration of curves in Figure \ref{E7-2}:

\begin{figure}[!htb]
 \begin{center}
  \includegraphics[width=150mm]{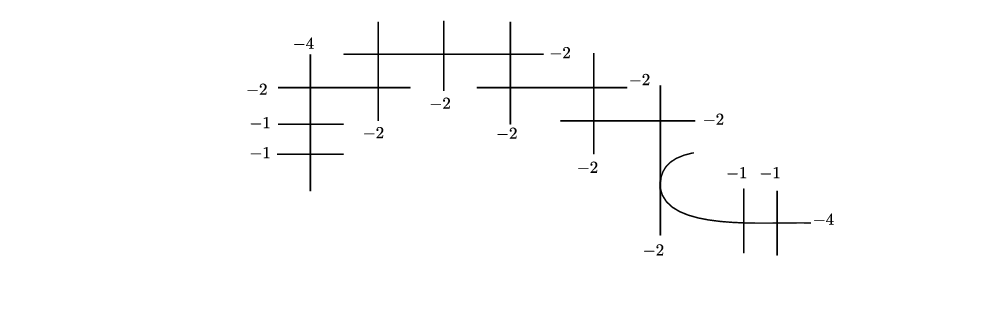}
 \end{center}
 \caption{}
 \label{E7-2}
\end{figure}

Let $X_a$ be the surface obtained by contracting the four exceptional curves.
The surface $X_a$ contains 11 $(-2)$-curves whose dual graph is given by Figure \ref{E7CDynkin}.
Note that any maximal parabolic subdiagram of this diagram is of type $\tilde{E}_7\oplus \tilde{A}_1$ 
or $\tilde{E}_8$.
Once we know that $X_a$ is an Enriques surface, a similar argument to the one in the previous subsection shows that the surface $X_a$ has a quasi-elliptic fibration of type $(2\III^*, 2\III)$ induced from the fiber space $\pi:R\to {\bf P}^1$
and a quasi-elliptic fibration of type $(\II^*)$.

\begin{theorem}\label{E7C2C-main}
The surfaces $\{X_a\}$ form a $1$-dimensional family of classical Enriques surfaces with the dual graph given in Figure {\rm \ref{E7CDynkin}}.
\end{theorem}
\begin{proof}
By using Lemmas \ref{E7C2-integral} and \ref{E7C2-canonical} and the same arguments as in the proofs of Theorems \ref{main} and \ref{main2}, 
$X_a$ is an Enriques surface with the dual graph given in Figure \ref{E7CDynkin}.
Since $X_{a}$ has a quasi-elliptic fibration with two double fibers, $X_{a}$ is classical (Proposition \ref{multi-fiber}).  
\end{proof}

\begin{theorem}\label{E7C2C-main2}
The automorphism group ${\rm Aut}(X_a)$ is ${\bf Z}/2{\bf Z}$ which is numerically trivial.
\end{theorem}
\begin{proof}
By a similar argument to the one in the case of Theorem \ref{E7C-main2}, we see 
$|\Aut_{nt}(X_a)|\leq 2$.  Since the dual graph of $(-2)$-curves on $X_a$ has no symmetries
(see Figure \ref{E7CDynkin}), we have $\Aut(X_a) = \Aut_{nt}(X_a)$.
Let $p$ be the quasi-elliptic fibration with singular fibers of type $(2\III^*, 2\III)$.
Since the Mordell-Weil group of the Jacobian fibration of $p$ is of order $2$ (see \cite{Ito}) and acts effectively on $X_{a}$, we have
$\Aut(X_{a}) \cong {\bf Z}/2{\bf Z}$.  
\end{proof}

\subsection{Supersingular case of type $\tilde{E_7}+\tilde{A_1}^{(1)}$}

In this subsection we give a construction of supersingular Enriques surfaces with the dual graph  of all 
$(-2)$-curves given in Figure \ref{E7SDynkin}.

Let $(X_0,X_1,X_2)\in {\bf P}^2$ and $(S,T)\in {\bf P}^1$ be homogeneous coordinates.
Consider the surface $R$ in ${\bf P}^2\times {\bf P}^1$ defined by
\begin{equation}\label{E7Sconic2}
S(X_0^2 +a^3X_2^2) + T(X_1^2+X_1X_2 + a^2X_0X_2) = 0 \ ( a \in k^*).
\end{equation}
Note that the projection to ${\bf P}^1$ defines a fiber space $\pi:R\to {\bf P}^1$ whose
general fiber is a non-singular conic.
The fiber over the point $(S,T)=(a^4,1)$ is a union of two lines denoted by $E_1, E_2$ and  
the fiber over the point $(S,T)=(1,0)$ is a double line denoted by $2E_3$.
The line defined by $X_2=0$ is a $2$-section, denoted by $F_0$, of the fiber space.

The surface $R$ has two rational double points $Q_i =((\alpha,\beta_i,1), (1,0))$ $(i=1,2)$
where 
$\alpha =\sqrt{a^3}$ and the $\beta_i$'s are the roots of the equation $y^2+y+a^3\sqrt{a}=0$.

Let $(x=X_0/X_2,\ y=X_1/X_2, s=S/T)$ be affine coordinates.
Define 
\begin{equation}\label{E7Sderivation}
D= (s^2+a)\frac{\partial}{\partial x} + (x^2 + a^2s^2)\frac{\partial}{\partial y}.
\end{equation}
Then $D^2=0$, that is, $D$ is 2-closed.
A direct calculation shows that $D$ has an isolated singularity at the intersection point of the 
$2$-section $F_0$ and the fiber over the point $(S,T)=(1,0)$. As in the case of the previous section, we blow up the two rational double points and the point of isolated singularity
of $D$ successively, and finally get a vector field without isolated singularities.
The configuration of curves is given in Figure \ref{E7S}.

\begin{figure}[!htb]
 \begin{center}
  \includegraphics[width=130mm]{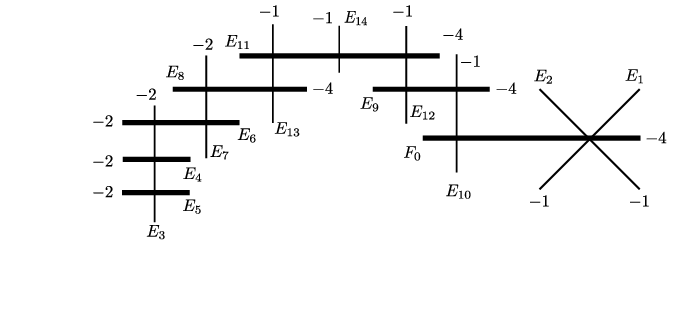}
 \end{center}
 \caption{}
 \label{E7S}
\end{figure}
Here we use the same symbols $F_0$, $E_1$, $E_2$, $E_3$ for the curves and their proper transforms, and the suffix $i$ of the other exceptional curve $E_i$ corresponds to the order of successive blow-ups.  

We denote by $Y$ the surface obtained by successive blow-ups.
By direct calculations, we have the following lemmas.

\begin{lemma}\label{E7S-integral}
{\rm (i)} The integral curves with respect to $D$ in Figure {\rm \ref{E7S}} are all horizontal curves {\rm (}thick lines{\rm )}.

{\rm (ii)} $(D) =-(F_0+4E_3+3E_4+3E_5+4E_6+2E_7+2E_8+2E_9+2E_{10}+3E_{11}+4E_{12}+4E_{13}+2E_{14})$.
\end{lemma}

\begin{lemma}\label{E7S-canonical}
{\rm (i)} $(D)^2 = -12$. 

{\rm (ii)} The canonical divisor $K_{Y}$ of ${Y}$ is given by
$K_{Y} = -(F_0+2E_3+E_4+E_5+2E_6+2E_7+2E_8+2E_9+2E_{10}+3E_{11}+4E_{12}+4E_{13}+2E_{14}).$

{\rm (iii)} $K_{Y}\cdot (D) = -4.$
\end{lemma}

\noindent
Now take the quotient $Y^{D}$ of $Y$ by $D$.  By using the same argument as in the proof of Lemma \ref{e6non-singular}, $Y$ is divisorial and hence $Y^{D}$ is non-singular.
By Proposition \ref{insep}, we have a configuration of curves as in Figure \ref{E7S2}.

\begin{figure}[!htb]
 \begin{center}
  \includegraphics[width=130mm]{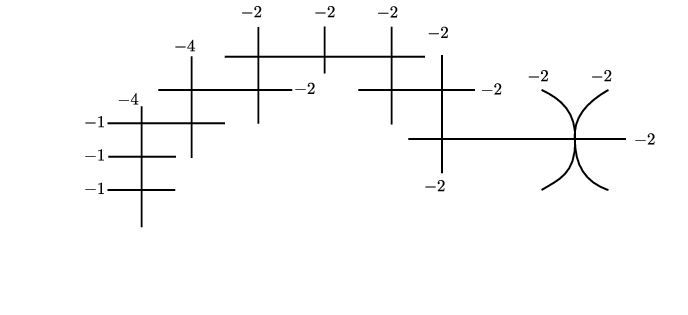}
 \end{center}
 \caption{}
 \label{E7S2}
\end{figure}

Let $X_a$ be the surface obtained by contracting the three exceptional curves and the curve meeting the three exceptional curves.  
The surface $X_a$ contains 11 $(-2)$-curves whose dual graph is given by Figure \ref{E7SDynkin}.
Recall that any maximal parabolic subdiagram of this diagram is of type $\tilde{E}_7\oplus \tilde{A}_1$ or  $\tilde{E}_8$.  By the same argument as in the discussion preceding Theorem \ref{E7C-main2}, we can deduce that the surface $X_a$ contains a unique quasi-elliptic fibration of type $(2\III^*, \III)$  induced from the fiber space $\pi:R\to {\bf P}^1$ and two quasi-elliptic fibrations of type $(\II^*)$ once we know that $X_a$ is an Enriques surface.

\begin{theorem}\label{E7S-main}
The surfaces $\{X_a\}$ are supersingular Enriques surfaces with the dual graph given in Figure {\rm \ref{E7SDynkin}}.
\end{theorem}
\begin{proof}
By using Lemmas \ref{E7S-integral} and \ref{E7S-canonical} and the same arguments as in the proofs of Theorems \ref{main} and \ref{main2}, 
$X_a$ is an Enriques surface with the dual graph given in Figure \ref{E7SDynkin}.  By construction, the normalization of the canonical cover 
has a rational double point of type $D_4$.  It now follows from Lemma \ref{SingNormal} that $X_a$ is supersingular.  
\end{proof}

In contrast to the previous cases, it is not possible to determine ${\rm Aut}(X_{a})$ using only the dual graph of $(-2)$-curves.

\begin{theorem}\label{E7S-mainAut}
If $a^7 \neq 1$, then the automorphism group ${\rm Aut}(X_{a})$ is ${\bf Z}/2{\bf Z}$ which is not numerically trivial. If $a^7 = 1$, then the automorphism group ${\rm Aut}(X_{a})$ is ${\bf Z}/14{\bf Z}$ 
and ${\rm Aut}_{nt}(X_{a})$ is ${\bf Z}/7{\bf Z}$.
\end{theorem}
\begin{proof}
We consider the vector field (\ref{E7Sderivation}), and we set 
$T = s^2$, $u = x + as + s^3$ and $v = y + sx^2 + a^2s^3$. Here,
$s =(y^2 + y + a^2x)/(x^2 + a^3)$ by (\ref{E7Sconic2}).
Then, we have $D(T) =0$, $D(u) = 0$, $D(v) = 0$ with the relation
\begin{equation}\label{E7A1SSequation}
v^2 + v= Tu^4 + a^2u + T^7 
\end{equation} 
and the quotient surface 
${\bf P}^2$ by $D$ is
birational to the surface defined by the equation (\ref{E7A1SSequation}) 
which is birational to our Enriques surface. 
For the sake of simplicity, we replace $a^2$ by $a$ and consider the change of coordinates
with new coordinates $x, y, t$
$$
T = t + a^4,\quad v = y + a^2x^2  + ax,\quad u = x.
$$
Then, the equation becomes 
\begin{equation}\label{E7SSequation2}
y^2 + y = tx^4 + (t + a^4)^7.
\end{equation} 
This equation gives a non-singular affine chart of a genus one fibration 
$f : X_a \to {\bf P}^1$ which is a quasi-elliptic fibration (cf. Subsection \ref{VectFieldsQueen}, Equation (2)) with singular fibers of type $(2\III^*, \III)$
by construction. Set
$$
A = k[t, x, y]/(y^2 + y +  tx^4 + (t + a^4)^7)
$$ 
and let $\sigma$ be an automorphism of
our Enriques surface. The double fiber of $f$, denoted by $2F_{\infty}$, of type $\III^*$ exists over the point defined by $t = \infty$.
Since $\sigma$ preserves the diagram of $(-2)$-curves, $\sigma$ preserves
$2F_{\infty}$. Hence $\sigma$ preserves the structure of this quasi-elliptic surface, and 
has the form in (\ref{automorphism}) (cf. Remark \ref{equ-automorphism}).
Moreover, this quasi-elliptic surface has a singular fiber over the point defined by $t = 0$
and $\sigma$ preserves also the singular fiber. 
Therefore, we have $\sigma^{*}(t) = c_1t$. 

Together with the equation (\ref{E7SSequation2}), we have an identity
$$
\begin{array}{l}
e_1(t,x)^2(y +  tx^4 + (t + a^4)^7) + e_2(t,x)^2 + 
(e_1(t,x)y +e_2(t,x)) \\
= c_1t (d_1(t)x + d_2(t))^4  + (c_1t + a^4)^7.
\end{array}
$$
$A$ is a free $k[t, x]$-module, and $1$ and $y$ are linearly independent over $k[t, x]$.
Taking the coefficient of $y$,  we have $e_1(t,x)^2 + e_1(t,x) = 0$.
Since $e_1(t,x) \neq 0$, we have $e_1(t,x) = 1$.
Hence we have
$$
\begin{array}{l}
tx^4 + (t + a^4)^7 + e_2(t,x)^2 +e_2(t,x) \\
= c_1t (d_1(t)x + d_2(t))^4  + (c_1t + a^4)^7.
\end{array}
$$
As a polynomial of $x$, if $e_2(t,x)$ has a term of degree greater than or equal to 3,
then $e_2(t,x)^2$ has a term greater than or equal to 6. We cannot kill this term in the equation.
By the equation, we know that $e_2(t,x)$ doesn't have terms of $x$ of odd degree.
Therefore, we can put $e_2(t,x) = a_0(t) + a_2(t)x^2$ with $a_0(t), a_2(t) \in k[t]$. 
We take the coefficients of $x^4$. Then, we have
$t + a_2(t)^2 +  c_1t d_1(t)^4= 0$. Hence we have two equations
$1 + c_1d_1(t)^4 = 0$ and $a_2(t)^2 = 0$. Thus we have $a_2(t)= 0$ and 
$d_1(t) = \frac{1}{\sqrt[4]{c_1}}$. The equation becomes 
$(t + a^4)^7 + a_0(t)^2 + a_0(t)  = c_1td_2(t)^4  + (c_1t + a^4)^7$.
Put $\deg~ d_2(t) = \ell$. Suppose $\ell \geq 2$. Then, the right-hand-side has an odd term
whose degree is equal to $4\ell + 1 \geq 9$. Therefore, the left-hand-side must have 
an odd term which is of degree $4\ell + 1$. This means $\deg~ a_0(t) = 4\ell + 1$.
However, in the equation we cannot kill the term of degree $8\ell + 2$ which comes from 
$a_0(t)^2$. Therefore, we can put $d_2(t) = b_0 + b_1t$ with $b_0, b_1 \in k$.
Then, the equation becomes
$$
(t + a^4)^7 + a_0(t)^2 + a_0(t)  
= c_1b_0^4t +  c_1b_1^4t^5  + (c_1t + a^4)^7
$$
If $\deg ~a_0(t) \geq 4$, we cannot kill the term of degree greater than or equal to 8
in the equation which comes from $a_0(t)^2$. Therefore, we can put 
$a_0(t) = \alpha_0 + \alpha_1t + \alpha_2t^2 + \alpha_3t^3$.
Then, we have equations:
$$
\begin{array}{l}
1 = c_1^7,\quad a^4 + \alpha_3^2 = c_1^6a^4,\quad a^8 = c_1b_1^4 + c_1^5a^8, 
a^{12} + \alpha_2^2 = c_1^4a^{12}, \\
a^{16} + \alpha_3 = c_1^3a^{16},\quad 
a^{20} + \alpha_1^2 +  \alpha_2 = c_1^2a^{20}, \\
a^{24} + \alpha_1 = c_1b_0^4 + c_1a^{24},\quad 
a^{28} + \alpha_0^2 + \alpha_0= a^{28}.
\end{array}
$$
Assume $a^7 \neq 1$. Since $\alpha_3 = (c_1^3 + 1)a^2 = (c_1^3 + 1)a^{16}$,
we have $(c_1^3 + 1)a^2(a^7 + 1)^2 = 0$. By $a^7 \neq 1$ and $a \neq 0$, we have
$c_1^3 = 1$. Since $1 = c_1^7$, we have $c_1 = 1$. Therefore, we have
$\alpha_1 =  \alpha_2 = \alpha_3 = 0$, $b_0 = b_1 = 0$, and $\alpha_0 = 1$ or $0$. 
Therefore,  we see that $\sigma$ is given by either
$t \mapsto t,~x \mapsto x, ~ y \mapsto y + 1$ or the identity.
Hence, we have ${\rm Aut}(X_a) \cong {\bf Z}/2{\bf Z}$ if $a^7 \neq 1$.
Now, assume $a^7 = 1$. By $c_1^7 = 1$, $c_1$ is a seventh root of unity.
We denote by $\zeta$ a primitive seventh root of unity.
Then we have a solution
$$
\begin{array}{l}
c_1 = \zeta, \alpha_1 =  0,\quad  \alpha_2 = (1 + \zeta^2)a^6,\quad \alpha_3 = (1 + \zeta^3)a^2, \\
b_0 = \frac{(\zeta^2 + 1)a^6}{\zeta^2},\quad  
b_1 = \frac{(\zeta^3 + 1)a^2}{{\zeta}^2}.
\end{array}
$$
We have also $\alpha_0 = 1$ or $0$. Using this data, we have an automorphism $\sigma$
which is defined by
$$
\begin{array}{rcl}
t & \mapsto &\zeta t\\
x &\mapsto & \frac{1}{\zeta^2}x + \frac{(\zeta^2 + 1)}{\zeta^2}a^6 
+ \frac{(\zeta^3 + 1)}{\zeta^2}a^2t\\
y & \mapsto & y + 1 + (1 + \zeta^2)a^6t^2 + (1 + \zeta^3)a^2t^3.
\end{array}
$$
This $\sigma$ is of order 14, and by our argument the automorphism group is generated by $\sigma$.
This means ${\rm Aut}(X_a) \cong {\bf Z}/14{\bf Z}$ if $a^7 = 1$. 

Finally, to show that ${\rm Aut}(X_a)/{\rm Aut}_{nt}(X_a) \cong {\bf Z}/2{\bf Z}$, it suffices to show that $X_a$ does not admit any numerically trivial involutions, since the symmetry group of the dual graph of $(-2)$-curves is ${\bf Z}/ 2{\bf Z}$. This is similar to the argument in the proof of Theorem \ref{E7C-main2}. By the discussion preceding  Theorem \ref{E7S-main}, we know that all genus one fibrations on $X_{a}$ are quasi-elliptic. Let $p_1$ be the quasi-elliptic fibration of type $(2\III^*,\III)$, let $p_2$ (resp. $p_3$) be a quasi-elliptic fibration of type $(\II^*)$ with a cuspidal double fiber $C_2$ (resp. $C_3$), and let $g$ be a numerically trivial involution. Note that $C_2.C_3 = 1$, as can be read off from the dual graph in Figure \ref{E7SDynkin}. Then, $g$ preserves $C_2$ and $C_3$ and fixes the cusps as well as the intersection $C_2 \cap C_3$ (which is distinct from the cusps since $C_2.C_3 = 1$). Hence, $g$ fixes $C_2$ and $C_3$ pointwise by Lemma \ref{genus1auto} (here is where we use that $g$ has order $2$). Thus, for a general fiber $F$ of $p_1$, $g$ fixes the cusp of $F$, $F \cap C_2$, and $F \cap C_3$, hence it fixes $F$ pointwise by Lemma \ref{genus1auto}. This implies that $g$ is trivial, which shows that there is no numerically trivial involution on $X_a$.
\end{proof}

\begin{remark}
The numerically trivial automorphism $\sigma$ of order $7$ found here leads to one of the exceptions in \cite{DM}.
\end{remark}

\begin{theorem}\label{E7S-main2}
The surfaces $\{X_a\}$ form a $1$-dimensional non-isotrivial family of supersingular Enriques surfaces with the dual graph given in Figure {\rm \ref{E7SDynkin}}.
\end{theorem}
\begin{proof}
It follows from Theorem \ref{E7S-mainAut} that not all $\{X_a\}$ are isomorphic. Therefore, the Theorem of Matsusaka and Mumford \cite[Theorem 1]{MM} implies that
the family $\{X_a\}$ is non-isotrivial.
\end{proof}

\section{Enriques surfaces of type ${\tilde D_8}$}\label{sec7}

In this section we give a construction of Enriques surfaces with the following dual graph of all 
$(-2)$-curves given in Figure \ref{D8Dynkin}.

\begin{figure}[htbp]
\centerline{
\xy
(-20,25)*{};
@={(-10,10),(0,10),(10,10),(20,10),(30,10),(40,10),(50,10),(10,20),(60,10),(50,20)}@@{*{\bullet}};
(-10,10)*{};(60,10)*{}**\dir{-};
(10,10)*{};(10,20)*{}**\dir{-};
(50,10)*{};(50,20)*{}**\dir{-};
\endxy
}
 \caption{}
 \label{D8Dynkin}
\end{figure}

\subsection{Supersingular case}

Let $(x,y)$ be affine coordinates of ${\bf A}^2 \subset {\bf P}^2$.
Consider a rational vector field $D$ defined by
\begin{equation}\label{D8S}
D= D_a= {1\over x^5} \left(x(x^4+x^2 + y^6)\frac{\partial}{\partial x} + (ax^6 + y(x^4 +x^2 + y^6))\frac{\partial}{\partial y}\right)
\end{equation}
where $a\in k^*.$ 
Then $D^2 = 0$, that is, $D$ is 2-closed.
Note that $D$ has poles of order 5 along the line $\ell$ defined by $x=0$, and this line is integral.  
We see that $D$ has a unique isolated singularity $(x,y)=(0,0)$.
First, we blow up at the point $(0,0)$.  Then we see that the exceptional curve is not integral and
the induced vector field has poles of order 2 along the exceptional curve.
Moreover, the induced vector field has a unique isolated singularity at the intersection of the proper transform of $\ell$ and the exceptional curve.
Continue this process until the induced vector field has no isolated singularities.
The final configuration of curves is given in Figure \ref{D8-1}.  Here $F_0$ is the proper transform of
$\ell$ and the suffix $i$ of the exceptional curve $E_i$ corresponds to the order of successive blow-ups.
\begin{figure}[!htb]
 \begin{center}
  \includegraphics[width=140mm]{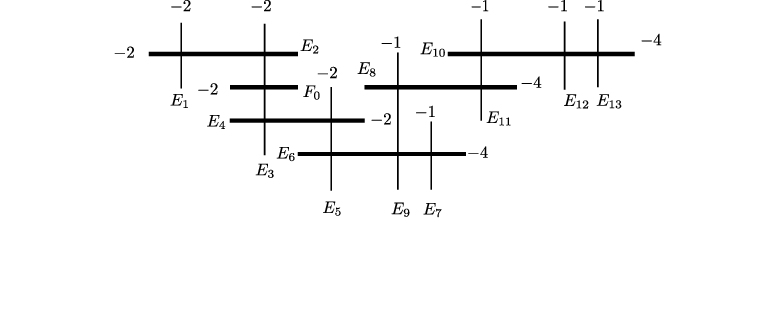}
 \end{center}
 \caption{}
 \label{D8-1}
\end{figure}

We denote by $Y$ the surface obtained by this process.  
Also, abusing notation, we denote by $D$ the induced vector field on $Y$.
By direct calculations, we have the following lemmas.

\begin{lemma}\label{D8S-integral}
{\rm (i)} The integral curves with respect to $D$ in Figure {\rm \ref{D8-1}} are all horizontal curves {\rm (}thick lines{\rm )}.

{\rm (ii)} $(D) =-(5F_0+2E_1+6E_2+8E_3+7E_4+4E_5+3E_6+2E_7+2E_8+4E_9+E_{10}+2E_{11})$.
\end{lemma}

\begin{lemma}\label{D8S-canonical}
{\rm (i)} $(D)^2 = -12$. 

{\rm (ii)} The canonical divisor $K_Y$ of $Y$ is given by
$K_Y= -(3F_0+2E_1+4E_2+6E_3+5E_4+4E_5+3E_6+2E_7+2E_8+4E_9+E_{10}+2E_{11}).$

{\rm (iii)} $K_Y \cdot (D) = -4.$
\end{lemma}

\noindent
Now take the quotient $Y^{D}$ of $Y$ by $D$.  By using the same argument as in the proof of Lemma \ref{e6non-singular}, $D$ is divisorial and $Y^{D}$ is non-singular.
By Proposition \ref{insep}, we have the following configuration of curves in Figure \ref{D8-2}.

\begin{figure}[!htb]
 \begin{center}
  \includegraphics[width=140mm]{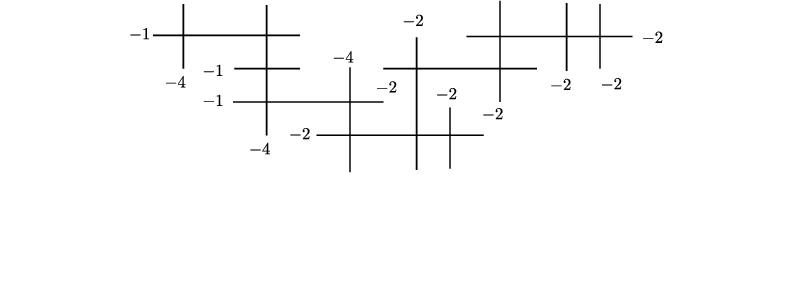}
 \end{center}
 \caption{}
 \label{D8-2}
\end{figure}

\noindent
By contracting the three exceptional curves, we get a new exceptional curve which is the image of the $(-4)$-curve meeting three exceptional curves.
Let $X_a$ be the surface obtained by contracting the new exceptional curve (Recall that the vector field (\ref{DE8S}) contains a parameter $a$). 
The surface $X_a$ contains 10 $(-2)$-curves whose dual graph is given by Figure \ref{D8Dynkin}.
Note that any maximal parabolic subdiagram of this diagram is of type $\tilde{D}_8$ or $\tilde{E}_8$.
Once we know that $X_a$ is an Enriques surface, we can use Tables \ref{ExtremalFibrationsConductrix} and \ref{QuasiellipticFibrationsConductrix} for one of the genus one fibrations with a fiber of type $\II^*$ to deduce that the left-most vertex in Figure \ref{D8Dynkin} is part of the conductrix. Hence, by Lemma \ref{quellorell}, the genus one fibration with a fiber of type $\I_4^*$ is quasi-elliptic of type $(2\I_4^*)$. 
This fibration is
induced from the pencil of lines in ${\bf P}^2$ through $(x,y)=(0,0)$.

\begin{theorem}\label{D8S-main}
The surfaces $\{X_a\}$ form a $1$-dimensional family of supersingular Enriques surfaces with the dual graph given in Figure {\rm \ref{D8Dynkin}}.
\end{theorem}
\begin{proof}
By using Lemmas \ref{D8S-integral} and \ref{D8S-canonical} and the same arguments as in the proofs of Theorems \ref{main} and \ref{main2},
$X$ is an Enriques surface with the dual graph given in Figure \ref{D8Dynkin}.  
By construction, the normalization of the canonical cover has a rational double point of type $D_4$.  
Hence, $X_a$ is supersingular by Lemma \ref{SingNormal}.
\end{proof}

\begin{remark}\label{D8S-coh}
Note that $X_a$ contains exactly three genus one fibrations.
Let $p_1$ be the quasi-elliptic fibration with a double singular fiber $2F_1$ of type $\I_4^*$, and let $p_i$ $(i=2,3)$ be two
genus one fibrations with a singular fiber $F_i$ of type $\II^*$. Note that, by Table \ref{QuasiellipticFibrationsConductrix},
the conductrix of $X_a$ is contained in the singular fiber of
type $\II^*$ of $p_2$ and $p_3$. Hence, these fibrations are elliptic by Lemma \ref{quellorell}.
Also, note that $F_1\cdot F_2=F_1\cdot F_3=F_2\cdot F_3=2$.  If both $F_2$ and $F_3$ are double fibers, then there are no canonical $U$-pairs on this Enriques surface which is a contradiction (Cossec and Dolgachev \cite[Theorem 3.4.1]{CD}).  Hence, one of them, say $F_2$, is double and the other, $F_3$, is simple.
Since there are no automorphisms which change a double fiber and a simple fiber, 
any automorphism of $X_a$ is cohomologically trivial.  
\end{remark}

\begin{theorem}\label{D8S-main2}
The automorphism group ${\rm Aut}(X_{a})$ is the quaternion group $Q_8$ of order $8$ which is cohomologically trivial.
\end{theorem}
\begin{proof}
We consider the vector field (\ref{D8S}), and we set 
$u = x^2$, $v = y^2$, $z = ax^7 + x^5y + x^3y+ xy^7$.
Then, we have $D(u) =0$, $D(v) = 0$, $D(z) = 0$ with the equation
\begin{equation}\label{D8Sequation}
z^2 = a^2u^7 + u^5v + u^3v + uv^7.  
\end{equation}
Therefore, the quotient surface of 
${\bf P}^2$ by $D$ is
birational to the surface defined by the equation (\ref{D8Sequation}),
which is birational to our Enriques surface. To do a change of coordinates,
we define new variables $x, y, t$ by 
$$
   x = 1/u, \quad y = z/u^4,\quad t = v/u
$$
and we replace $a^2$ by $a$ for the sake of simplicity.
Then, the equation becomes 
\begin{equation}\label{D8Sequation2}
y^2 + tx^4 + tx^2 + ax + t^7= 0.
\end{equation}
This equation gives a normal affine chart of a genus one fibration
$$f:X_{a,b}\to {\bf P}^1.$$
Set
$$
A = k[t, x, y]/(y^2 + tx^4 + tx^2 + ax + t^7)
$$ 
and let $\sigma$ be an automorphism of
our Enriques surface. The double fiber of $f$, denoted by $2F_{\infty}$, of type $\I_4^*$ exists over the point defined by $t = \infty$.
Since $\sigma$ preserves the diagram of $(-2)$-curves, $\sigma$ preserves
the curve $C$ of cusps and $2F_{\infty}$.
Thus $\sigma$ has the form in (\ref{automorphism}).

Therefore, 
together with the equation (\ref{D8Sequation2}), we have an identity
$$
\begin{array}{l}
e_1(t,x)^2( tx^4 + tx^2 + ax + t^7) +e_2(t,x)^2 \\
= (c_1 t  + c_2)(d_1(t)x + d_2(t))^4  + (c_1 t  + c_2)(d_1(t)x + d_2(t))^2\\
\quad + a(d_1(t)x + d_2(t)) + (c_1 t  + c_2)^7.
\end{array}
$$
Using Lemma~\ref{trivial} and taking the coefficients of $x$, we have $ae_1(t,x)^2 + ad_1(t) = 0$.
Therefore, $e_1(t,x)$ is a polynomial of $t$, i.e. we can put $e_1(t,x) = e_1(t)$, and
$d_1(t) = e_1(t)^2$. Taking the coefficients of $t$, we have
$e_1(t)^2x^4 +e_1(t)^2x^2 + e_1(t)^2t^6 +  c_1(d_1(t)x + d_2(t))^4  + 
c_1(d_1(t)x + d_2(t))^2 +ad_2(t)_{odd}/t + c_1(c_1 t  + c_2)^6 = 0$.
Here, $d_2(t)_{odd}$ is the odd terms of $d_2(t)$.
Considering the coefficients of $x^4$ of this equation, we have 
$e_1(t)^2 = c_1d_1(t)^4=c_1e_1(t)^8$. Since we have $e_1(t)\not\equiv 0$, we have
$e_1(t)^6 = 1/c_1$. Therefore, $e_1(t)$ is a constant and we set $e_1(t) =e_1\in k$. Then,
$e_1^6 = 1/c_1$. Considering the coefficients of $x^2$, we have 
$e_1^2 = e_1(t)^2 = c_1d_1(t)^2 = c_1e_1^4$. Hence $e_1^2 = 1/c_1$. 
Therefore, we have $c_1 = 1$ and so $e_1 = d_1 = 1$. The equation becomes 
$t^6 + d_2(t)^4 + d_2(t)^2 + ad_2(t)_{odd}/t + (t  + c_2)^6 = 0$.
If the degree of $d_2(t)$ is greater than or equal to 2, then the highest term of 
$d_2(t)^4 $ cannot be killed in the equation. Therefore, we can put $d_2(t) = b_0 + b_1t$ $(b_0, b_1 \in k)$
and we have an identity 
$$
t^6 + (b_0 + b_1t)^4 + (b_0 + b_1t)^2 + ab_1 + (t  + c_2)^6 = 0.
$$
Hence we have $c_2 =b_1^2$, $c_2^2 = b_1$ and $b_0^4 + b_0^2 + ab_1 + c_2^6 = 0$.
Thus we have either $c_2 = 0$, $b_1 = 0$, $b_0 = 0, 1$, or $c_2 =\omega$, $b_1 = \omega^2$ and 
$b_0 = \alpha$ is any root of $z^2 + z + \omega\sqrt{a} + 1 = 0$. Here, $\omega$ is any cube root of unity. There exist 8 solutions. Putting these data into the original equation, we have 
$e_2(t,x) = \sqrt{a}$ or $\omega^2x^2 + \omega^2x + \omega^2 t^3 +\sqrt{a\alpha} + \sqrt{a}$.
Thus we have
$$\sigma(t) = t + \omega, \quad \sigma(x) = x + \alpha + \omega^2t\quad \sigma(y) = y  + \omega^2x^2 + \omega^2x + \omega^2 t^3 +\sqrt{a\alpha} + \sqrt{a},$$
and we conclude
${\rm Aut}(X) \cong {Q}_8$.  The cohomological triviality follows from Remark \ref{D8S-coh}.
\end{proof}

\begin{remark}
The group $Q_8$, which appears here as the first example of a non-commutative group of cohomologically trivial automorphisms, leads to one of the exceptions in \cite{DM}.
\end{remark}

\subsection{Classical case}

Let $Q = {\bf P}^1\times {\bf P}^1$ be a non-singular quadric and let $((u_0, u_1), (v_0, v_1))$ be 
homogeneous coordinates of $Q$. Let $x=u_0/u_1,\ x'=u_1/u_0,\ y =v_0/v_1, y'=v_1/v_0$.
Consider a rational vector field $D$ defined by
\begin{equation}\label{D8Cderi}
D= {1\over xy^2} \left(ax^2y^2\frac{\partial}{\partial x} + (x^4y^4 + by^4 +x^2y^2 + x^2)\frac{\partial}{\partial y}\right)
\end{equation}
where $a, b \in k, a, b \not=0$.
Then $D^2 = aD$, that is, $D$ is 2-closed.
Note that $D$ has a pole of order 1 along the divisor defined by $x=0$, a pole of order 3 along the divisor defined by $x=\infty$ and a pole of order 2 along the divisor defined by $y=0$.
Moreover $D$ has isolated singularities at $(x,y)= (0,0), (\infty, 0)$.
As in the case of supersingular Enriques surfaces of type $\tilde{E}_8$, we successively blow up the points of isolated 
singularities 
of $D$ and finally arrive at a vector field without 
isolated singularities which we will denote by $D$ again.
The configuration of curves is given in Figure \ref{D8C}.

\begin{figure}[!htb]
 \begin{center}
  \includegraphics[width=140mm]{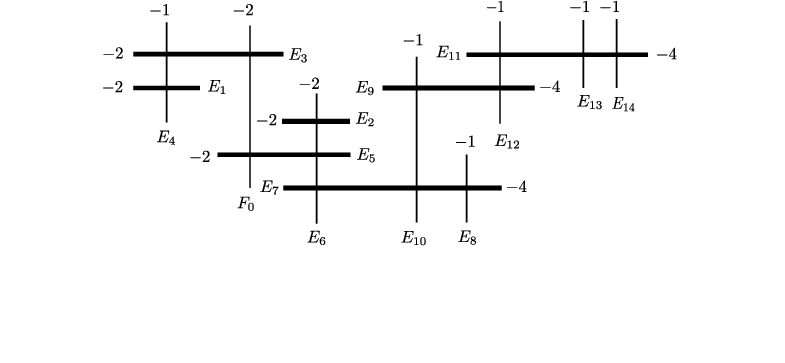}
 \end{center}
 \caption{}
 \label{D8C}
\end{figure}
\noindent
Here $F_0$, $E_1$, or $E_2$ is the proper transform of the curve defined by $y=0$, $x=0$, or $x=\infty$, respectively.  

We denote by $Y$ the surface obtained by the successive blow-ups.
A direct calculation shows the following two lemmas.

\begin{lemma}\label{D8C-integral}
{\rm (i)} The integral curves with respect to $D$ in Figure {\rm \ref{D8C}} are all horizontal curves {\rm (}thick lines{\rm )}.

{\rm (ii)} $(D) =-(2F_0+E_1+3E_2+2E_3+4E_5+4E_6+3E_7+2E_8+2E_9+4E_{10}+E_{11}+2E_{12})$.
\end{lemma}

\begin{lemma}\label{D8C-canonical}
{\rm (i)} $(D)^2 = -12$. 

{\rm (ii)} The canonical divisor $K_Y$ of $Y$ is given by
$K_Y = -(2F_0+2E_2+E_3+3E_5+4E_6+3E_7+2E_8+2E_9+4E_{10}+E_{11}+2E_{12}).$

{\rm (iii)} $K_Y\cdot (D) = -4.$
\end{lemma}

\noindent
Now take the quotient $Y^{D}$ of $Y$ by $D$.  By using the same argument as in the proof of Lemma \ref{e6non-singular}, $D$ is divisorial and hence $Y^{D}$ is non-singular.
By Proposition \ref{insep}, we have the following configuration of curves in Figure \ref{D8C2}.
\begin{figure}[!htb]
 \begin{center}
  \includegraphics[width=140mm]{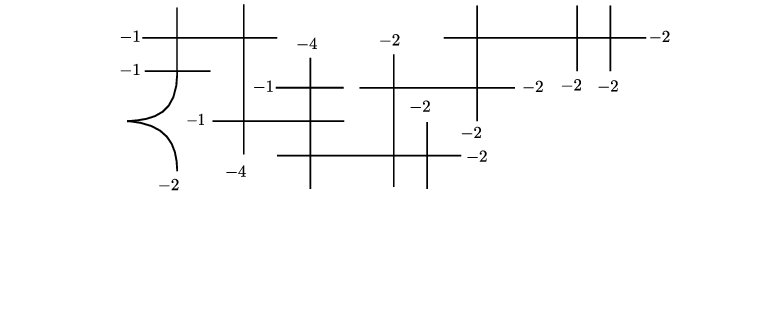}
 \end{center}
 \caption{}
 \label{D8C2}
\end{figure}

Let $X_{a,b}$ be the surface obtained by contracting the four exceptional curves in Figure \ref{D8C2}
(Recall that the vector field $D$ contains two parameters $a, b$ (see (\ref{D8Cderi}))).
On $X_{a,b}$, there exist 10 $(-2)$-curves whose dual graph is given by Figure \ref{D8Dynkin}.
Recall that any maximal parabolic subdiagram of this diagram is of type $\tilde{D}_8$ or $\tilde{E}_8$.
As in the previous case, once we know that $X_{a,b}$ is an Enriques surface, we can deduce that there exists a quasi-elliptic fibration of type $(\I_4^*)$ 
induced by the first projection from $Q$ to ${\bf P}^1$.

\begin{theorem}\label{D8C-main}
The surfaces $\{X_{a,b}\}$ form a $2$-dimensional family of classical Enriques surfaces with the dual graph given in Figure {\rm \ref{D8Dynkin}}.
\end{theorem}
\begin{proof}
By using Lemmas \ref{D8C-integral} and \ref{D8C-canonical} and the same arguments as in the proofs of Theorems \ref{main} and \ref{main2},
$X_{a,b}$ is an Enriques surface with the dual graph given in Figure \ref{D8Dynkin}.
Since $X_{a,b}$ has a genus one fibration with two double fibers (see Figure \ref{D8C2}), $X_{a,b}$ is classical (Proposition \ref{multi-fiber}).
\end{proof}

\begin{remark}\label{D8C-symm}
There are two genus one fibrations with a singular fiber of type $\II^*$.
As we explained in Remark \ref{D8S-coh}, one of these fibers is double, the other is simple and both fibrations are elliptic. If they were of type $(\II^*)$, then their $j$-invariant would be zero (Lang \cite{L2}) and hence all non-singular fibers are supersingular elliptic curves by Lemma \ref{smoothdoublefiber}.  This contradicts the fact that a double fiber of a genus one fibration on a classical Enriques surface is an ordinary elliptic curve or of additive type (Proposition \ref{genus1}).  Thus, the two elliptic fibrations are of type $(\II^*, \I_1)$ by Lang \cite{L2}. 
\end{remark}

\begin{theorem}\label{D8C-mainAut}
The automorphism group ${\rm Aut}(X_{a,b})$ is ${\bf Z}/2{\bf Z}$ which is numerically trivial.
\end{theorem}
\begin{proof}
It follows from Remark \ref{D8C-symm} that $\Aut(X_{a,b}) = \Aut_{nt}(X_{a,b})$.
We consider the vector field (10.2), and we set 
$u = x^2$, $v = y^2$, $z = x^5y^4 + bxy^4 + x^3y^2+ x^3 + ax^2y^3$.
Then, we have $D(u) =0$, $D(v) = 0$, $D(z) = 0$ with the equation
\begin{equation}\label{D8Cequation}
z^2 = u^5v^4 + b^2uv^4 + u^3v^2+ u^3 + a^2u^2v^3 \quad (a, b \in k^*).
\end{equation}  
Therefore, the quotient surface of  
${\bf P}^1\times {\bf P}^1$ by $D$ is
birational to the surface defined by the equation (\ref{D8Cequation}),
which is birational to our Enriques surface. To do a change of coordinates,
we define new variables $x, y, t$ by 
$$
   x = \sqrt[4]{b}/uv, \quad y = \sqrt[4]{b^3}z/u^4v^2,\quad t = \sqrt{b}/u.
$$
and we replace $\frac{1}{\sqrt{b}}$ and $\frac{a^2}{\sqrt[4]{b^5}}$ 
by $a$ and $b$, respectively, for the sake of simplicity.
Then, the equation becomes 
\begin{equation}\label{D8Cequation2}
y^2 + tx^4 + at^3x^2 + bt^3x +t^3 + t^7= 0. 
\end{equation}
This equation gives a normal affine chart of a genus one fibration
$$f:X_{a,b}\to {\bf P}^1.$$
Set
$$
A = k[t, x, y]/(y^2 + tx^4 + at^3x^2 + bt^3x +t^3 + t^7)
$$ 
and let $\sigma$ be an automorphism of
our Enriques surface. The double fiber of $f$, denoted by $2F_{\infty}$, of type $\I_4^*$ exists over the point defined by $t = \infty$.
Since $\sigma$ preserves the dual graph of $(-2)$-curves, $\sigma$ preserves
the curve $C$ of cusps and $2F_{\infty}$.
Therefore,  $\sigma$ has the form in (\ref{automorphism}).
Moreover, this quasi-elliptic surface has a singular fiber over the point defined by $t = 0$,
$\sigma$ preserves also the singular fiber. Therefore, we know $c_2 = 0$ and  we have
$\sigma^{*}(t) = c_1t$.

Therefore, 
together with the equation (\ref{D8Cequation2}), we have an identity
$$
\begin{array}{l}
e_1(t,x)^2(tx^4 + at^3x^2 + bt^3x +t^3 + t^7) +e_2(t,x)^2 \\
= c_1 t (d_1(t)x + d_2(t))^4  + a(c_1 t)^3(d_1(t)x + d_2(t))^2\\
\quad + b(c_1 t)^3(d_1(t)x + d_2(t)) + (c_1 t)^3 + (c_1 t)^7.
\end{array}
$$
Differentiate both sides by $x$, and we have 
$be_1(t,x)^2t^3 + bc_1^3d_1(t)t^3 = 0$, that is, $e_1(t,x)^2 = c_1^3d_1(t)$.
Therefore, $e_1(t,x)$ is a polynomial of $t$, i.e. we can put $e_1(t,x) = e_1(t)$, and
$d_1(t) = c_1^{-3}e_1(t)^2$. Using Lemma~\ref{trivial} and  taking the coefficients of $t$, 
we have
$e_1(t,x)^2(x^4 + at^2x^2 +t^2 + t^6) 
+ c_1(c_1^{-3}e_1(t)^2x + d_2(t))^4  + ac_1^3 t^2(c_1^{-3}e_1(t)^2x + d_2(t))^2
+ bc_1^3d_2(t)_{even}t^2 + c_1^3t^2 + c_1^7t^6 = 0$.
Here, $d_2(t)_{even}$ is the even terms of $d_2(t)$.
Considering the coefficients of $x^4$ of this equation, we have 
$e_1(t)^2 =c_1^{-11}e_1(t)^8$. Since we have $e_1(t)\not\equiv 0$, we have
$e_1(t)^6 = c_1^{11}$. Therefore, $e_1(t)$ is a constant and we set $e_1(t) =e_1\in k$. 
Then, we have $e_1^6 = c_1^{11}$. 
Considering the coefficients of $x^2$ of this equation, we have 
$ae_1^2t^2 = ac_1^{-3}e_1^4t^2$, i.e. $e_1^2 =  c_1^3$. Hence we have
$c_1^9 = c_1^{11}$. Since $c_1 \neq 0$, we have $c_1 = 1$. Therefore, we have $e_1 = 1$ 
and $d_1(t) = 1$. Then, the equation becomes 
$d_2(t)^4  + at^2d_2(t)^2 + bd_2(t)_{even}t^2 = 0$.
If the degree of $d_2(t)$ is greater than or equal to 2, then the highest term of 
$d_2(t)^4 $ cannot be killed in the equation. 
Therefore, we can put $d_2(t) = b_0 + b_1t$ $(b_0, b_1 \in k)$
and we have an identity 
$(b_0 + b_1t)^4 + a(b_0 + b_1t)^2t^2 + bb_0t^2  = 0$.
Hence we have $b_1^4 = ab_1^2$, $ab_0^2 = bb_0$ and $b_0^4=  0$.
Thus we have $b_0 = 0$, and $b_1 = \sqrt{a}$ or $0$.  Going to the original equality,
we have $e_2(t,x)^2 =  bt^3\sqrt{a}t$, i.e. $e_2(t,x) = \sqrt[4]{a}\sqrt{b}t^2$.
Therefore, we conclude that $\sigma$ is given by either
$t \mapsto t, ~x \mapsto x + \sqrt{a}t,~ y \mapsto y + \sqrt[4]{a}\sqrt{b}t^2$
or the identity. Hence, we have ${\rm Aut}(X) \cong {\bf Z}/2{\bf Z}$.
\end{proof}

\section{Enriques surfaces of type $\tilde{D_4}+\tilde{D_4}$}\label{sec5}

In this section we give a construction of Enriques surfaces with the following dual graph of all 
$(-2)$-curves given in Figure \ref{2D4Dynkin} .

\begin{figure}[htbp]
\centerline{
\xy
(-10,25)*{};
@={(10,0),(50,0),(0,10),(10,10),(20,10),(30,10),(40,10),(50,10),(10,20),(60,10),(50,20)}@@{*{\bullet}};
(0,10)*{};(60,10)*{}**\dir{-};
(10,0)*{};(10,20)*{}**\dir{-};
(50,0)*{};(50,20)*{}**\dir{-};
\endxy
}
 \caption{}
 \label{2D4Dynkin}
\end{figure}

Let $Q = {\bf P}^1\times {\bf P}^1$ be a non-singular quadric and let $((u_0, u_1), (v_0, v_1))$ be 
homogeneous coordinates of $Q$. Let $x=u_0/u_1,\ x'=u_1/u_0,\ y =v_0/v_1, y'=v_1/v_0$.
Consider a rational vector field $D$ defined by 
\begin{equation}
D= {1\over x^2y^2} \left(bx^3y^2\frac{\partial}{\partial x} + (ax^2y^2 + x^2 + x^4y^4 + y^4 + bx^2y^3)\frac{\partial}{\partial y}\right)
\end{equation}
where $a, b \in k, b\not=0$.  Note that $D^2=bD$, that is, $D$ is 2-closed.
Denote by $E_1, E_2$ and $F_0$ the curves defined by $x=0$, $x'=0$ and $y=0$, respectively.
The vector field $D$ has poles of order 2 along $E_1, E_2, E_3$, and
has isolated singularities $(x,y)=(0,0)$ and $(x',y)=(0,0)$.
The curves $E_1, E_2$ are integral.
Now blow up at the two points $(x,y)=(0,0)$ and $(x',y)=(0,0)$.  The two exceptional curves are integral with respect to the induced vector field.  The induced vector field has poles of order 3 along two exceptional curves and has isolated singularities at the intersections of the exceptional curves and the proper transforms of $E_1$ and $E_2$.  Then, blow up at the isolated singularities of the induced vector field and continue this process until the induced vector field has no isolated singularities.
We denote by $Y$ the surface obtained by this process and use the same symbols $E_1, E_2, F_0$ for the curves and their proper transforms. 
Also we denote by the same symbol $D$ the induced vector field on $Y$.
The final configuration of curves is given in Figure \ref{D4-2}.

\begin{figure}[!htb]
 \begin{center}
  \includegraphics[width=140mm]{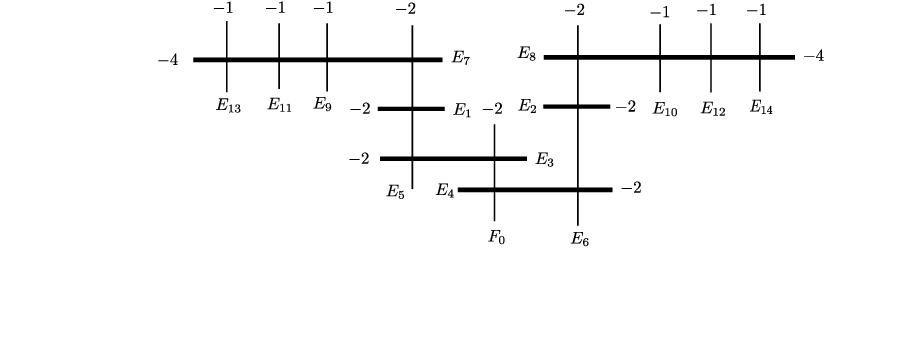}
 \end{center}
 \caption{}
 \label{D4-2}
\end{figure}


\begin{lemma}\label{2D4-integral}
{\rm (i)} The integral curves with respect to $D$ in Figure {\rm \ref{D4-2}} are all horizontal curves {\rm (}thick lines{\rm )}.

{\rm (ii)} $(D) =-(2F_0+2E_1+2E_2+3E_3+3E_4+2E_5+2E_6+E_7+E_8)$.
\end{lemma}

\begin{lemma}\label{2D4-canonical}
{\rm (i)} $(D)^2 = -12$. 

{\rm (ii)} The canonical divisor $K_Y$ of $Y$ is given by
$K_Y = -(2F_0+E_1+E_2+2E_3+2E_4+2E_5+2E_6+E_7+E_8).$

{\rm (iii)} $K_Y\cdot (D) = -4.$
\end{lemma}

\noindent
Now take the quotient $Y^D$ of $Y$ by $D$.
By using the same argument as in the proof of Lemma \ref{e6non-singular}, $D$ is divisorial and $Y^D$ is non-singular.
By Proposition \ref{insep}, we have the following Figure \ref{D4-3}.

\begin{figure}[!htb]
 \begin{center}
  \includegraphics[width=150mm]{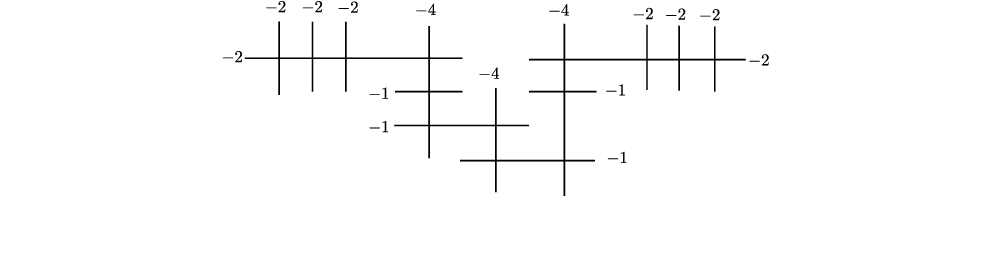}
 \end{center}
 \caption{}
 \label{D4-3}
\end{figure}

Let $X_{a,b}$ be the surface obtained by contracting the four exceptional curves which 
contains 11 $(-2)$-curves whose dual graph is given by Figure \ref{2D4Dynkin}.
Note that any maximal parabolic subdiagram of this diagram is of type $\tilde{D}_8$ or $\tilde{D}_4\oplus \tilde{D}_4$.
The surface $X_{a,b}$ contains a quasi-elliptic fibration $p_1$ with singular fibers of type
$(2\I_0^*, 2\I_0^*)$ induced from the first projection from $Q$ to ${\bf P}^1$ and nine genus one fibrations with a singular fiber of type $(\I_4^*)$.
These nine genus one fibrations are elliptic by comparing to the conductrix given in Ekedahl and Shepherd-Barron \cite[Theorem 2.2, Theorem 3.1]{ES} (see Tables \ref{ExtremalFibrationsConductrix} and \ref{QuasiellipticFibrationsConductrix} in the Section \ref{possibledualgraph}).
%

\begin{theorem}\label{D4C-main}
The surfaces $\{X_{a,b}\}$ form a $2$-dimensional family of classical Enriques surfaces with the dual graph
given in Figure {\rm \ref{2D4Dynkin}}.  It contains an at least $1$-dimensional non-isotrivial family.
\end{theorem}
\begin{proof}
By using Lemmas \ref{2D4-integral} and \ref{2D4-canonical} and the same arguments as in the proofs of Theorems \ref{main} and \ref{main2}, 
$X_{a,b}$ is an Enriques surface with the dual graph given in Figure \ref{2D4Dynkin}.

By equation (\ref{equation0}) in Subsection \ref{ExampleD4D4}, the surface $X_{a,b}$ is the quasi-elliptic surface given affinely  by the equation
$$
u^2 + Sv^4 + a^2S^3v^2 + b^2S^4v + S^3 + S^7= 0
$$
By Queen \cite[Theorem 2]{Q1}, its Jacobian is the quasi-elliptic surface given by
$$
u^2 + Sv^4+a^2S^3v^2 + b^2S^4v = 0
$$
Now we change coordinates
$$
   Y = {u\over bS^2v^2},\quad X = {1\over v} + {a^2\over b^2S},\quad T = {1\over S}
$$
which yields
$$
Y^2 = X^3 + (a^4/b^4)T^2X + (1/b^2)T^3
$$
Since these Jacobian quasi-elliptic surfaces form a $1$-dimensional non-isotrivial family by Ito \cite{Ito}, the family $\{X_{a,b}\}$ contains an at least $1$-dimensional non-isotrivial family.
\end{proof}

\begin{theorem}\label{D4C-main2}
The automorphism group ${\rm Aut}(X_{a,b})$ is isomorphic to
$({\bf Z}/2{\bf Z})^3$.  Moreover, ${\rm Aut}_{nt}(X_{a,b}) \cong ({\bf Z}/2{\bf Z})^2$.
\end{theorem}
\begin{proof}
We use the equation of $X_{a,b}$ given in Theorem \ref{D4C-main}.
We set 
$x = v, ~y = u,~ t = S$
and we replace $a^2$ (resp. $b^2$) by $a$ (resp. $b$) for the sake of simplicity.
Then, the equation becomes 
\begin{equation}\label{2D4equation}
y^2 + tx^4 + at^3x^2 + bt^4x + t^3 + t^7= 0. 
\end{equation}
This equation gives a normal affine
surface. Set
$$
A = k[t, x, y]/(y^2 + tx^4 + at^3x^2 + bt^4x + t^3 + t^7).
$$
Our quasi-elliptic surface $\varphi : X_{a,b} \longrightarrow {\bf P}^1$ has
two double fibers of type ${\rm I}_0^*$ over the points defined by $t = 0$ (resp. $t = \infty$). 
First, we consider an automorphism $\tau$ of $X_{a,b}$ defined by
$$
\tau: ~t \mapsto 1/t,\quad x \mapsto x/t^2, \quad y \mapsto y/t^5.
$$
This automorphism is of order 2 and exchanges the two double fibers.
Let $\sigma$ be an automorphism of
our Enriques surface. Then $\sigma$ either preserves the double fibers or exchanges them.
If $\sigma$ exchanges the double fibers, then we consider $\tau \circ \sigma$.
This preserves the double fibers. Therefore, we assume that $\sigma$ preserves
the double fibers.
Since $\sigma$ preserves the diagram of $(-2)$-curves, $\sigma$ preserves
the curve $C$ of cusps and the double fiber $2F_{\infty}$ over $t=\infty$.
Therefore,  $\sigma$ has the form in (\ref{automorphism}). Moreover,
by our assumption, $\sigma$ preserves the double fiber over the point defined by $t = 0$.
Therefore, we may assume $\sigma^* (t) = c_1t$.
 Using these data, 
we have an identity
$$
\begin{array}{l}
e_1(t,x)^2(tx^4 + at^3x^2 + bt^4x + t^3 + t^7) +e_2(t,x)^2 \\
= c_1 t (d_1(t)x + d_2(t))^4 + a(c_1 t)^3(d_1(t)x + d_2(t))^2\\
\quad + b(c_1 t)^4(d_1(t)x + d_2(t)) + (c_1 t)^3 + (c_1 t)^7.
\end{array}
$$
Using Lemma~\ref{trivial} and taking the coefficients of $x$, we have 
$be_1(t,x)^2t^4 + bc_1^4t^4d_1(t) = 0$. Therefore, we have 
$e_1(t,x)^2 + c_1^4d_1(t) = 0$
and $e_1(t,x)$ is a polynomial of $t$, i.e. we can put $e_1(t,x) = e_1(t)$, and
$d_1(t) = e_1(t)^2/c_1^4$. Taking the coefficients of $t$, we have
$$
\begin{array}{l}
e_1(t)^2x^4 +ae_1(t)^2t^2x^2 +  e_1(t)^2t^2 +  e_1(t)^2t^6 
+  c_1(d_1(t)x + d_2(t))^4  \\
+ ac_1^3t^2(d_1(t)x + d_2(t))^2 +bc_1^4t^4d_2(t)_{odd}/t + c_1^3t^2 +  c_1^7t^6= 0.
\end{array}
$$
Here, $d_2(t)_{odd}$ collects the odd terms of $d_2(t)$.
Considering the coefficients of $x^4$ of this equation, we have 
$e_1(t)^2 = c_1d_1(t)^4=e_1(t)^8/c_1^{15}$. Since we have $e_1(t)\not\equiv 0$, we have
$e_1(t)^6 = c_1^{15}$. Therefore, $e_1(t)$ is a constant and we set $e_1(t) =e_1\in k$. Then,
$e_1^6 = c_1^{15}$. Considering the coefficients of $x^2$, we have 
$at^2e_1^2 = ac_1^3t^2 d_1(t)^2  =ac_1^3t^2 (e_1^2/c_1^4)^2= at^2e_1^4/c_1^5$. Therefore, $e_1^2 = c_1^5$ and $d_1(t) = c_1$.  The equation becomes 
$$
e_1^2t^2 + e_1^2t^6 + c_1d_2(t)^4 + ac_1^3t^2d_2(t)^2 + bc_1^4t^4d_2(t)_{odd}/t 
+c_1^3t^2 +c_1^7t^6 = 0.
$$
If $\deg\ d_2(t) \geq 2$, then we cannot kill the highest term of $c_1d_2(t)^4$ in the equation.

Therefore, we can put $d_2(t) = b_0 + b_1t$, and we have equations
$$
e_1^2 = c_1^7,~c_1b_1^4  + ac_1^3b_1^2 +bc_1^4b_1 = 0 ,
e_1^2 + ac_1^3b_0^2 + c_1^3 = 0 , c_1b_0^4 = 0.
$$
Solving these equations with $e_1^2 = c_1^5$, we have $b_0 = 0$, $c_1 = e_1 = d_1 = 1$, and 
$b_1$ is either 0 or a root of the equation $z^3 + az  + b = 0$. Putting this data into the
original equation, we have $e_2(t, x) = 0$. Hence, we have 4 automorphisms, 
which are the identity and
$$\sigma_{\alpha}(t) = t, \quad \sigma_{\alpha}(x) = x + \alpha t,\quad \sigma_{\alpha}(y) = y,$$
where $\alpha$ is a root of the equation $z^3 + az + b= 0$.
By direct calculation we can prove that 
the involution $\tau$ and these automorphisms commute with each other. We now conclude ${\rm Aut}(X) \cong ({\bf Z}/2{\bf Z})^3$.

Obviously $\tau$ is not numerically trivial.  We show that any involution $\sigma$ preserving each double fiber of type $\I_0^*$ is numerically trivial.  Let $F$ be a double fiber of type $\I_0^*$ and let $E$ be the component with multiplicity 2 of $F$.
Then $\sigma$ preserves $E$ and a simple component $C$ of $F$ meeting with the special 2-section of the fibration, and hence it preserves one more simple component
$C'$ of $F$.  This implies that $\sigma$ fixes two points on $E$ which are intersection points of $E$ with
$C$ and $C'$.  Therefore, $\sigma$ fixes $E$ pointwise and hence $\sigma$ preserves all components of $F$.
Thus $\sigma$ is numerically trivial.
\end{proof}

\begin{remark}
The group ${\rm Aut}_{nt}(X_{a,b}) \cong ({\bf Z}/2{\bf Z})^2$ is the first example of a non-cyclic group of numerically trivial automorphisms of a classical Enriques surface.
\end{remark}

\section{Appendix}\label{appendix}

\subsection{Genus one fibrations}\label{AppendixGenus1fibration}

We summarize genus one fibrations on each of the Enriques surfaces in Theorems \ref{mainSupersingular} \rm{(B)} and \ref{mainClassical} \rm{(B)}.
The list holds not only for the above examples but also for any Enriques surface with the same dual graph of $(-2)$-curves.
We indicate that it is either elliptic or quasi-elliptic 
after the type of singular fibers.    

\begin{itemize}
\item Type $\tilde{E}_8$: \ $(2\II^*)$ (quasi-elliptic);
\item Type $\tilde{E}_7 + \tilde{A}_1^{(1)}$
 supersingular : $(2\III^*, \III)$ (quasi-elliptic),\ \ $(\II^*)$ (quasi-elliptic);\\
 classical: $(2\III^*, \III)$ (quasi-elliptic),\ \ $(\II^*)$ (quasi-elliptic);
\item Type $\tilde{E}_7 + \tilde{A}_1^{(2)}$ 
: $(2\III^*, 2\III)$ (quasi-elliptic),\ \ $(\II^*)$ (quasi-elliptic);
\item Type $\tilde{E}_6+\tilde{A}_2$, \\ 
supersingular: $(2\IV^*, \IV)$ (elliptic),\ \ $(\III^*, 2\III)$ (quasi-elliptic);\\
classical : $(2\IV^*, \I_3, \I_1)$ (elliptic),\ \ $(\III^*, 2\III)$ (quasi-elliptic);
\item Type $\tilde{D}_8$: \\
 supersingular: $(2\I_4^*)$ (quasi-elliptic),\ \ $(2\II^*)$ (elliptic),\ \ $(\II^*)$ (elliptic); \\
 classical: $(2\I_4^*)$ (quasi-elliptic),\ \ $(2\II^*, \I_1)$ (elliptic),\ \ $(\II^*,\I_1)$ (elliptic);
\item Type $\tilde{D}_4 + \tilde{D}_4$: \ $(2\I_0^*, 2\I_0^*)$ (quasi-elliptic),\ \ $(\I_4^*)$ (elliptic),\ \ $(2\I_4^*)$ (elliptic); 
\item Type $\VII$: \ $(\I_9, \I_1, \I_1, \I_1)$ (elliptic),\quad $(\I_8, 2\III)$ (elliptic),\quad $(\I_5, \I_5, \I_1, \I_1)$ (elliptic),\quad
$(\I_6, 2\IV, \I_2)$ (elliptic);
\item Type $\VIII$: \ $(2\I_1^*, \I_4)$ (elliptic),\ \ $(\I_2^*, 2\III, 2\III)$ (quasi-elliptic),\ \ $(\IV^*, \I_3, \I_1)$ (elliptic). 
\end{itemize}

\subsection{List of examples of equations and automorphisms}\label{LIST}

In the following cases, we calculated equations of Enriques surfaces to determine their automorphism
groups: type $\tilde{E}_6 + \tilde{A}_2$, supersingular; type $\tilde{E}_8$, supersingular and classical;
type $\tilde{D}_8$, supersingular and classical; type $\tilde{D}_4+\tilde{D}_4$. 
In this appendix, we give equations for the remaining examples of Enriques surfaces in Theorems \ref{mainSupersingular} (B), \ref{mainClassical}, (B).  

\medskip
\noindent
(1) Enriques surfaces of type $\tilde{E}_6 + \tilde{A}_2$, Classical case:
$$y^2 + c^2txy + \beta c^3t^2y = tx^4 + c^2t^3x^2 + (c^3t^4 + c^5\alpha t^3)x + t^7 + t^3,$$
\noindent
where
$c = \frac{1}{a + \sqrt[4]{a^3}}\quad (a \in k \setminus \{0,1\})$, 
$\alpha$ is a root of $z^8 + z^6 + z^5 + a^2 z^4 + a^4z^3 + a^8z^2 + a^{16} = 0$, and 
$\beta = \frac{\alpha^2 + a^4}{\alpha}$.


\medskip
\noindent
(2) Enriques surfaces of type ${\rm VII}$: 
$$y^2 + t(t+1)(t+a^2)(t+b^2)xy + \{(ab+1)t+ab\}(t+1)(t+a^2)(t+b^2)y = $$
$$tx^4 +\{(ab+1)t + ab\}(t+1)(t+a^2)(t+b^2)x^3 + \{t^2 + (t+1)(t+a^2)(t+b^2)\}(t+1)(t+a^2)(t+b^2)x^2$$
$$+ \{(ab+1)t+ab\}t(t+1)(t+a^2)(t+b^2)x + t^3 +t^3(t+1)(t+a^2)(t+b^2)$$
$$+ t(t+1)^2(t+a^2)^2(t+b^2)^2 + t(t+1)^3(t+a^2)^3(t+b^2)^3,$$
where $a, b\in k, \ a+b=ab, \ a^3\not= 1$.

\medskip
\noindent
(3) Enriques surfaces of type ${\rm VIII}$:
$$y^2 = tx^4 + at^2x^3 + at^3(t + 1)^2x + t^7 + t^3 \quad (a \in k^*).$$


\medskip
\noindent
(4) Enriques surfaces of type $\tilde{E}_7 + \tilde{A}_1^{(1)}$, Classical case: 
$$y^2 + at^2y = tx^4 + bt^3x + t^7 + t^3, \quad \quad (a, b\in k^*).$$


\medskip
\noindent
(5) Enriques surfaces of type $\tilde{E}_7 + \tilde{A}_1^{(2)}$:
$$y^2 +  at^2y = tx^4 + t^7 + t^3, \quad (a\in k^*).$$


\end{document}